\documentclass{amsart}
\usepackage[margin=1in]{geometry}
\usepackage[utf8]{inputenc}
\usepackage{amsmath, amsthm, amssymb, amsfonts, mathtools}
\usepackage[usenames,dvipsnames,svgnames,table]{xcolor}
\usepackage{enumerate}
\usepackage{microtype}
\usepackage{cite}
\usepackage[colorlinks,linkcolor = blue,citecolor = ForestGreen]{hyperref}
\usepackage{esint}
\usepackage{nicefrac}
	\newcommand{\nfrac}{\nicefrac}
	\newcommand{\sfrac}{\nfrac}
\usepackage{stmaryrd} 
\usepackage{caption}
\usepackage{cancel}

\usepackage[bb=ncmbbr]{mathalpha}

\newcommand{\lowvee}{ \raisebox{-0.2ex}{\scalebox{0.5}{$\vee$}}}

\usepackage{mathtools}
\mathtoolsset{showonlyrefs}

\newcounter{num} \numberwithin{num}{section}
\newtheorem{theorem}[num]{Theorem}
\newtheorem*{restatedthm}{\restatedthmname}
\newcommand{\restatedthmname}{}
\newenvironment{recallthm}[1]
  {\renewcommand{\restatedthmname}{Theorem \ref*{#1}}\begin{restatedthm}}
  {\end{restatedthm}}

\newtheorem{proposition}[num]{Proposition}
\newtheorem{lemma}[num]{Lemma}
\newtheorem{corollary}[num]{Corollary}
\theoremstyle{definition}

\theoremstyle{remark}

\numberwithin{equation}{section}

\usepackage{zref-clever}
	\newcommand{\Cref}[1]{\zcref[S]{#1}}

\AddToHook{env/theorem/begin}
    {\zcsetup{countertype={num=theorem}}}
\AddToHook{env/proposition/begin}
    {\zcsetup{countertype={num=proposition}}}
\AddToHook{env/lemma/begin}
	{\zcsetup{countertype={num=lemma}}}
\AddToHook{env/definition/begin}
	{\zcsetup{countertype={num=definition}}}
\AddToHook{env/remark/begin}
	{\zcsetup{countertype={num=remark}}}
\AddToHook{env/corollary/begin}
	{\zcsetup{countertype={num=corollary}}}
\AddToHook{env/heuristic/begin}
	{\zcsetup{countertype={num=heuristic}}}
    \zcRefTypeSetup{heuristic}{
		Name-sg = Heuristic Result,
		name-sg = heuristic result,
		Name-pl = Heuristic Results,
		name-pl = heuristic results
	}
\AddToHook{env/table/begin}
	{\zcsetup{countertype={num=table}}}

\DeclareMathOperator{\im}{Image}

\newcommand{\R}{\mathbb R}

\newcommand{\eps}{\varepsilon}
\newcommand{\e}{\eps}

\newcommand{\1}{\mathbb{1}}

\newcommand{\cK}{\mathcal{K}}
\newcommand{\cL}{\mathcal{L}}

\renewcommand{\phi}{\varphi}

\newcommand{\BNORM}[1]{{\left\vert\kern-0.25ex\left\vert\kern-0.25ex\left\vert #1 
    \right\vert\kern-0.25ex\right\vert\kern-0.25ex\right\vert}_{\mathcal{B}_{T,\alpha}^{2}}}

\newcommand{\be}{\begin{equation}}
\newcommand{\ee}{\end{equation}}

\newcommand{\dz}{\partial_z}

\newcommand\smallO{
  \mathchoice
    {{\scriptstyle\mathcal{O}}}
    {{\scriptstyle\mathcal{O}}}
    {{\scriptscriptstyle\mathcal{O}}}
    {\scalebox{.7}{$\scriptscriptstyle\mathcal{O}$}}
  }

\begin{document}

\title{Front Location for Go or Grow Models of Aerotaxis}

\author{Mete Demircigil}
\address[Mete Demircigil]{Department of Mathematics, University of Arizona}
\email{mete.demircigil@math.cnrs.fr}

\author{Christopher Henderson}
\address[Christopher Henderson]{Department of Mathematics, University of Maryland}
\email{ckhend@umd.edu}

\begin{abstract}
We investigate the pushed-to-pulled transition for a minimal model for invasive fronts 
influence by ``aerotaxis,'' that is, when organisms follow oxygen gradients. 
We consider two singular reaction-advection-diffusion models for this.  The version of primary interest arises as a hydrodynamic limit of a system of branching, rank-based interacting Brownian particles and features a nonlinear, nonlocal advection. The second version is introduced here as a local counterpart. We establish well-posedness for both models, with the local case requiring a novel use of the ``shape defect function.'' We further characterize the front location up to $O(1)$ precision in all cases, including the delicate boundary ``pushmi-pullyu'' case.  
\end{abstract}

\maketitle


\section{Introduction}\label{s.intro}

\subsection{The model and its context}

\subsubsection{The model}

In this work we study one-dimensional reaction-diffusion-advection equations of ``\textit{go} or \textit{grow}''-type describing the invasion of a proliferating cell population along a narrow capillary tube. The spatial coordinate $x\in\R$ represents position along the tube. The cell population is initially placed on the left and invades the region towards the right.
The central feature of these ``\textit{go} or \textit{grow}'' models is a sharp dichotomy between two regimes: a migratory ``\textit{go}'' phase and a proliferative ``\textit{grow}'' phase. 
In the \textit{go} regime, to which the trailing cells further to the left are subject, cells do not divide, but their motion is polarized to the right through a biased drift with strength $\chi\geq 0$. . 
By contrast, the leading cells in the invasion front are subject to the \textit{grow} regime: their diffusive motion is undirected and they 
undergo cell division, thereby creating new mass near the leading edge. 

Our analysis focuses on two \textit{go} or \textit{grow} models that differ in how the transition between the \textit{go} and \textit{grow} regime is triggered. In a first, \emph{nonlocal model}, with cell density $\rho(t,x)$, the state of a cell at position $x$ is governed by the cumulative mass to its right,
\be
\label{defi.P}
P(t,x) := \int_x^{+\infty} \rho(t,y)\,dy.
\ee
The switch between the two phases is mediated by $P(t,x)$. More precisely, after a renormalization of the switching threshold to the value 1, cells at positions $x$ with $P(t,x)\geq 1$ are in the migratory \textit{go} regime and those with $P(t,x)<1$ in the proliferative \textit{grow} regime. In a second, \emph{local model}, with density $u(t,x)$, the switch is instead mediated directly by the local density $u(t,x)$: the \textit{go} regime is determined by $u(t,x)\geq 1$, and the grow regime by $u(t,x)<1$. These switching rules are encoded at the PDE level by the following one-dimensional reaction–diffusion–advection equations for $\rho$ and $u$:
\begin{align}
    \label{pde.rho}
	&\rho_t + \chi (\alpha(P) \rho )_x = \rho_{xx} + \rho(1 - \alpha(P)),\\
    \label{pde.u}
	&u_t + \chi (\alpha(u) u   )_x = u_{xx} + u(1-\alpha(u)),
\end{align}
where $\chi \geq 0$ and
\be 
\alpha(s):=\mathbb{1}_{\{s\geq 1\}}.
\ee
{Before further discussion on the meaning of these equations, let us observe two obstructions to their mathematical analysis that are clear from their structure. First, these models are degenerate.  Indeed, $\alpha(u)$ and $\alpha(P)$, the coefficients of the transport terms, are discontinuous, so that~\eqref{pde.rho} and~\eqref{pde.u} must be interpreted in divergence form.  Second,~\eqref{pde.rho} does not enjoy a comparison principle.}

In both models, the \textit{go} regime corresponds to the region where $\alpha=1$: there, cells undergo advection with drift $\chi\geq 0$ and diffusion, but no cell division. The \textit{grow} regime corresponds to the region where $\alpha=0$: there, cells divide at rate $1$ and diffuse without drift. For the local model, the state $u\equiv 1$ is stable, and any initial density $u_{\rm in}$ with values in $[0,1]$ remains in $[0,1]$ for all later times. Accordingly, in the local case we restrict our attention only to solutions $u$ taking values in $[0,1]$. By contrast, this property fails for the nonlocal model: even at points where the density $\rho\geq 1$, it may happen that $P<1$ and cells would still be in the \textit{grow} regime. Thus, in the nonlocal case we allow $\rho$ to take values in $[0,+\infty)$.

To quantify the progression of the population, we track the position of the invasion front. In the local model \eqref{pde.u}, we define
\be 
\label{defi.xt.loc}
    \overline{x}^u(t) := \sup\left\{x\in\R  |  u(t,x) = 1\right\},
\ee
\textit{i.e.}, the rightmost point of the \textit{go} region, where the density has reached the stable state $u\equiv 1$. In the nonlocal model \eqref{pde.rho}, using the cumulative mass (to the right) $P$ from \eqref{defi.P}, we define the front position by
\be
\label{defi.xt.nonloc}
    \overline{x}^P(t) := \sup\left\{x\in\R | P(t,x) = 1\right\}.
\ee
Since we have $\rho(t,x)>0$ for every $t>0$, the level-set equation $P(t,x)=1$ admits a unique solution, which is the front position $\overline{x}^P(t)$. In the sequel, we drop the superscripts and simply write $\overline{x}(t)$, whenever the model we refer to is clear. A central goal of this work is to obtain asymptotic descriptions up to order 1 of $\overline{x}(t)$ as $t\to\infty$ for both the local and the nonlocal dynamics. {In particular, we aim to understand the relationship between the strength of the aerotaxis, represented by $\chi$, and the front position $\overline x(t)$.}

{From the biological point of view, our main interest is in~\eqref{pde.rho}, which first appeared in~\cite{demircigil2025}. It is, in some sense, more faithful to the model introduced in~\cite{cochet2021}. 
Mathematically, however, it is instructive to also consider~\eqref{pde.u} at the same time.  Initially our interest in~\eqref{pde.u} was mainly as a local counterpart for~\eqref{pde.rho}; however, as we detail in \Cref{ss.BerestyckiBrunetDerrida}, it is related to a free boundary problem introduced by Berestycki, Brunet, and Derrida~\cite{berestycki2018}.}

\subsubsection{The biological motivation}\label{sss.bio}

The present work is motivated by models of aerotactic cell invasion. 
In earlier work of one of the authors \cite{cochet2021,demircigil2022}, a ``go or grow''\footnote{The initial occurrence of ``\textit{go} or \textit{grow}'' describes a slightly different mechanism in glioma cells \cite{hatzikirou2012}.} mechanism was proposed for the amoeba \textit{Dictyostelium discoideum}. In that experimental setting, cells migrate up a self-generated oxygen gradient: as cells consume oxygen, they create a depleted zone behind them, which drives their displacement toward oxygen-rich regions.
There, the switching between migratory and proliferative behavior is mediated by the local oxygen concentration. Cells in oxygen-rich regions are in an accommodating environment and enter an essentially non-migratory but proliferative \textit{grow} state, whereas cells in oxygen-poor regions find themselves in a depleted environment and respond by engaging in directed aerotactic motion without dividing. This is consistent with the observation that cell division and aerotactic motion do not occur simultaneously.

In subsequent work \cite{demircigil2025}, this mechanism was revisited at the level of a finite-population stochastic individual-based model, in which particles on the real line branch and switch between \textit{go} and \textit{grow} states according to their relative position within the population. In this formulation, the oxygen-mediated switch is replaced by a rank-based rule that encodes the depletion of a limiting resource. 
A fixed threshold rank $K$ separates leading cells that have already traversed an accommodating, nutrient-rich environment from trailing cells that remain in a depleted region: the first $K$ cells are in the \textit{grow} state, while the remaining cells are in the \textit{go} state.
In a suitable large-population limit, this rank-based dynamics gives rise precisely to our nonlocal model \eqref{pde.rho}. The local model \eqref{pde.u} introduced here follows the same ``go or grow'' mechanism, but replaces the nonlocal switching based on the cumulative mass to the right by a purely density-dependent rule acting directly on the cell density $u(t,x)$. This can be interpreted as modeling cells that escape crowded regions, analogously to how cells in the aerotaxis setting leave zones of low oxygen.

\subsubsection{A broader reaction–diffusion–advection class}
A fundamental issue with analyzing the nonlocal model~\eqref{pde.rho} is that it does not enjoy a comparison principle.  Our first contribution is to observe that, by integrating~\eqref{pde.rho}, we obtain an equation for the cumulative mass $P$ for which comparison applies.  Indeed, using $P_x=-\rho$ one finds
\be\label{pde.P}
	P_t + \chi (A^P(P))_x = P_{xx} + P-A^P(P),
\ee
where
\be\label{e.A^P}
A^P(s):=\int_0^s\alpha(p)\,dp = (s-1)_+.
\ee
Here, for any $r\in \R$, $r_+ := \max\{0,r\}$ is the positive part.
This puts $P$ into a broader class of scalar reaction–diffusion–advection equations of the form
\be \label{pde.v}
v_t+ \chi (A(v))_x = v_{xx} + v-A(v),
\ee
for a nondecreasing nonlinearity $A$. {These equations {\em do} have a comparison principle.} The nonlinearity $A$ is the flux associated with the advective term, but in this class of models it is also intrinsically tied to the reaction term, which takes the form $v-A(v)$. This class \eqref{pde.v} includes, for instance, the classical Burgers–FKPP \cite{murray2002,leach2016,an2021} equation obtained by choosing $A(s)=s^2$, as well as both of our ``go or grow'' models. Indeed, setting
\be\label{e.A^u}
A^u(s):=s\alpha(s),
\ee
we see that the local model \eqref{pde.u} is another instance of \eqref{pde.v} with $A=A^u$. 

{There are several important differences, however, between the fluxes considered in our work versus what was considered in the previous investigations~\cite{AnHendersonRyzhik_Quantitative,an2021}.  First, our flux terms are significantly less regular: $A^P$ is only Lipschitz and $A^u$ is even discontinuous at $s=1$.  Second, it is crucial in previous work that $A(1)=1$, or, more generally, that the reaction term $v- A(v)$ vanishes at some positive $v$.  This provides an immediate $L^\infty$-bound on $v$. In our nonlocal case, we see that $P - A^P(P) >0$ for all $P>0$.  Finally, it is crucial in the analysis of~\cite{AnHendersonRyzhik_Quantitative,an2021} that the $A$ appearing there is strictly convex, which is neither the case for~\eqref{pde.P}, nor for~\eqref{pde.u}.  The sum of these differences is that the analysis is significantly more delicate in our setting.  In many places, completely new arguments are required.  In fact, even the well-posedness of the models requires nontrivial analysis.}

In the sequel, when referring to the nonlocal model we shall freely switch between the density formulation \eqref{pde.rho} and the cumulative mass formulation \eqref{pde.P}. These two descriptions are, of course, equivalent, since $\rho=-P_x$, and should be viewed as two faces of the same model rather than as distinct equations. Conceptually, we regard \eqref{pde.rho} as the primary, biologically motivated description of the dynamics, while~\eqref{pde.P} is the more mathematically tractable model.  In the sequel, we use the fact that~\eqref{pde.P} fits into the general class of models~\eqref{pde.v}.  This is convenient for the well-posedness analysis and for the study of the shape defect function.

\subsubsection{Pushed, pulled, and pushmi-pullyu fronts}

Our main interest in this work is to understand how the aerotaxis term influences front propagation starting from localized initial data, for instance
\be\label{e.c030701}
    \rho_{\rm in}(x),\,u_{\rm in}(x) = \mathbb{1}_{(-\infty,0]}(x),
\ee
describing a populated region on the left invading the empty state on the right. For both the local
and the nonlocal dynamics, the state $\rho,u\equiv 0$ is unstable, and one therefore expects invasion fronts of this state by mass from the left. The front positions $\overline x^P(t),\overline x^u(t)$ introduced above, \eqref{defi.xt.loc} and \eqref{defi.xt.nonloc}, captures this invasion.

A closely related situation has been analyzed for the Burgers–FKPP equation, which corresponds
to \eqref{pde.v} with the smooth flux $A(s)=s^2$; see, for example, \cite{an2021}. In that setting,
one finds a family of traveling waves connecting $1$ on the left to $0$ on the right, with a minimal
speed $c_*(\chi)$ depending on the drift strength $\chi$. The traveling waves are explicit, and the
minimal speed is given by
\be \label{e.c_star_chi}
    c_*(\chi)
    =\left\{ \begin{array}{ll}
        \chi+\frac1\chi
            \quad& \text{ if }\chi\geq 1\\
        2
            \quad& \text{ if }\chi \in [0,1)\\
        \end{array} \right.,
\ee
so that the parameter $\chi$ naturally splits the dynamics into three regimes usually referred to as
pulled ($\chi\in [0,1)$), pushed ($\chi>1$), and an intermediate, or pushmi–pullyu, case at $\chi=1$. 

Our two ``go or grow'' models admit explicit traveling waves with exactly the same minimal speed function $c_*(\chi)$. Indeed, for both $\rho$, solution of \eqref{pde.rho}, and $u$, solution of \eqref{pde.u}, there exists a minimal speed $c_*(\chi)$, given by \eqref{e.c_star_chi}, and a corresponding traveling wave. This fact has already been observed in \cite{cochet2021,demircigil2022} for the analogous ``go or grow'' model. The speed and the traveling wave shapes can be summarized as follows:

\medskip
\begin{center}
\begin{tabular}{ || c | c | c | c | c ||}
	\hline\rule{0pt}{2.2ex}
	 &  $c_*(\chi)$ & ${u_{\rm tw}}(x)$ & ${\rho_{\rm tw}}(x)$ & ${P_{\rm tw}}(x)$ \\
	\hline\hline
		\rule[-4.0ex]{0pt}{9.2ex}
	$\chi \in [0,1)$ & $2$ & 
     $\begin{cases}  1  & \text{ if } x \leq 0 ,\\
     ((1-\chi)x + 1) e^{-x}  & \text{ if } x > 0.\end{cases}$
     & $\frac{1}{2-\chi}u_{\rm tw}(x)$ & $\begin{cases}
        1 - \frac{x}{2-\chi}, &  \text{ if }  x \leq 0,\\[0.15cm]
        \frac{(1-\chi)x + 2 - \chi}{2-\chi}\, e^{-x}, & \text{ if }  x > 0.
      \end{cases}$
	\\
	\hline 
		\rule[-3.5ex]{0pt}{8.2ex}
        $\chi =1$ & $2$ &
        $\min\{1, e^{-x}\}$
        &$u_{\rm tw}(x)$
 &      
      $\begin{cases}
        1 - x, &  \text{ if }  x \leq 0,\\[0.1cm]
        e^{-x}, &  \text{ if } x > 0,
      \end{cases}$
     \\
	\hline
		\rule[-3.5ex]{0pt}{8.2ex}
	$\chi > 1$ & $\chi + \frac{1}{\chi}$ & 
    $
     \min\{1, e^{-\chi x}\}$ & $\chi u_{\rm tw}(x)$ &
      $\begin{cases}
        1 - \chi x, &  \text{ if } x \leq 0,\\[0.1cm]
        e^{-\chi x}, &  \text{ if } x > 0.
      \end{cases}$
     \\
	\hline
\end{tabular}
\captionof{table}{The minimal speed traveling waves for the local and nonlocal ``go or grow'' models. $c_*(\chi)$ denotes the minimal wave speed and is given by \eqref{e.c_star_chi}, $u_{\rm tw}$ is a traveling wave solution of the local equation \eqref{pde.u}, $\rho_{\rm tw}$ a traveling wave solution of the nonlocal density equation \eqref{pde.rho}, and 
$P_{\rm tw}$ the cumulative mass to the right of $\rho_{\rm tw}$, i.e. $P_{\rm tw}(x)=\int_x^{+\infty} \rho_{\rm tw}(y)\,dy$, which is itself a traveling wave solution of 
the cumulative mass equation \eqref{pde.P}.
} \label{table.tw}
\end{center}

From  this table one sees that the qualitative behavior changes precisely at $\chi=1$.  When $\chi \in [0,1)$, the minimal speed is fixed at $c_*(\chi)=2$ and for all $\chi\in [0,1)$ the traveling wave profile has the same qualitative shape: its tail is always exponential, multiplied by a linear prefactor. In this regime, the selected speed coincides with the linear spreading speed determined by the behavior near $\rho,u=0$, and one speaks of a \emph{pulled} front.
When $\chi>1$, the situation is different: the speed $c_*(\chi)=\chi+\chi^{-1}$ increases strictly with $\chi$ and the profile  has a purely exponential tail with decay rate $\chi$. Here, the nonlinear dynamics near the bulk of the front modify the speed selected by the linearization, and the front is said to be \emph{pushed}.
The borderline case $\chi=1$ is particularly sensitive, with a purely exponential tail at speed $2$ but both linear and nonlinear effects playing a role; following \cite{an2021}, we refer to this intermediate regime as \emph{pushmi–pullyu}.

The differences between these three cases also manifest in the front asymptotics for the Cauchy problem.  In general, one expects pulled fronts to have a Bramson delay \cite{AveryScheel,ebert2000,berestycki2018,needham2004, hamel2013,an2021, AnHendersonRyzhik_Quantitative,giletti2022, roberts2013, bramson1978, bramson1983}:
\be\label{e.pulled_delay}
	\overline x(t) = 2t - \frac{3}{2}\log t + O(1),
\ee
which is notable for the unbounded error term between the front $\overline x$ for the solution of~\eqref{pde.rho} or \eqref{pde.u} starting from initial data like~\eqref{e.c030701} and that of the traveling wave.  Here we are abusing notation by allowing $\overline x$ to refer to the front for any problem.  In contrast, pushed fronts do not have such a delay~\cite{fife1977,rothe1981,roquejoffre1997,lucia2004}:
\be\label{e.pushed_front}
	\overline x(t) = c_*(\chi) t + O(1).
\ee
Finally, pushmi-pullyu fronts keep a logarithmic delay, but with the coefficient $-\sfrac12$ in place of the classical Bramson value $-\sfrac32$~\cite{ebert2000, needham2004, an2021, berestycki2019,AnHendersonRyzhik_Quantitative,giletti2022}:
\be\label{e.pp_delay}
	\overline x(t) = 2t - \frac12\log t + O(1).
\ee

Our goal is to establish the front asymptotics as above for the equations~\eqref{pde.rho} and \eqref{pde.u}.  We leave the question of stability of the traveling wave for a future work.

\subsubsection{The shape defect function}

A central tool in our analysis is the \emph{shape defect function}, which was introduced in \cite{an2021,AnHendersonRyzhik_Quantitative,AnHendersonRyzhik_convergence} to quantify how far a solution is from a traveling wave profile. We adapt this notion to the general flux equation \eqref{pde.v}.

Let $v_{\rm tw}$ be a nonincreasing traveling wave solution of \eqref{pde.v} with speed $c>0$, and assume that its image is either $[0,1]$ (\textit{e.g.} for $u_{\rm tw}$) or $[0,\infty)$  (\textit{e.g.} for $P_{\rm tw}$). We associate to $v_{\rm tw}$ a \emph{wave profile function} $\eta\geq0$ defined implicitly by 
\be
\label{def.eta}
    - v_{\rm tw}'(z)= \eta\left(v_{\rm tw}(z)\right), \quad z\in\R.
\ee
In other words, $\eta(s)$ records the slope of the traveling wave at points where the solution takes the value $s$.  We emphasize that properties of $\eta^u$ and $\eta^P$ can be derived from the explicit traveling waves; see \Cref{table.tw}.

Given a wave profile function $\eta$ and a nontrivial solution $v(t,x)$ of \eqref{pde.v}, we then define the corresponding \emph{shape defect function}
\be \label{def.omega}
    \omega(t,x)
    := -v_x(t,x) - \eta\left(v(t,x)\right).
\ee
By~\eqref{def.eta}, $\omega\equiv 0$ if and only if $v$ is exactly a translate of the traveling wave $v_{\rm tw}$ associated with $\eta$. 
More generally, the sign and magnitude of $\omega$ measure how much the spatial profile $x\mapsto v(t,x)$ deviates from the traveling wave shape. In particular, $\omega\geq 0$ corresponds heuristically to a profile that is at least as steep as the reference wave.

We use the notation $\omega$, when referring to the shape defect function associated with $v$ solution to \eqref{pde.v} and use $w$ (resp. $W$), when referring specifically to the shape defect function associated with $u$ (resp. $P$) solution to \eqref{pde.u} (resp. \eqref{pde.P}):
\be\label{def.w}
\begin{split}
w=-u_x-\eta^u(u)
\\
W:=-P_x-\eta^P(P).
\end{split}\ee
A key result of the paper is that, for the minimal speed waves, the nonnegativity of $w$ and $W$ is preserved by the dynamics. This provides control on the steepness of the front and will be used repeatedly to derive estimates on the front position and on the decay of the solution ahead of the invasion front.

\subsection{Main Results}\label{ss.main_results}

We begin by addressing the well-posedness of the models. For the nonlocal model \eqref{pde.P}, the flux $A^P$ is Lipschitz continuous. Consequently, its well-posedness and the corresponding comparison principle follow from standard parabolic tools. On the other hand, the well-posedness of the local model \eqref{pde.u} is much more delicate due to the discontinuity of the flux $A^u$, and we highlight this result here.  Define $\zeta_\gamma(x) = \exp\{-\gamma \sqrt{1 + x^2}\}.$

\newcommand{\MainTheoremText}{
Fix $\chi \geq 0$, $\gamma \in (0,\chi^{-1})$, and let $c=c_*(\chi)$ be as in~\eqref{e.c_star_chi}. 
Consider $u_{\rm in} \in L^2(\R, \zeta_\gamma dx)$ taking values in $[0,1]$ such that $u_{\rm in} \not\equiv 0,1$. 
Suppose that, in the distributional sense, 
\be
    w_{\rm in}
    :=-u_{{\rm in},x}
        -\eta^u(u_{\rm in})\geq 0,
\ee
where $\eta^u$ is the wave profile function associated with the wave speed $c$. 
Then, there exists a unique weak solution $u\in C([0,+\infty); L^2(\R, \zeta_\gamma dx))\cap L^2_{\rm loc}([0,+\infty); H^1(\R, \zeta_\gamma dx))$ to \eqref{pde.u} satisfying $u(0,\cdot)=u_{\rm in}$, such that the associated shape defect function $w:=-u_{x}-\eta^u(u)$ satisfies 
\be 
    w\geq 0 \qquad \text{ and }\qquad w\in L^2_{\rm loc}([0,+\infty) ;H^1_{\rm loc}(\R)).
\ee
Moreover, define
\be
    \overline x(t) = \sup \{ y : u(t,y) = 1\}.
\ee   
If $\chi = 0$, then $u \in C^{1-\delta,2-\delta}_{\rm loc}((0,+\infty)\times \R)$ and $\overline x \in C^{\sfrac12 - \delta}_{\rm loc}((0,+\infty))$ for any $\delta\in (0,\nicefrac{1}{2})$. 
If $\chi > 0$, then $u \in C^{\nicefrac{1}{2},1}_{\rm loc}((0,+\infty)\times \R)$ and $\overline x \in C^{\sfrac12}_{\rm loc}((0,+\infty))$.
}

\begin{theorem}[Well-posedness for the Local Model]
\label{thm.wellposedness_u}
\MainTheoremText
\end{theorem}

The proof of \Cref{thm.wellposedness_u} is based on a regularization procedure. We approximate the discontinuous flux $A^u$ by a family of Lipschitz regularizations. By choosing a convenient regularization, we are able to prove the preservation of nonnegativity for the shape defect function. The difficulty here is subtle; the definition of $w$ depends on $\eta$, which, in turn, depends on the flux $A$.  Hence, when one regularizes the flux, it is not clear that the initial nonnegativity of $w$ for the unregularized problem is preserved for the regularized problem.

The existence for the regularized problem is standard. The obstacle is in establishing strong enough bounds to use compactness in the deregularization limit.  We do this by applying the ideas of Nash to obtain $L^1-L^\infty$ bounds on the shape defect function.  This provides spatial  Lipschitz regularity of $u$, which, after interpolation, yields regularity in time as well.

At this point one may apply compactness to extract a convergent subsequence that inherits both the regularity and the nonnegative shape defect. The regularity of this limit then allows us to establish the regularity of the front position $\overline{x}(t)$. Finally, we identify the limit and show that it is a weak solution to \eqref{pde.u}. The primary technical difficulty in this final step is justifying the convergence of the nonlinear term $(A^u(u))_x$, which we are able to do by using the regularity of $\overline{x}(t)$.

The uniqueness of this solution relies on a surprising comparison principle (see \Cref{thm.CP_u} for a precise statement). Indeed, the comparison principle requires the supersolution to have its associated shape defect function to be nonnegative. We were unable to prove a comparison principle without this restriction. To our knowledge, this is the only example of an approach where the shape defect function plays such a crucial role in establishing uniqueness. 

With the well-posedness established, our second main result characterizes the long-time behavior up to order $\mathcal{O}(1)$ of the invasion front $\overline{x}(t)$.

\begin{theorem}[Front asymptotics]
\label{thm-asymptotics}
Fix $\chi \geq 0$ and let $c_*(\chi)$ be as in~\eqref{e.c_star_chi}.
        Suppose that $P$ (resp. $u$) solves~\eqref{pde.P} (resp.~\eqref{pde.u}) with nontrivial initial data. Assume the initial shape defect function is nonnegative, meaning $W_{\rm in} := -P_{{\rm in},x} - \eta^P(P_{\rm in}) \geq 0$ (resp. $w_{\rm in} := -u_{{\rm in},x} - \eta^u(u_{\rm in}) \geq 0$), where $\eta^P$ (resp. $\eta^u$) is the wave profile function associated with the wave speed $c_*(\chi)$. 
        Suppose additionally that, for $\chi_{\lowvee} := \max\{ 1, \chi \}$ and some $\gamma>0$, the following integrability and decay assumptions are satisfied:
        \begin{itemize}
            \item \textbf{Nonlocal case:} 
            \be 
                \int P_{\rm in}(x)e^{\nicefrac{x}{\chi_{\lowvee}}}\,dx<\infty
                \qquad\text{and}\qquad
                W_{\rm in}(x)\leq \gamma^{-1}e^{-\gamma x_+^2}.
            \ee
            Furthermore, if $\chi \in [0,1)$, we assume the strict bound $\| \rho_{\rm in}\|_\infty < 1$.
            
            \item \textbf{Local case:}
            \be 
                \int u_{\rm in}(x)e^{\nicefrac{x}{\chi_{\lowvee}}}\,dx<\infty
                \qquad\text{and}\qquad
                0\leq u_{\rm in}(x)\leq \min\left\{1,\gamma^{-1}e^{-\gamma x_+^2}\right\}.
            \ee
        \end{itemize}
        Then, the position $\overline x(t)$ of the front exhibits the following trichotomy. There exists $C>0$ such that, for $t\geq 0$:
        \begin{itemize}
            \item \textbf{Pushed case:} if $\chi>1$, then
            \be
                c_*(\chi)t  - C\leq \overline{x}(t) \leq c_*(\chi)t  + C;
            \ee
            \item \textbf{Pushmi-pullyu case:} if $\chi=1$, then, recalling $c_*(\chi)=2$,
            \be
                2t - \frac{1}{2} \log t - C\leq \overline{x}(t) \leq 2t - \frac{1}{2} \log t + C;
            \ee
            \item \textbf{Pulled case:} if $\chi\in [0,1)$, then, recalling $c_*(\chi)=2$,
            \be
                2t - \frac{3}{2} \log t - C\leq \overline{x}(t) \leq 2t - \frac{3}{2} \log t + C.
            \ee
        \end{itemize}
\end{theorem}

This single theorem summarizes six distinct results (an upper and a lower bound for each of the three regimes $\chi>1$, $\chi=1$, and $\chi \in [0,1)$), each requiring its own specific assumptions. The integrability and decay conditions listed in \Cref{thm-asymptotics} guarantee that the requirements for all six individual cases are met simultaneously. However, we refer the reader to the statements of the individual propositions in the subsequent sections for the exact assumptions required for our proofs of each specific bound.

Regarding the bound $\|\rho_{\rm in}\|_\infty < 1$ imposed in the pulled regime ($\chi \in [0,1)$), we note that this condition is natural, as the minimal traveling wave obeys an even stronger bound (see \Cref{table.tw}). We impose this restriction purely for technical reasons to facilitate the construction of our supersolution for the upper bound. We expect that this is merely an artifact of our proof and could be relaxed at the expense of a more delicate argument.

While the proofs for each regime differ in spirit, they primarily rely on two main tools: the estimation of exponential moments and the construction of sub- and super-solutions. Several aspects of these proofs are noteworthy. First, any supersolution constructed for the local model must possess a nonnegative shape defect function in order for the comparison principle to apply, as we have mentioned above. Second, for the upper bound in the pushmi-pullyu case ($\chi=1$) of the nonlocal model, we are not able to construct a supersolution directly for $P$. Instead, we construct a supersolution directly on the shape defect function $W$ and deduce the bound on $\overline{x}(t)$ from this framework. This appears to be a novel approach in the study of front propagation {that is significantly more efficient than previous arguments. We believe this could be useful in other problems, and we see it as one of the main mathematical contributions of this paper.} 
Finally, an additional technical difficulty arises in the nonlocal model: because of the structure of the equation for $\rho$ \eqref{pde.rho}, there is no easy \textit{a priori} upper bound for its $L^\infty$-norm (unlike more standard reaction-diffusion equations where an upper bound of 1 follows straightforwardly). To overcome this, we utilize the shape defect function and a decay property of its weighted version to control the $L^\infty$-norm of $\rho$ dynamically, which ultimately allows us to conclude the argument through a functional inequality.

\subsection{Connection Between the Local Model and a Model of Berestycki, Brunet, and Derrida}\label{ss.BerestyckiBrunetDerrida}

Interestingly, the local equation~\eqref{pde.u} is intimately connected to a free-boundary problem introduced by Berestycki, Brunet, and Derrida \cite{berestycki2018}. To study the front location asymptotics of Fisher-KPP equations in a more explicitly computable way, they proposed replacing the nonlinear Fisher-KPP equation with a linear equation coupled to a free boundary condition:
\be\label{e.BBD}
    \begin{cases}
        u_t = u_{xx} + u \qquad \text{for } x > \overline x(t),\\
        u(t, \overline x(t)) = 1 \quad\text{and}\quad u_x(t,\overline x(t)^+) = - \chi.
    \end{cases}
\ee
Note that, in~\eqref{e.BBD}, $\overline x$ is treated as an unknown that allows the system to support the overdetermined Dirichlet and Neumann boundary conditions.

In fact, the solution $u$ to our local model \eqref{pde.u} exactly satisfies the free-boundary problem \eqref{e.BBD}: on the set $\{x > \overline x(t)\}$, we have $u < 1$, so $\alpha(u)=0$ and the nonlinear equation \eqref{pde.u} simplifies to $u_t = u_{xx} + u$. The boundary conditions are recovered through the regularity of the shape defect function $w = -u_x - \eta^u(u)$. By \Cref{thm.wellposedness_u}, $w \in L^2_{\rm loc}([0,+\infty); H^1(\R, \zeta_\gamma dx))$, which implies that for almost every $t > 0$, $x \mapsto w(t, x)$ is continuous. Because $\eta^u(u) \to \chi$ as $u \to 1^-$, the continuity of $w(t,\cdot)$ at the free boundary $x = \overline x(t)$ guarantees that $u$ satisfies a Rankine-Hugoniot type condition and in particular that 
\be
\label{e.RH_u}
u_x(t, \overline x(t)^+) = -\chi.
\ee 
Thus, our solution satisfies the free-boundary problem for almost every $t$. As a side note, the density $\rho$ in the nonlocal model satisfies a similar Rankine-Hugoniot condition across the free boundary:
\be
\label{e.RH_rho}
    \rho_x(t,\overline{x}(t)^+) - \rho_x(t,\overline{x}(t)^-) = -\chi \rho(t,\overline{x}(t)),
\ee
for almost every $t\geq 0$. While we do not explicitly use this condition in the present article, it can be established through similar arguments by showing that $W \in L^2_{\rm loc}([0,+\infty); H^2(\R, \zeta_\gamma dx))$. 

The formulation \eqref{e.BBD} can be viewed as a generalization of the widely studied $\chi=0$ regime. In the probabilistic literature, the $\chi=0$ case arises naturally as the complementary cumulative distribution function (the cumulative mass to the right) of the hydrodynamic limit of the $N$-branching Brownian motion ($N$-BBM) \cite{demasi2019}. Furthermore, it is closely connected to the inverse first passage time problem (see, e.g., Theorem B.1 in \cite{berestycki2024}). For this $\chi=0$ regime, well-posedness in the context of classical solutions was established in \cite{lee2018,berestycki2019}. 

In contrast to the aforementioned literature, which fundamentally relies on classical regularity, our work establishes well-posedness for weak solutions. Within this weak framework, the $\chi=0$ regime is structurally much simpler because the discontinuous advection term $\chi(A^u(u))_x$ vanishes entirely, allowing us to prove uniqueness without requiring the nonnegativity of the shape defect function $w$. For $\chi>0$, however, the presence of this discontinuous term makes the weak formulation more delicate and necessitates the restriction $w \geq 0$ to ensure uniqueness (see \Cref{thm.CP_u}).

Very recently, Berestycki, Penington, and Tough \cite{berestycki2025} conducted a rigorous study of the general $\chi \geq 0$ free-boundary problem \eqref{e.BBD}, relying primarily on probabilistic methods. In Theorem 1.14 of \cite{berestycki2025}, they establish well-posedness for classical solutions when $\chi>0$ via an explicit one-to-one transformation that maps the problem back to the $\chi=0$ case. Regarding the front asymptotics, they achieve $\smallO(1)$ accuracy across the pulled, pushmi-pullyu, and pushed regimes, and they are able to classify the convergence to non-minimal traveling waves depending on the decay properties of the initial condition.

While their impressive work provides highly precise asymptotics for the linear free-boundary formulation, our approach remains distinct and complementary. We achieve $\mathcal{O}(1)$ accuracy for the convergence to the minimal wave speed in the case of steep initial data, but we do so by working directly with the PDE formulation \eqref{pde.u} on $\R$. Rather than relying on probabilistic representations, our methodology uses inherently PDE techniques, leveraging the shape defect function and comparison principles to extract the front dynamics. Although we did not pursue this thread of inquiry in general here, our approach lays the ground for sharp quantitative estimates on the convergence to the traveling wave; see~\cite{AnHendersonRyzhik_convergence}, the forthcoming work of Patterson~\cite{Patterson}, and the discussion in \Cref{sss.convergence_pp}. {Additionally, we point out that \Cref{thm.wellposedness_u} appears to be the first result establishing H\"older regularity of the free boundary (the work in~\cite{berestycki2019,berestycki2025} obtains only continuity).}

{We remind the reader that, from the point of view of the application to understanding aerotaxis (recall \Cref{sss.bio}), our main interest lies in the nonlocal model.  This is another distinction between the works, as~\cite{berestycki2025} does not investigate~\eqref{pde.rho}.}

\subsection{Organization of the paper}

The remainder of the paper is organized roughly in two main parts. 
The first half of the paper focuses on the development of a theory of the shape defect function in our setting and then establishes the well-posedness of solutions to our equations. The second half of the paper considers the front propagation problem and relies on much of the general theory developed in the first half.  We give a more precise outline now.

The first well-posedness half of the paper consists in \Cref{s.shapedefect} and \Cref{s.wp}.  First, \Cref{s.shapedefect} is devoted to the shape defect function $\omega$ associated with the general equation \eqref{pde.v}. We derive its governing evolution equation under the assumption that the flux $A$ is Lipschitz continuous, a condition that notably excludes the local model. We then detail the structural properties of the traveling wave profile functions $\eta^P$ and $\eta^u$. Next, in \Cref{s.wp}, we establish the well-posedness and comparison principles for both the local and nonlocal models, alongside rigorous proofs for the nonnegativity and decay estimates of the shape defect functions.

The front propagation half of the paper is the remainder of the sections. 
In \Cref{s.prelim}, we introduce a family of exponential moments and establish a preliminary, rough upper bound on the front location $\overline{x}(t)$ that is utilized throughout the subsequent analysis. \Cref{s.pushed} treats the pushed regime ($\chi > 1$), where we leverage the exponential moment from \Cref{s.prelim} combined with the nonnegativity of the shape defect function to derive a matching lower bound and conclude the $\mathcal{O}(1)$ asymptotics. \Cref{s.pp} is dedicated to the pushmi-pullyu regime ($\chi = 1$). Here, the lower bound employs techniques that are similar to, though slightly more subtle than, those used in the pushed regime, whereas the upper bound relies on the construction of explicit supersolutions. In particular, for the nonlocal model, the supersolution is constructed directly on the shape defect function $W$. In \Cref{s.pulled}, we address the pulled regime ($\chi \in [0,1)$), deriving the classical Bramson delay via comparison with the Fisher-KPP equation and further supersolution constructions. These sections comprise the theory of the front location.

\subsection{Notation}\label{ss.notation}
Throughout the paper, we use $C$ to be a positive constant, depending only on $\chi$ and the initial data.  We often transform our solutions in several ways.  We adopt consistent notation for each of these.  Let us elucidate this here.  Fix any function $v: [0,+\infty)\times  \R \to \R$.  We use the tilde for the $ct$ moving frame:
\be
    \tilde v(t,z) = v(t,z+ct).
\ee
We also change to using a $z$ variable. 
We use breve for the log-shifted moving frame:
\be
    \breve v(t,z) = v(t,z+ct - r \log (t+t_0)).
\ee
where $r= \sfrac12$ or $\sfrac32$ and $t_0$ is a large constant.  We use hat to denote an exponential transform such as
\be
    \hat v(t,x) = e^{\nicefrac{x}{\chi_{\lowvee} }} v(t,x),
\ee
where $\chi_{\lowvee} :=\max\{ \chi, 1\}$. Finally, we use bar to denote a supersolution $\overline v$ (resp. subsolution $\underline v$) to the equation $v$ solves.  Often we use bar and hat in conjunction with one of the shifts associated to tilde or breve; however, we do not use both decorations at the same time.  For example, we simply write $\hat v$ in place of $\hat{\tilde v}$ for simplicity.

\subsubsection*{Acknowledgements}

CH warmly thanks Julien Berestycki and \'Eric Brunet for discussions about their prior work, and both authors thank Julien Berestycki, Sarah Penington, and Oliver Tough for a discussion of their work at the BIRS workshop ``Emerging Connections between Reaction-Diffusion, Branching Processes, and Biology'' while both manuscripts were being completed.

CH was supported by NSF grants DMS-2204615, DMS-2337666, and DMS-2617615. MD has received funding from the European Research Council (ERC)
under the European Union’s Horizon 2020 research and innovation program (grant agreement
No 865711). The authors acknowledge support of the Institut Henri Poincar\'e (UAR 839 CNRS-Sorbonne Université), and LabEx CARMIN (ANR-10-LABX-59-01).

\section{The Shape Defect Function $\omega$}
\label{s.shapedefect}

In this section, we present a general framework to study \emph{shape defect functions}
associated with~\eqref{pde.v}. We first work at an abstract level, under structural
assumptions on the flux $A$ and on the traveling wave, and then specialize to our two
models of interest:
\be
    A^P(s) = (s-1)\1_{\{s\geq 1\}}
    \quad\text{(nonlocal model)}, 
    \qquad\text{ and }\qquad
    A^u(s) = \1_{\{s=1\}}
    \quad\text{(local model)}.
\ee
In order to introduce this quantity, let $v_{\rm tw}$ be a nonincreasing traveling wave solution of~\eqref{pde.v} with speed $c>0$, whose image satisfies $\overline{\im(v_{\rm tw})} = [0,1]$ or $[0,\infty)$. We associate to $v_{\rm tw}$ a  wave profile function $\eta\geq 0$, defined implicitly by
\be
\label{d.eta}
    - v_{\rm tw}'= \eta(v_{\rm tw}).
\ee
By differentiating \eqref{d.eta}, we have $-v_{\rm tw}''=\eta'(v_{\rm tw})v_{\rm tw}'=-\eta'(v_{\rm tw})\eta(v_{\rm tw})$. Furthermore, since $v_{\rm tw}$ is a traveling wave and
\be\label{pde.general_tw}
    - c v_{\rm tw}'
        + \chi A(v_{\rm tw})'
        - v_{\rm tw}''
        = v_{\rm tw} - A(v_{\rm tw}),
\ee
we find that $\eta$ satisfies, for $s\in \im(v_{\rm tw})$,
\be 
\label{ode.eta}
    c\eta(s)-\eta'(s)\eta(s)-\chi A'(s)\eta(s)
        =s-A(s).
\ee
At the boundary $\partial \im(v_{\rm tw})$, we extend $\eta$ by setting $\eta=0$. Indeed, if $s \in \partial \im(v_{\rm tw}) \cap \im(v_{\rm tw})$, then $v_{\rm tw}'=0$ and thus $\eta(s)=0$ by~\eqref{d.eta}. If instead $s \in \partial \im(v_{\rm tw}) \setminus \im(v_{\rm tw})$, then $\eta(\sigma)\to 0$ as $\sigma\to s$ with $\sigma\in\im(v_{\rm tw})$. Conversely, any nonnegative function $\eta$ solving~\eqref{ode.eta} on
$\im(v_{\rm tw})$ with $\eta=0$ on $\partial\im(v_{\rm tw})$ determines a traveling wave
$v_{\rm tw}$ (up to translation) through~\eqref{d.eta}.

Given a wave profile function $\eta$ and a solution $v(t,x)$ to the parabolic equation~\eqref{pde.v}, we define the associated shape defect function
\be \label{d.w}
    \omega
    =-v_x-\eta(v).
\ee
We recall that in the nonlocal model ($v=P$), the corresponding shape defect function is denoted by $W$, while in the local model ($v=u$), it is denoted by $w$. 
Intuitively, $\omega$ measures how much the spatial profile of $v(t,\cdot)$ deviates from the traveling wave shape.

The analysis of shape defect functions is a powerful tool in the study of convergence toward traveling waves \cite{AnHendersonRyzhik_convergence}. 
In the sequel, we rely on two crucial properties of the shape defect function: the preservation of nonnegativity, which is used below in the proof of the upper and lower bounds of $\overline {x}(t)$, and a weighted decay property for the nonlocal shape defect function $W$, which is used to establish the bound on $\rho$ in \Cref{c.rhoLinfty}. 
The rigorous proof of both properties is the purpose of the present section, together with the next one. 
More precisely, in the present section we lay the analytical groundwork by deriving an evolution equation for the shape defect function $\omega$ and by studying the structural properties of the profile function $\eta$ in the local and nonlocal cases. In the next section, we establish the well-posedness theory for both models. 
This framework will then allow us to prove the two key properties described above: preservation of nonnegativity and the weighted decay estimate used in the asymptotic analysis.

In \Cref{ss.eq_omega}, we first rewrite the equation for $v$ in the moving frame in terms of the shape defect function $\omega$.
Under the sole assumption that 
\be 
\label{d.Q}
Q:=\eta+\chi A
\ee
is Lipschitz continuous, we obtain a formulation of the evolution of $v$ (in the moving frame) in terms of $\omega$ and the auxiliary quantity
\be
\label{d.R}
R=c-Q'.
\ee
We note that $R$ allows us to consolidate~\eqref{ode.eta} into the compact form:
\be 
\label{e.R}
    \eta(s)R(s)=s-A(s).
\ee
Under the stronger assumption that both $A$ and $\eta$ are Lipschitz continuous, we further derive a parabolic equation satisfied by $\omega$ itself. This latter formulation applies directly in the nonlocal model, where the flux $A^P$ is Lipschitz continuous, but not in the local model, where $A^u$ is discontinuous.

In \Cref{ss.prop_eta}, we then establish the structural properties of the traveling wave profile functions $\eta^P$ and $\eta^u$ associated with the nonlocal and local models, respectively. In particular, we describe their regularity, concavity, and boundary behavior, as well as the corresponding properties of the functions $Q^P$ and $Q^u$. These results will be essential in the subsequent analysis of the shape defect function in both models.

\subsection{Equation of the shape defect function $\omega$.}
\label{ss.eq_omega}

In this subsection, we first derive an equation for the evolution of $v$, solution of~\eqref{pde.v}, in terms of the shape defect
function
\be 
  \omega = -v_x - \eta(v).
\ee 
Then, under certain structural assumptions, we derive an evolution equation of the shape defect function $\omega$. More precisely, $\omega$ solves a parabolic equation with drift
and reaction terms expressed in terms of $\eta$ and the auxiliary function $R$ defined
in~\eqref{d.R}. This evolution equation on $\omega$ applies in the nonlocal case, where the flux $A^P$ is Lipschitz, but not directly in the local case. 
Regardless, this equation of the shape defect function will be the starting point for our nonnegativity and decay estimates, in the nonlocal and in the local case.

\begin{proposition}\label{l.w_eqn} Let $\chi\geq 0$ and $v\in L^2_\text{\rm loc}([0,+\infty)\times \R)\cap L^2_\text{\rm loc}([0,+\infty); H^1_\text{\rm loc}(\R)) $ be a nonnegative solution to~\eqref{pde.v}.
Consider $\eta$, a wave profile function satisfying \eqref{ode.eta} for a certain wave speed $c$, as well the corresponding shape defect function $\omega=-v_x-\eta(v)$. Assume that $Q:=\eta+\chi A$ is Lipschitz continuous.
Then, the equation on $v$ in the moving frame (with $\tilde{v}(t,z):=v(t,z+ct), \tilde{\omega}(t,z):=\omega(t,z+ct)$) can be written as
\be 
\label{e.fw}
\tilde{v}_t = -\tilde{\omega}_z -R(\tilde{v})\tilde{\omega},
\ee 
where $R=c-Q'$. 
Furthermore, if we assume $A$ to be Lipschitz continuous, then $\tilde{\omega}$ satisfies, in the distributional sense,
 \be\label{e.w_eqn}
 	\tilde \omega_t  = \tilde{\omega}_{zz} + (R(\tilde{v})\tilde{\omega})_z +\eta'(\tilde{v})\tilde{\omega}_z+\eta'(\tilde{v})R(\tilde{v})\tilde{\omega}.
 \ee
\end{proposition}

Before moving on with the proof, let us make several remarks. For the first identity \eqref{e.fw} it suffices that $Q$ be Lipschitz continuous. For the evolution equation~\eqref{e.w_eqn} of $\omega$, we additionally assume that $A$ is Lipschitz continuous. This, along with the assumption on $Q$ yields the regularity of $\eta$ and makes all terms in~\eqref{e.w_eqn} well-defined. In this case, it is clear from~\eqref{e.w_eqn} that $\omega$ remains nonnegative if it is initially nonnegative, as desired. In the next subsection, we show that the Lipschitz continuity assumption 
on $Q$ is met in the nonlocal case (see \Cref{l.l.eta_nonloc}). Moreover, we  see that in the $\chi \geq 1$ case applying \eqref{e.w_eqn} to $\tilde{\omega}=\tilde W$ and using that 
\be
    R
        \equiv\frac{1}{\chi}
    \qquad\text{ and }\qquad
    \eta'(\tilde{P})
        =\chi(1-\alpha(\tilde{P}))
\ee
leads to  
\be 
\label{e.tildeWchipush}
\tilde W_t +\left( \chi \alpha(\tilde P)-c\right) \tilde W_z = \tilde{W}_{zz} +(1-\alpha(\tilde P)) \tilde W.
\ee
In the general case, and thus in the $\chi\in [0,1)$ case, we obtain 
\be 
\label{e.tildeWchipull}
\tilde W_t +\left( \chi \alpha(\tilde P)-c\right) \tilde W_z = \tilde{W}_{zz} +  \left( R'(\tilde{P})\tilde{P}_z+
\left(\eta^P\right)'(\tilde{P})R(\tilde{P})
\right) \tilde{W}.
    \ee
Using the identity $R+\eta'=c-\chi A'$, which follows from the definition $R=c-Q'$ and $Q=\eta+\chi A $, and differentiating $\eta(s)R(s)=s-A(s)$, we may express the equation as follows
\be 
\tilde W_t +\left( \chi \alpha(\tilde P)-c\right) \tilde W_z = \tilde{W}_{zz} +(1-\alpha(\tilde P)) \tilde W+2Q''(\tilde P)\eta(\tilde{P}) \tilde{W}+Q''(\tilde P) \tilde{W}^2.
\ee
We observe that for $\chi\geq 1$, we have $Q''\equiv 0$ and thus we recover \eqref{e.tildeWchipush}.

In the local case, there are issues interpreting~\eqref{e.w_eqn} due to the lower regularity of $\eta$.  
Nevertheless, from the well-posedness theory developed below in \Cref{ss.loc.wp}, we still get that the shape defect function $\tilde{w}=-\tilde{u}_z-\eta(\tilde{u})$ preserves nonnegativity.

\begin{proof}[Proof of \Cref{l.w_eqn}]
Rewriting, $\omega$ and taking a derivative, we see that, in the distributional sense
\be
    v_{zz} = - (\eta(v))_z - \omega_z.
\ee
Additionally, we find that 
\be
    (Q(\tilde{v}))_z
    = Q'(\tilde{v})\tilde{v}_z.
\ee
Recall that $Q'\in L^\infty$ by assumption. We, thus, obtain
\be 
\begin{aligned}
    \tilde v_t&=\tilde v_{zz}+c\tilde v_z-\chi (A(\tilde v))_z+(\tilde v-A(\tilde v)) 
    =-\tilde \omega_z+\left(c\tilde v-\eta(\tilde v)-\chi A(\tilde v)\right)_z+(\tilde v-A(\tilde v))
    \\&
    =-\tilde \omega_z+\left(c\tilde v-Q(\tilde v)\right)_z+(\tilde v-A(\tilde v))
    =-\tilde \omega_z+\left(c-Q'(\tilde v)\right)\tilde v_z+(\tilde v-A(\tilde v))
    \\&
    =-\tilde \omega_z+R(\tilde{v})(-\eta(\tilde v)-\tilde \omega)+(\tilde v-A(\tilde v))
    =-\tilde \omega_z-R(\tilde{v})\tilde{\omega}+(\tilde v-A(\tilde v)) -R(\tilde{v})\eta(\tilde v)
    \\&
    =-\tilde \omega_z-R(\tilde{v})\tilde{\omega},
\end{aligned}
\ee
where the last equality is due to the identity~\eqref{e.R}: $\eta R=s-A(s)$.

Next, under the additional assumption that $A$ be Lipschitz continuous, the solution $\tilde v  \in L^2_\text{\rm loc}([0,+\infty)\times \R)\cap L^2_\text{\rm loc}([0,+\infty); H^1_\text{\rm loc}(\R))$ has increased regularity. Indeed,
\be 
    \tilde v_t-\tilde v_{xx}
    = c \tilde v_x - \chi A'(\tilde v)\tilde v_x+\tilde v-A(\tilde v)
    \in L^2_\text{\rm loc}([0,+\infty)\times \R),
\ee
and, by standard parabolic regularity theory, we have that $\tilde v_t\in L^2_\text{\rm loc}([0,+\infty)\times \R)$. Observe that the Lipschitz continuity of $A$ and $Q$ yields the same of $\eta$; recall~\eqref{d.Q}.  This justifies the application of the chain rule:
\be
    (\eta(\tilde v))_t=\eta'(\tilde v)\tilde v_t.
\ee
To obtain then \eqref{e.w_eqn}, we differentiate $\tilde{\omega}=-\tilde{v}_z-\eta(\tilde{v})$ with respect to $t$ and use~\eqref{e.fw}, leading to
\be 
\tilde{\omega}_t=\tilde{\omega}_{zz}+(R(\tilde{v})\tilde{\omega})_z+\eta'(\tilde{v})\tilde{\omega}_z+\eta'(\tilde{v})R(\tilde{v}) \tilde{\omega},
\ee
which is the desired result.
\end{proof}

\subsection{Properties of the Traveling Wave Profile Functions $\eta^P$ and $\eta^u$.}\label{ss.prop_eta}

In this subsection, we describe in detail the traveling wave profile functions $\eta^P$ and $\eta^u$
associated with the nonlocal and local fluxes $A^P$ and $A^u$, respectively. Our goal is to understand
the regularity and monotonicity properties of $\eta$ and the quantity
$Q=\eta+\chi A$.

First, we can notice that for $\chi\geq 1$ and $c=\chi+\sfrac{1}{\chi}$, the profile function satisfies $\eta(s)=\chi(s-A(s))$. This expression is the consequence of a straightforward algebraic manipulation and is true for all $A$. In our specific cases, this gives
\be\label{e.eta_and_Q_expressions}
\begin{cases}
    \eta^u(s)=\chi s \1_{\{s<1\}},\\
    \eta^P(s)=\chi \min\{1,s\},
\end{cases}
\quad\text{ and }\quad
\begin{cases}
    Q^u(s)=\chi s, \\
    Q^P(s)=\chi s .
\end{cases}
\ee
In the regime $\chi<1$, the wave profile functions $\eta^u$ and $\eta^P$ are no longer piecewise linear and exhibit a more involved structure than in the case $\chi\geq 1$. 
Using the explicit expressions of the traveling waves (see~\Cref{table.tw}), we represent the corresponding profile functions in terms of the secondary real branch $W_{-1}$ of the Lambert $W$ function. 
Recall that $W_{-1}$ is the inverse of the mapping $(-\infty,-1]\ni r\mapsto r e^{r}\in[-e^{-1},0)$. 
This leads to the following explicit formulas and structural properties.

\begin{lemma}[Properties of $\eta^u$ and $Q^u$.]\label{l.l.eta_loc}
Fix $\chi \geq 0$, $c=c_*(\chi)$ as in~\eqref{e.c_star_chi}. The traveling wave profile, $\eta = \eta^u$, associated to $A^u$ (recall~\eqref{e.A^u}) is given by, for $s\in [0,1)$,
\be 
\label{eta.exp_u}
\eta(s)=
\begin{cases}
    \chi s
    & \text{ if } \chi \geq 1,
    \\
    \left( 1+ \frac{1}{W_{-1}(-\kappa s)}\right)s
    & \text{ if } \chi \in [0,1),
\end{cases}
    \ee
    where 
    \be 
    \kappa:= \frac{1}{1-\chi}e^{-\frac{1}{1-\chi}},\ee
    and $W_{-1}$ is the secondary branch of the Lambert $W$ function. 
    For $s=1$, we have $\eta(1)=0$. 

\smallskip
    
\noindent
\textbf{Behavior near $s=0$ and near $s=1$.} The profile $\eta$ satisfies
\be 
\label{value.eta_u'}
\begin{split}
   &\eta(0^+) =0,
    \quad
    \eta'(0^+) = \max\{1,\chi\},
    \quad
    \eta(1^-)= \chi,
    \quad \text{and}\quad
    \eta'(1^-) = 
    \begin{cases}
     c - \frac{1}{\chi}
    \qquad 
    &\text{if } \chi>0, \\
    -\infty
    \qquad 
    &\text{if } \chi=0,
    \end{cases},
\end{split}
\ee
For $\chi \geq 1$ and $s< 1$, $\eta''(s)=0$, whereas for $\chi \in [0,1)$, we have
\be
\label{value.eta_u''}
    \eta''(0^+)=
    -\infty
    \quad\text{and}\quad
    \eta''(1^-)=
    -\frac{(1-\chi)^2}{\chi^3}.
\ee

\smallskip

\noindent
\textbf{Uniform bounds of $\eta$.} For all $s \in [0,1]$, the function $\eta$ satisfies
\be 
\label{eta.unif.bound_u}
    \chi (s - A^u(s))
    \leq \eta^u(s)
    \leq \max\{1,\chi\}(s - A^u(s)).
\ee

\noindent
\textbf{Properties of $Q$.} The associated function $Q=Q^u$ belongs to $C^0([0,1]) \cap C^1_{\rm loc}([0,1))$, and if $\chi>0$ then $Q \in C^1([0,1])$. 
Moreover, for all $s\in (0,1)$,
\be
    Q''(s) \leq 0
		\quad
		\text{ and }
		\quad
        \max\{1,\chi\}
		\geq Q'(s)
		\geq \begin{cases}
     c - \frac{1}{\chi}
    \quad 
    &\text{if } \chi>0, \\
    -\infty
    \quad 
    &\text{if } \chi=0 .
\end{cases}
\ee
\end{lemma}

\begin{lemma}[Properties of $\eta^P$ and $Q^P$.]\label{l.l.eta_nonloc}
Fix $\chi \geq 0$, $c=c_*(\chi)$ as in~\eqref{e.c_star_chi}. The traveling wave profile, $\eta = \eta^P$, is given by, for $s\in [0,1)$,
\be 
\label{eta.exp_P}
\eta(s)=
\begin{cases}
    \chi s & \text{ if } \chi \geq 1,\\
      \left( 1+ \frac{1}{W_{-1}(-\lambda  s)}\right)s & \text{ if } \chi \in [0,1),
\end{cases}
    \ee
    where 
    \be 
    \lambda:= \frac{2-\chi}{1-\chi}e^{-\frac{2-\chi}{1-\chi}},\ee
    and $W_{-1}$ is the secondary branch of the Lambert $W$ function. 
    For $s\geq 1$, we have 
    \be 
\eta(s)= \frac{1}{c-\chi}
    \ee

\smallskip
    
\noindent
\textbf{Behavior near $s=0$ and $s=1$.} The profile $\eta$ satisfies
\be 
\begin{split}
    &\eta(0^+) = 0,
    \quad
    \eta'(0^+) = \max\{1,\chi\},
    \quad
    \eta(1^-)=\frac{1}{c-\chi},
    \quad\text{and}\quad
    \eta'(1^-)
    =\chi.
    \end{split}
\ee
For $\chi \geq 1$ and $s\neq 1$, $\eta''(s)=0$, whereas for $\chi \in [0,1)$, we have
\be
    \eta''(0^+)=
    -\infty
    \quad\text{and}\quad
    \eta''(1^-)=
    -(2-\chi)(1-\chi)^2.
\ee

\smallskip

\noindent
\textbf{Uniform bounds of $\eta$.} For all $s \geq 0$, the function $\eta$ satisfies
\be 
\label{eta.unif.bound_P}
    \frac{1}{c - \chi}(s - A^P(s))
    \leq \eta(s)
    \leq \max\{1,\chi\}(s - A^P(s)).
\ee

\noindent
\textbf{Properties of $Q$.} The associated function $Q=Q^P$ belongs to $C^2((0,1)\cup(1,+\infty))$,  $Q'' $ is bounded at $s=1$, and $Q$ satisfies for $s> 0$,
\be
    Q''(s) \leq 0
		\quad
		\text{ and }
		\quad\chi 
			\leq Q'(s)
			\leq 
         \max\{1,\chi\}.
\ee
\end{lemma}

We now prove both lemmata in order.

\begin{proof}[Proof of \Cref{l.l.eta_loc}]
The case $\chi \geq 1$ follows directly from~\eqref{e.eta_and_Q_expressions}.  We consider only the case $\chi \in [0,1)$. The formula~\eqref{eta.exp_u} follows from straightforward algebraic manipulations using the definition of the secondary branch $W_{-1}$ of the Lambert function. We omit the details.

In the sequel, we use often the following, which follow directly \Cref{table.tw} and straightforward computations: for $z>0$,
\be \label{e.c011501}
    u_{\rm tw}(z)
        = ((1-\chi)z + 1) e^{-z}
    \quad
    u'_{\rm tw}(z)
    =-((1-\chi)z + \chi )e^{-z}
    \quad\text{ and }
    \quad
    u''_{\rm tw}(z)
    =((1-\chi)z + 2\chi-1 )e^{-z}.
\ee  
From this, we prove the preliminary estimates on $\eta$.  The behavior as $s\searrow 0$ is:
\be
\begin{split}
    \eta(0^+)
    = \lim_{s \to 0^+}\eta(s)
    = \lim_{z \to +\infty} - u'_{\rm tw}(z)
    = 0
    \quad\text{ and }\quad
    \eta'(0^+)
    =\lim_{u \to 0}\frac{\eta(u)}{u}
    = \lim_{z \to +\infty} \frac{- u'_{\rm tw}(z)}{u_{\rm tw}(z)}
    = 1.
\end{split}
\ee
The behavior as $s\nearrow1$ is:
\be
\begin{split}
    &\eta(1^-)
    =
    \lim_{u \to 1^-}\eta(u)
    = \lim_{z \to 0^+} - u'_{\rm tw}(z)
    = \chi
    \quad\text{ and}
    \\&
    \eta'(1^-)
    = \lim_{s\to 1^-} \frac{\eta(1^-) - \eta(s)}{1 - s}
    = \lim_{z \to 0^+}\frac{\chi + u'_{\rm tw}(z)}{1 - u_{\rm tw}(z)}
    = \begin{cases}
     2 - \frac{1}{\chi}
    \qquad 
    &\text{if } \chi>0, \\
    -\infty
    \qquad 
    &\text{if } \chi=0 .
\end{cases}
\end{split}
\ee

Now, we establish the properties of $Q$. 
From the values above, we see that $Q(0^+)=0$ and $Q(1^-)=\chi$.  Since $Q(0)=0$ and $Q(1)=\chi$, we have that $Q\in C^0([0,1])$. Then, using the fact that $\eta'=Q'$ on $(0,1)$, we have that $Q\in C^1([0,1])$ for $\chi>0$ and $Q\in C^1([0,1))$ for $\chi=0$.

Let us then show that $Q'' = \eta'' \leq 0$.  Notice that $\eta'(u_{\rm tw}) u_{\rm tw}' = - u_{\rm tw}''$.  Differentiating once more, we have
\be
    \eta''(u_{\rm tw}) (u_{\rm tw}')^2 + \eta'(u_{\rm tw}) u_{\rm tw}''
        = - u_{\rm tw}'''.
\ee
This leads to
\be
    \eta''(u_{\rm tw})
        = \frac{ \left(u''_{\rm tw}\right)^2- u'''_{\rm tw} u'_{\rm tw}}{( u'_{\rm tw})^3}
        =-\frac{1}{ u'_{\rm tw}}\left( \frac{\ u''_{\rm tw}}{ u'_{\rm tw}}\right)'.
\ee
Recall that $u_{\rm tw}' < 0$.  Also, using that $ u''_{\rm tw}(z)=((1-\chi)z + 2\chi-1 )e^{-z}$ for $z>0$ yields
\be
\left( \frac{u''_{\rm tw}}{u'_{\rm tw}}\right)' =-\left( \frac{(1-\chi)z + 2\chi-1}{(1-\chi)z + \chi }\right)' =-\left( \frac{\chi-1	}{(1-\chi) z +\chi}\right)^2.
\ee 
From this, we deduce that $\eta'' = Q'' \leq 0$ in $(0,1)$ as well as the limit values of $\eta''$ in \eqref{value.eta_u''}.

Observe that~\eqref{eta.unif.bound_u} follows directly from~\eqref{e.c011501} and the fact that $\eta(u_{\rm tw}) = -u_{\rm tw}'$.  We omit the details.

Finally, we obtain the estimates on $Q$ by concavity.  Indeed, for $s\in (0,1)$,
\be 
    \max\{1,\chi\}
    = \eta'(0^+)
    = Q'(0^+)
    \geq Q'(s)
    \geq Q'(1^-)
    = \begin{cases}
     2 - \frac{1}{\chi}
    \qquad 
    &\text{if } \chi>0, \\
    -\infty
    \qquad 
    &\text{if } \chi=0 ,
\end{cases}
\ee
which concludes the proof.
\end{proof}

\begin{proof}[Proof of \Cref{l.l.eta_nonloc}.]
The proof follows the same strategy as above, using the explicit expressions of the traveling waves in \Cref{table.tw}.
For $\chi\in [0,1)$ and $z>0$,
\be
    P_{\rm tw}'(z)=-\frac{1}{2-\chi}\left((1-\chi)z + 1 \right)e^{-z}
    \quad\text{ and }\quad
    P_{\rm tw}''(z)=\frac{1}{2-\chi}\left((1-\chi)z + \chi \right)e^{-z},
\ee
while, for $z<0$,
\be 
    P'_{\rm tw}(z)
    =-\frac{1}{2-\chi}.
\ee
We merely compute the values $\eta''(0^+)$ and $\eta''(1^-)$ for $\chi<1$:
\be
\begin{split}
    &\eta''(0^+)
    = \lim_{s\to 0^+} \frac{\eta'(s) - \eta'(0^+)}{s}
    = \lim_{z \to +\infty}\frac{-\frac{P''_{\rm tw}(z)}{P'_{\rm tw}(z)} -1}{P_{\rm tw}(z)}
    = \lim_{z \to +\infty}\frac{-P''_{\rm tw}(z)-P'_{\rm tw}(z)}{P_{\rm tw}(z)P'_{\rm tw}(z)}
    = -\infty
    \quad\text{ and}\\
    &\eta''(1^-)
    = \lim_{s\to 1^-} \frac{\eta'(1^-) - \eta'(s)}{1-s}
    = \lim_{z \to 0^+}\frac{\chi +\frac{P''_{\rm tw}(z)}{P'_{\rm tw}(z)} }{1-P_{\rm tw}(z)}
    = \lim_{z \to 0^+}\frac{\chi P'_{\rm tw}(z)+P''_{\rm tw}(z)}{P'_{\rm tw}(z)(1-P_{\rm tw}(z))}
    = -(2-\chi)(1-\chi)^2.
\end{split}
\ee
The rest of the proof is omitted.
\end{proof}

\section{Well-Posedness and Comparison Principle} 
\label{s.wp}

The purpose of this section is to provide the well-posedness and comparison tools needed for both the local model~\eqref{pde.u} and the nonlocal model~\eqref{pde.P}. 
In addition, within this well-posedness framework, we establish two fundamental properties of the associated shape defect function: preservation of nonnegativity and a weighted decay estimate. These properties were used repeatedly in the preceding asymptotic analysis.

Throughout this section, solutions are constructed in a weighted energy space adapted to the unbounded spatial domain and to accommodate solutions with at most linear growth (\textit{e.g.} $P(t,x)=\int_x^{+\infty}\rho(t,y)dy$ will have linear growth as $x\to-\infty$ as $\rho $ will roughly behave like a plateau at $x=-\infty$). More precisely, for $\gamma\in (0, \chi^{-1})$, we introduce an exponentially decaying weight
\be 
\label{d.zeta}
\zeta_\gamma(x)=e^{-\gamma\sqrt{1+x^2}}
\ee
and the functional space
\be 
\label{d.V}
V_\gamma:=C([0,T]; L^2(\R,\zeta_\gamma(x) dx))\cap L^2([0,T]; H^1(\R,\zeta_\gamma(x) dx)).
\ee 
The weight allows us to justify integrations by parts on  $\R$, whereas the restriction $\gamma\in (0, \chi^{-1})$ ensures that the derivatives of $\zeta_\gamma$ can be absorbed by the reaction term and plays a purely technical role.  It can almost certainly be relaxed, but this is not our main interest.

In \Cref{ss.nonlocal}, we study the nonlocal model. First, the well-posedness theory is established more generally for \eqref{pde.v}, which we recall
\be 
v_t + \chi (A(v))_x - v_{xx} = v - A(v),
\ee
under the condition that $A$ be nondecreasing and Lipschitz continuous with $A(0)=0$. Recall that $A^P(s)=(s-1)_+$ satisfies these conditions. This general framework is needed for several arguments developed later in the section: the local model is constructed through Lipschitz regularizations $A^\e$, and several steps in its analysis rely on estimates that are robust at this level of generality. Then, we establish the preservation of nonnegativity and a weighted decay estimate  for the nonlocal shape defect function $W$.

In \Cref{ss.loc.wp}, we turn our attention to the local model \eqref{pde.u}. To overcome the difficulty posed by the discontinuity of $A^u$ at $u=1$, our analysis proceeds through a regularization procedure. First, we introduce a Lipschitz approximation $A^\e$ of the flux and carefully detail the properties of its associated wave profile function $\eta^\e$. We then construct a family of approximate solutions $u^\e$ and establish crucial uniform bounds, including the preservation of nonnegativity and the weighted decay of the approximate shape defect function $w^\e$. From the uniform bounds on $w^\e$, we deduce uniform H\"{o}lder bounds and extract a converging subsequence. Leveraging this H\"{o}lder regularity, we establish the regularity of $\overline{x}(t)$, which is a crucial element in proving that the limit of the subsequence satisfies the PDE \eqref{pde.u}. Finally, we conclude by proving the comparison principle.

\subsection{The Nonlocal Model}
\label{ss.nonlocal}

We begin with the well-posedness of~\eqref{pde.v} and its comparison principle in \Cref{sss.well-posedness-nonlocal}.  Afterwards, in \Cref{sss.shape-defect-nonlocal}, for the shape defect function we prove nonnegativity and establish a decay estimate.

\subsubsection{Well-posedness of the Nonlocal Model}\label{sss.well-posedness-nonlocal}

\begin{proposition}[Well-posedness for a Lipschitz continuous flux $A$]
\label{thm.wellposedness_P}
Consider a Lipschitz continuous and nondecreasing flux $A$ with $A(0)=0$.
For any $v_{\rm in}\in L^2(\R,\zeta_\gamma dx)$, there exists a unique weak solution $v\in V_\gamma$, given by \eqref{d.V},
to \eqref{pde.v} satisfying $v(0,\cdot)=v_{\rm in}$.
\end{proposition}

Let us note that the argument used to establish uniqueness in \Cref{thm.wellposedness_P} readily extends to sub- and super-solutions.  As a result, we easily deduce a comparison principle.

\begin{corollary}[Comparison Principle for a Lipschitz continuous flux $A$]\label{thm.CP_P}
Consider a Lipschitz continuous and nondecreasing flux $A$ with $A(0)=0$, and 
$
\underline{v},\overline{v}\in  V_\gamma
$, given by \eqref{d.V},
a sub- and a super-solution  to \eqref{pde.v}, \textit{i.e.} 
\be\begin{split}
&\underline{v}_t+\chi (A(\underline{v}))_x-\underline{v}_{xx}-(\underline{v}-A(\underline{v}))\leq 0 \quad\text{ and} \\
&\overline{v}_t+\chi (A(\overline{v}))_x-\overline{v}_{xx}-(\overline{v}-A(\overline{v}))\geq 0.
\end{split}\ee
If $ \underline{v}_{\rm in}\leq \overline{v}_{\rm in}$, then $\underline v \leq \overline v$.
\end{corollary}

\begin{proof}[Proof of \Cref{thm.wellposedness_P}]
Let us first obtain an energy estimate using \eqref{pde.v} and integrating by parts:
\be
\begin{split}
\frac{1}{2}\frac{d}{dt}\int v^2 \zeta_\gamma
&=
-\int \left(v_x\right)^2\zeta_\gamma -\int vv_x(\zeta_\gamma)_x
+\chi \int v_x A(v) \zeta _\gamma
+\chi \int vA(v) (\zeta_\gamma)_x
+\int v^2\zeta_\gamma
-\int vA(v) \zeta_\gamma.
\end{split}
\ee
Using the fact that $A(0)=0$ and $A$ nondecreasing, we have $vA(v)\geq 0$. Combined with the fact that $|(\zeta_\gamma)_x|\leq \gamma \zeta_\gamma$, this leads to
\be 
\chi \int vA(v) (\zeta_\gamma)_x-\int vA(v) \zeta_\gamma\leq \gamma\chi \int vA(v) \zeta_\gamma-\int vA(v) \zeta_\gamma\leq 0,
\ee
since $\gamma\chi<1$.
Then, again using $A(0)=0$ and the fact that $A$ is Lipschitz continuous, we have that for $s\in \R$,
\be 
\label{bd.A.Lip}
|A(s)|\leq \text{Lip}(A)|s|.
\ee
Using this bound with Young's inequality leads to the energy estimate
\be
\label{bd.energy_est}
\frac{1}{2}\frac{d}{dt}\int v^2 \zeta_\gamma + \frac12 \int \left(v_x\right)^2\zeta_\gamma
\leq  \left(\gamma^2 + \chi^2 \text{Lip}(A)^2 + 1 \right)\int v^2\zeta_\gamma.
\ee
Via a classical Galerkin approximation combined with this energy estimate and the Aubin-Lions lemma, we obtain existence of a solution $v\in V_\gamma$, where $V_\gamma$ is given by \eqref{d.V}. 

Uniqueness follows by obtaining a similar energy estimate \eqref{bd.energy_est}. Consider two solutions $v_1,v_2$ of \eqref{pde.v}, testing $(v_1-v_2)_+\zeta_\gamma$ against the difference equation and using $(v_1-v_2)_+(A(v_1)-A(v_2))\geq 0$ by monotonicity of $A$ yields
\be
\label{bd.energy_est_unique}
\frac{1}{2}\frac{d}{dt}\int (v_1-v_2)_+^2 \zeta_\gamma + \frac12 \int \left((v_1-v_2)_+\right)_x^2\zeta_\gamma
\leq  \left(\gamma^2 + \chi^2 \text{Lip}(A)^2 + 1 \right)\int (v_1-v_2)_+^2\zeta_\gamma.
\ee
Uniqueness is then a simple consequence of Gr\"onwall's inequality.
\end{proof}

\subsubsection{The shape defect function for the nonlocal model}\label{sss.shape-defect-nonlocal}

We prove here that nonnegativity is preserved.  Recall the wave profile function $\eta^P$ from \Cref{l.l.eta_nonloc}.

\begin{proposition}\label{prop.W.pos}
    Suppose that  $P\in V_\gamma$ solves~\eqref{pde.P} and consider the associated shape defect function $W:=-P_{x}-\eta^P(P)$.
    Suppose that 
    \be 
    W_{\rm in}:=-P_{{\rm in},x}-\eta^P(P_{\rm in})\geq 0,
    \ee
    then on $[0,+\infty)\times \R$,
    \be 
W\geq 0.
    \ee
\end{proposition}

\begin{proof}
   We work in the moving frame $(t,z)=(t,x-ct)$, with $c=c_*(\chi)$ as in~\eqref{e.c_star_chi}, so that we may directly apply \Cref{l.w_eqn}, which is stated in that frame. Therefore, as $A^P$ and $\eta^P$ are Lipschitz continuous, $\tilde{W}(t,z)=W(t,z+ct)$ satisfies \eqref{e.w_eqn}. By using that $R+\eta'=c-\chi A'$ and differentiating $\eta(s)R(s)=s-A(s)$, we may rewrite the equation as follows
\be 
\label{e.tildeW}
\tilde W_t +\left( \chi \alpha(\tilde P)-c\right) \tilde W_z = \tilde{W}_{zz} +(1-\alpha(\tilde P)) \tilde W+2Q''(\tilde P)\eta(\tilde{P}) \tilde{W}+Q''(\tilde P) \tilde{W}^2.
\ee
Since $Q''\leq 0$, we have that $Q''(\tilde P) \tilde{W}^2 \leq 0$ and thus by a classical argument, we can show that the $L^2$-norm of the negative part satisfies a Gr\"onwall inequality, from which we conclude that $\tilde{W}(t,\cdot)\geq 0 $  for all $t\geq 0$.
\end{proof}

Now we introduce a weighted version of the shape defect function,
\be 
\label{d.hatW}
\hat{W}(t,z)=e^{\nicefrac{z}{\chi_{\lowvee}}}\tilde{W}(t,z),
\ee
where $\chi_{\lowvee}= \max\{ 1, \chi \}$. This weighted version has a nice decay property, which is a key property used in the proof of the upper bound on $\rho$ (\Cref{c.rhoLinfty}) and enables us to control the $L^\infty$-norm of $\rho$ in the go region $\{P\geq 1\}$. 

\begin{proposition}\label{thm.W.decay}
    Suppose that  $P\in V_\gamma$ solves~\eqref{pde.P}, that $\hat W$ is given by~\eqref{d.hatW}, and that 
    \be \int P_{\rm in}(x)e^{\nicefrac{x}{\chi_{\lowvee}} }dx<+\infty.
    \ee
    Then, there exists $C>0$ such that 
	\be
    \label{bd.W.decay}
		\left\| \sqrt{t}\hat{W} \right\|_{L^\infty([0,+\infty)\times\R)} \leq  C\left\| \hat{W}_{\rm in} \right\|_{L^1(\R)} .
	\ee
\end{proposition}

In place of proving the \Cref{thm.W.decay} in this specific setting, we show that it is a consequence of a general lemma that we state now, but whose proof is momentarily postponed.

\begin{lemma}[$L^1$-$L^\infty$ bound]
\label{l.Linf_L1} Consider $a(t,x)\in L^\infty([0,+\infty)\times \R)$ and $b(t,x)\in L^\infty_{\rm loc}((0,+\infty)\times \R)$ satisfying
 \be 
a_x\leq 0, \quad \text{ and }\quad b\geq 0.
    \ee
Suppose that $g$ is a solution to
    \be 
g_t-g_{xx}+a(t,x)g_x+b(t,x)g=0,
    \ee
    with initial datum $g_{\rm in}\in L^1(\R)$.
    
    Then, there exists a universal constant $C>0$ independent of $a$ and $b$, such that for all $t>0$,
    \be 
 \left\| \sqrt{t}g \right\|_{L^\infty([0,+\infty)\times\R)} \leq  C\left\| g_{\rm in} \right\|_{L^1(\R)}
    \ee
    Moreover, if $g_{\rm in}\in L^2(\R)$, then for all $T>0$, there exists $C_T>0$ such that 
    \be 
    \label{bd.nash.l2dx}
    \| g_x\|_{L^2([0,T]\times\R)}
 \leq C_T\| g_{\rm in}\|_{L^2(\R)}.
    \ee
\end{lemma}

Let us observe that \eqref{bd.nash.l2dx} is only briefly used in the sequel to show that $w$ the shape defect function associated with $u$ is in $L^2_{\rm{loc}}([0,+\infty);H^1(\R,\zeta_\gamma dx))$, which we then use to observe that the solution of the local model satisfies the Rankine-Hugoniot condition \eqref{e.RH_u} almost everywhere.

Let us now check that \Cref{l.Linf_L1} applies to $\hat W$.

\begin{proof}[Proof of \Cref{thm.W.decay}]
Let us note that, up to approximation, we may assume that $W_{\rm in}$ is continuous and bounded.  Using the arguments developed below, it is easy to check that these properties persist forward in time.

In the $\chi \geq 1$ case, we find
\be 
\label{e.hatWpush}
\hat{W}_t + \underbrace{\left( \chi \alpha(\tilde P)-\chi+\frac{1}{\chi}\right) }_{a(t,z):=}\hat{W}_z =  \hat{W}_{zz}, 
\ee
Recall that $\alpha = \1_{[1,\infty]}$ so that $a_z = \chi \alpha'(\tilde P)\tilde P_z \leq 0$.  Hence \Cref{l.Linf_L1} applies.

In the $\chi \in [0,1)$ case, we find
\be 
\label{e.hatWpull}
\hat{W}_t + \underbrace{ \chi \alpha(\tilde P) }_{=:a(t,z)}\hat{W}_z + \underbrace{ \left((1-\chi)\alpha(\tilde P)-2Q''(\tilde P)\eta(\tilde P)-Q''(\tilde{P})\tilde{W}\right) }_{=: b(t,z)}\hat{W}=  \hat{W}_{zz}.
\ee
Clearly $a$ is bounded and nonincreasing since $\alpha$ is bounded and nondecreasing and $\tilde P$ is nonincreasing.  Hence, we need only consider $b$. Let us recall from \Cref{l.l.eta_nonloc} that for $\chi\geq 1$, $Q''\equiv 0$ and thus $b=(1-\chi)\alpha(\tilde P)\in L^\infty([0,+\infty)\times \R)$. For $\chi\in [0,1)$, $Q''(0^+)=-\infty $, but is elsewhere bounded.
Yet, $\tilde{P}(t,z)=\int_z^\infty\tilde{\rho}(t,y)dy$ is clearly locally bounded from below by a positive constant since $\tilde{\rho}(t,z)> 0$ for any $t>0$ and $z\in \R$, and thus $b\in L^\infty_{\rm{loc}}((0,+\infty)\times \R)$. 
If we had additionally assumed that $W_{\rm in}\geq 0$, we would be finished at this point, since $Q''\leq 0$, according to  \Cref{l.l.eta_nonloc}. Let us show how to proceed without the additional assumption on $W_{\rm in}$.

We first rewrite $b$:
\be 
\begin{aligned}
    b(t,z)&=(1-\chi)\alpha(\tilde P)-2Q''(\tilde P)\eta(\tilde P)-Q''(\tilde{P})\tilde{W}
    \\
    &=(1-\chi)\alpha(\tilde P)-Q''(\tilde P)\eta(\tilde P)-Q''(\tilde{P})\tilde{P}_z \quad &(\text{using }\tilde{W}=-\tilde P_z-\eta(\tilde{P}))
    \\ 
    &=(1-\chi)\alpha(\tilde P)+R'(\tilde P)\eta(\tilde P)-Q''(\tilde{P})\tilde{P}_z
    \quad &(\text{using }Q''=-R')
    \\ 
    &=(1-\chi)\alpha(\tilde P)+1-\alpha(\tilde P)-R(\tilde P)\eta'(\tilde P)-Q''(\tilde{P})\tilde{P}_z
    \quad &(\text{differentiating }R(s)\eta(s)=s-A(s))
    \\ 
    &=1-\chi\alpha(\tilde P)-R(\tilde P)\eta'(\tilde P)-Q''(\tilde{P})\tilde{P}_z
    \\ 
    &=-1+R(\tilde P)+\eta'(\tilde P)-R(\tilde P)\eta'(\tilde P)-Q''(\tilde{P})\tilde{P}_z 
    \quad &(\text{using }R=2-\chi \alpha -\eta')
    \\ 
    &=\left(R(\tilde P)-1\right)\left(1-\eta'(\tilde P)\right)-Q''(\tilde{P})\tilde{P}_z .
\end{aligned}
\label{exp.c}
\ee
By \Cref{l.l.eta_nonloc}, $-Q'' \tilde P_z \geq 0$.  Hence, we consider only the terms in parentheses.  Since $\eta$ is concave and $\eta'(0) = 1$, we immediately see that $\eta'(\tilde P) \leq 1$.  For the other term, notice that $R$ is nondecreasing due to the concavity of $Q$ (recall \Cref{l.l.eta_nonloc}).  Hence
\be
    R(\tilde P) - 1
    \geq R(0) - 1
    = 0.
\ee
It follows that $b\geq 0$ and \Cref{l.Linf_L1} applies, which completes the proof.
\end{proof}

We conclude this section by giving the proof of \Cref{l.Linf_L1}, which is based on a standard application of the Nash inequality. 

\begin{proof}[Proof of \Cref{l.Linf_L1}] If $b\geq 0$ is only locally bounded, the result follows by approximation with the truncations $b^N=\min\{b,N\}$. Hence, without loss of generality, we assume $b\in L^\infty([0,+\infty)\times\R)$ in the proof. Observe that, unlike $b$, the coefficient $a$ cannot be truncated (\textit{e.g.} $a^N=\min\{\max\{a,-N\},N\}$) without affecting the sign condition $a_x\leq 0$. For this reason, global boundedness of $a$ is assumed from the outset.

\noindent
    {\bf \# Step (1): $L^\infty$-contraction. } 
Since $b\geq 0$, observe that the parabolic maximum principle yields that the $L^\infty$-norm of $g$ is decreasing in time.

\noindent
    {\bf \# Step (2): $L^1$-contraction. } 
Momentarily consider the case where $g_{\rm in}\geq 0$, and therefore $g \geq 0$. Integration by parts gives
\be 
\frac{d}{dt}\int_{\R} g(t,x)\,dx= \int_{\R} a_x(t,x) g(t,x)dx-\int_{\R} b(t,x)g(t,x)dx
\leq  0,
\ee
as $a_x\leq 0$, $b\geq 0$, and $g\geq 0$. Hence,
\be 
\label{bd.L1_contract_pos}
\|g(t)\|_{L^1(\R)} \leq \|g_{\rm in}\|_{L^1(\R)}.
\ee 
For general initial data, let $p$ and $n$ solve the same equation with initial data $p(0,\cdot)=(g_{\rm in})_+$ and $n(0,\cdot)=(g_{\rm in})_-$. Then $p,n\geq 0$ and by linearity
$g=p-n$, so $|g|\le p+n$. Using \eqref{bd.L1_contract_pos} for $p$ and $n$,
\be
\label{bd.L1_contract}
\|g(t)\|_{L^1(\R)}\leq \|p(t)\|_{L^1(\R)}+\|n(t)\|_{L^1(\R)}
\leq  \|p(0)\|_{L^1(\R)}+\|n(0)\|_{L^1(\R)}
=\|g_{\rm in}\|_{L^1(\R)}.
\ee

\noindent
    {\bf \# Step (3): $L^1$-$L^2$ bound using Nash inequality.} 
    Multiplying the equation by $g$ and integrating over $\R$, we obtain
\be
\label{bd.energy}
\frac{d}{dt}\int_{\R} g^2 dx + 2\int_{\R} g_x^2 dx
= \int_{\R} a_x g^2 dx - 2\int_{\R} b g^2 dx \leq 0,
\ee
using $a_x\leq 0$ and $b\geq 0$. \eqref{bd.nash.l2dx} is a direct consequence of this bound. Furthermore, using the Nash inequality, we obtain the differential inequality 
	\be
		\frac{d}{dt} \int g^2 dx
			\leq -2\int_{\R} g_x^2 dx \leq  - \frac{1}{C} \frac{\left( \int g^2 dx\right)^3}{\left( \int |g|\right)^4}
			\leq - \frac{1}{C} \frac{\left( \int g^2 dx\right)^3}{\left( \int| g_{\rm in}| dx\right)^4},
	\ee
	which can easily be solved and leads to
    \be 
\int g^2 dx \leq \frac{C}{\sqrt{t}}\left(\int |g_{\rm in}| dx\right)^2.
    \ee
    
\noindent
    {\bf \# Step (4): $L^1$-$L^\infty$ bound using Gagliardo-Nirenberg inequality.}
    Integrating \eqref{bd.energy} in time leads to
\be 
\int_{\frac{t}{2}}^t  \int g_x(t,x)^2 dx dt  \leq \int g\left( \frac{t}{2},x \right)^2 dx \leq \frac{C}{\sqrt{t}} \| g_{\rm in}\|_{L^1(\R)}^2
\ee
By the mean value theorem, there exists  $s\in \left[ \frac{t}{2},t\right]$, such that:
\be 
 \frac{t}{2}\| g_x(s)\|_{L^2(\R)}^2\leq  \frac{C}{\sqrt{t}} \| g_{\rm in}\|_{L^1(\R)}^2,
 \ee
 which may be rewritten as
 \be 
 \| g_x(s)\|_{L^2(\R)}\leq \frac{C}{t^{\frac{3}{4}}}  \| g_{\rm in}\|_{L^1(\R)}.
\ee 
Then, using the Gagliardo-Nirenberg inequality and the fact that the $L^\infty$-norm is decreasing, we have
\be
 \| g(t)\|_{L^\infty(\R)}^2 
 \leq  \| g(s)\|_{L^\infty(\R)}^2 
 \leq \| g(s)\|_{L^2(\R)}\| g_x(s)\|_{L^2(\R)}
 \leq \frac{C}{s^{\frac{1}{4}}}   \| g_{\rm in}\|_{L^1(\R)} \frac{C}{t^{\frac{3}{4}}}  \| g_{\rm in}\|_{L^1(\R)}
 \leq \frac{C}{t}   \| g_{\rm in}\|_{L^1(\R)}^2,
\ee
which finishes the proof.
\end{proof}

\subsection{The Local Model.}
\label{ss.loc.wp}

\subsubsection{Approximating $A^u$}
Before diving into the well-posedness theory, we define a Lipschitz approximation of $A^u$ and state some of the key properties of its profile function.  This approximation is used in the construction of solutions, and the properties of the profile function are key to establishing the nonnegativity of the shape defect function for the local problem. We define the approximation:
\be\label{def.A.eps}
    A^\varepsilon(s)
    = \left\{\begin{array}{ll}
0 & \text{ if } s\in [0,1-\varepsilon],\\
\frac{1}{\varepsilon}(s-(1-\varepsilon)) & \text{ if } s\in [1-\varepsilon,1].
\end{array}
\right.
\ee
We now state a technical lemma related to $A^\eps$.

\begin{lemma}[Approximation of $\eta^u$]
\label{l.approx_eta_loc}
For $\chi \geq 1$ and any $s\in [0,1]$, set
\be 
\eta^\e(s):=\chi(s-A^\e(s)).
\ee
For $\chi \in [0,1)$, we define
\be 
    \psi_*^{\e,+}
        = \frac{(\chi-2\e)+ \sqrt{(\chi-2\e)^2+4\e(1-\e)}}{2}
    \quad\text{ and}\quad
    m_\e
    = \frac{\psi_*^{\e,+}}{\e},
\ee
as well as $k_\e$, which is defined implicitly as the unique solution in $(-\infty,-\log(1-\e))$ to the equation
\be 
\label{d.keps}
e^{-k^\e}\eta^u\left(e^{k^\e}(1-\e)\right) = \psi_*^{\e,+}.
\ee
For $s\in [0,1]$, let
\be \label{e.eta_eps_e_k_eps}
\eta^\e(s):=\begin{cases}
e^{-k^\e}\eta^u\left(e^{k^\e}s\right)  &  \text{ if }s\in [0,1-\e], \\
    m_\e(1-s) & \text{ if }s\in (1-\e,1],
\end{cases}
\ee
where we recall $\eta^u$ from \Cref{l.l.eta_loc}.

Then, $\eta^\varepsilon$ is wave profile function associated to~\eqref{pde.general_tw} with the flux $A^\eps$ and the speed $c_*(\chi)$, which is defined in~\eqref{e.c_star_chi}.  
In addition, $Q^\e:=\eta^\e+\chi A^\e$ is concave, 
\be 
\label{bd.eta_e}
    \eta^\varepsilon\leq\eta^u
    \text{ for $\eps$ sufficiently small }
    \qquad\text{ and }\qquad
    \eta^\e \xrightarrow[\e \to 0]{C^1([0,1-\delta]) } \eta^u
\ee
for any $\delta\in(0,1)$. 
\end{lemma}

The proof of \Cref{l.approx_eta_loc} is given in \Cref{ss.eta_eps}.

\subsubsection{Approximate Solutions and Their Regularity}

For each $\e>0$, let $A^\e$ be given by \eqref{def.A.eps}. As $A^\eps$ is Lipschitz continuous, \Cref{thm.wellposedness_P} yields $u^\e\in V$, a solution to \eqref{pde.v} with $A=A^\e$, \textit{i.e.}
    \be 
\label{e.f.molA}
u_{t}^\e +\chi\left( A^\e(u^\e)\right)_x = u_{xx}^\e+u^\e-A^\e(u^\e),
\ee
starting from the initial data $u_{\rm in}$.

We prove that these solutions enjoy nice regularity properties that persist in the limit $\eps\to0$.
\begin{proposition}\label{p.u eps regularity}
    Let $u^\eps\in V_\gamma$ be the solution of~\eqref{e.f.molA} with measurable initial data $u_{\rm in}$ 
    taking values in $[0,1]$ and satisfying
    \be
        \int u_{\rm in}(x) e^{\sfrac{x}{\chi_{\lowvee}}} < \infty,
    \ee
    where, $\chi_{\lowvee} := \max\{1,\chi\}$.  Suppose that, in a distributional sense,
    \be
        w_{\rm in} := - u_{{\rm in},x} - \eta^u(u_{\rm in}) \geq 0.
    \ee
    Fix $T>0$.  Then the shape defect function $w^\eps := - u_x^\eps - \eta^\eps(u^\eps)$ is nonnegative.  Also, there is a constant $C$, depending only on $\chi$, such that
    \be
    \label{e.vhminH1}
    \|u^\e\|_{L^\infty([0,T]; L^2(\R, \zeta_\gamma dx))}^2
        + \|u^\e_x\|_{L^2([0,T]; L^2(\R,\zeta_\gamma dx))}^2
        \leq e^{CT}\|u_{\rm in}\|_{L^2(\R,\zeta_\gamma dx)}^2,
    \ee
    where $\zeta_\gamma$ is given by~\eqref{d.zeta}.  
    
    Moreover, define
\be
    \hat w^\e(t,z) = e^{\sfrac{z}{\chi_{\lowvee}}} \tilde w^\e(t,z),
\ee
where we recall that the tilde indicates a shift to the moving frame $(t,x) = (t,z + ct)$.  There exists a constant $C$, depending only on $\chi$ and $u_{\rm in}$, such that
    \be\label{e.w eps decay}
        \|\sqrt t \hat w^\eps\|_{L^\infty([0,\infty)\times \R)}
        \leq C,
    \ee    
and there exists, for any $T>0$, a constant $C_T>0$, depending only on $T$, $\chi$, and $u_{\rm in}$, such that
    \be\label{e.w_eps_dx}
        \left\|\hat w^\eps_z\right\|_{L^2([0,T]\times \R)}
        \leq C_T.
    \ee    
\end{proposition}

Before proving this proposition, let us observe that the bound on $\hat w^\eps$ implies that $\tilde u^\e$ is locally Lipschitz continuous in space (away from $t=0$).  On the other hand, by~\Cref{p.u eps regularity}, \Cref{l.w_eqn}, and \Cref{l.approx_eta_loc}, the time derivative satisfies
\be
    \tilde u^\eps_t
        = - \tilde w^\eps_z - R^\eps(\tilde u^\eps) \tilde w^\eps,
\ee
which implies that $ \tilde u^\eps_t \in L^\infty_{\rm{loc}}((0,+\infty);W_{\rm{loc}}^{-1,\infty}(\R))$. 
Through a standard interpolation argument, we then obtain that $\tilde u^\e\in W^{1-s,\infty}_{\rm{loc}}((0,+\infty);W_{\rm{loc}}^{2s-1,\infty}(\R))$.
Taking $s=\sfrac12$, we find that $\tilde u^\e$ is locally $C^{\sfrac12}$ in time (away from $t=0$). Of course, $u^\e$ inherits the same regularity properties as $\tilde u^\e$. In conclusion, we summarize this without a proof now:
\begin{corollary}\label{c.regularity}
    Under the assumptions of \Cref{p.u eps regularity}, $u^\e\in C^{\nicefrac{1}{2},1}_{\rm loc}((0,+\infty)\times \R)$ uniformly in $\e$, \textit{i.e.} for any $T>0$ and $L>0$, there is a constant $C_{T,L}$ such that for any $\e>0$,
    \be
    \label{bd.equicontinuity}
    |u^\eps(t,x) - u^\eps(s,y)|
    \leq C_{T,L}\left(|x-y| + \sqrt{|t-s|}\right),
    \ee
    for all $t,s \in \left[\sfrac{T}{2},T\right]$ and $x,y \in [-L,L]$.
\end{corollary}

We are now in a position to prove the crucial regularity estimates in \Cref{p.u eps regularity}.

\begin{proof}[Proof of \Cref{p.u eps regularity}]
Let us begin with~\eqref{e.vhminH1}. We may repeat exactly the same arguments that led us to obtain the energy estimate in the proof of \Cref{thm.wellposedness_P} by replacing \eqref{bd.A.Lip} with for all $s\in \R$,
    \be 
|A^\e(s)|\leq |s|.
    \ee Thus, we obtain 
    \be
\begin{split}
\frac{1}{2}\frac{d}{dt}\int \left(u^\e\right)^2 \zeta_\gamma + \frac12 \int \left(u^\e_x\right)^2\zeta_\gamma
&\leq  \left(\gamma^2 + \chi^2  + 1 \right)\int \left(u^\e\right)^2\zeta_\gamma.
\end{split}
\ee
Solving this differential inequality, we find that, for any $T>0$,
\be\label{e.vhminH1}
    \|u^\e\|_{L^\infty([0,T]; L^2(\R, \zeta_\gamma dx))}^2
        + \|u^\e_x\|_{L^2([0,T]; L^2(\R,\zeta_\gamma dx))}^2
        \leq e^{CT}\|u_{\rm in}\|_{L^2(\R,\zeta_\gamma dx)}^2.
\ee

Next we show the nonnegativity of $w^\eps$.  Using \Cref{l.approx_eta_loc}, we find, for $\eps$ sufficiently small,
\be 
    w^\e_{\rm in}
    :=
    -u_{{\rm in},x}-\eta^\e(u_{\rm in})
    \geq 
    -u_{{\rm in},x}-\eta^u(u_{\rm in})
    =
    w_{\rm in}
    \geq
    0.
\ee
Furthermore, $A^\e$ and $\eta^\e$ are Lipschitz continuous and, thus, $w^\e$ satisfies \eqref{e.w_eqn} according to \Cref{l.w_eqn}. Using the same arguments from \Cref{prop.W.pos} that lead us to establish the preservation of nonnegativity for the nonlocal shape defect function, we may conclude that,
on $[0,+\infty)\times \R$,
\be
\label{e.weps_pos}
    w^\e
    \geq 0.
\ee

Finally, we establish the decay of $w^\eps$. Without loss of generality, we proceed under the assumption that, up to $t=0$, $u$ is $C^1$ and, thus, $w_\eps$ is continuous. Indeed, were this not the case, parabolic regularity theory shows that $w^\eps(t)$ is continuous for all $t>0$.  The proof then proceeds using an approximation procedure as $t\to 0$.

When $\chi \geq 1$, computing as in the proof of \Cref{thm.W.decay} yields that $\hat{w}^\e$ satisfies
    \be 
\hat{w}^\e_t + \underbrace{\left( \chi (A^\e)'(\tilde u^\e)-\chi+\frac{1}{\chi}\right) }_{=:a(t,z)}\hat{w}^\e_z =  \hat{w}^\e_{zz}.
\ee
Recall that $(A^\e)'(s) =\e^{-1}\1_{[1- \e,1]}(s)$ so that $a_z = \chi (A^\e)''(\tilde u^\e)\tilde u^\e_z \leq 0$ (since $\tilde u^\e$ has values in $[0,1]$).  Hence \Cref{l.Linf_L1} applies.

In the $\chi \in [0,1)$ case, we find
\be 
\hat{w}^\e_t + \underbrace{ \chi (A^\e)'(\tilde u^\e) }_{=:a(t,z)}\hat{w}^\e_z + \underbrace{ \left((1-\chi)(A^\e)'(\tilde u^\e)-2(Q^\e)''(\tilde u^\e)\eta^\e(\tilde u^\e)-(Q^\e)''(\tilde{u^\e})\tilde{w}^\e\right) }_{=: b(t,z)}\hat{w}^\e=  \hat{w}^\e_{zz}.
\ee
Clearly $a$ is bounded and decreasing. One readily checks that $b\in L^\infty_{\rm{loc}}((0,+\infty)\times\R)$ by the same arguments as in the proof of \Cref{thm.W.decay}. Moreover, the same computations as in \eqref{exp.c} apply to $b$ here, and we find
\be 
b(t,z)=\left(R^\e(\tilde u^\e)-1\right)\left(1-(\eta^\e)'(\tilde u^\e)\right)-(Q^\e)''(\tilde{u}^\e)\tilde{u}^\e_z 
\ee
Using the concavity of $Q^\e$ established in \Cref{l.approx_eta_loc}, the same arguments as in the proof of \Cref{thm.W.decay} yield
\be 
b(t,z) \geq 0.
\ee
Thus, \Cref{l.Linf_L1} applies and yields
\be
\label{bd.aux.decay}
	\left\| \sqrt{t}\hat{w}^\e\right\|_{L^\infty(\R)}
    \leq C\left\| \hat{w}^\e_{\rm  in} \right\|_{L^1(\R)}.
\ee
To conclude, we must show that $\hat w^\eps_{\rm in}$ is uniformly bounded in $L^1$.  Observe that
\be
    \lim_{L\to\infty}\int_{-\infty}^L  \hat w^\eps_{\rm in}
    = \|\hat w^\eps_{\rm in}\|_{L^1(\R)}.
\ee
Then, we integrate by parts to find
\be
\begin{split}
    \int_{-\infty}^L \hat w^\eps_{\rm in}
    &= \int_{-\infty}^L e^{\sfrac{z}{\chi_{\lowvee}}} \left( - u_{{\rm in},x} - \eta^\eps(u_{\rm in})\right)
    = - e^{\frac{L}{\chi_{\lowvee}}} u_{\rm in}(L)
        + \int_{-\infty}^L e^{\sfrac{z}{\chi_{\lowvee}}} \left( u_{\rm in} - \eta^\eps(u_{\rm in})\right)
    \leq \int e^{\sfrac{z}{\chi_{\lowvee}}} u_{\rm in}(z).
\end{split}
\ee
The proof is, thus, complete.
\end{proof}

\subsubsection{Well-posedness and the Statement of the Comparison Principle}
We now establish our main well-posedness result, which we recall here for convenience:
\begin{recallthm}{thm.wellposedness_u}[Well-posedness for the Local Model]
\MainTheoremText
\end{recallthm}

We are now ready to prove the existence and uniqueness of solutions in the local case. 
The uniqueness in \Cref{thm.wellposedness_u} is due to the comparison principle.  Due to the singular nature of the equation, it is somewhat delicate to prove.  Let us state our comparison principle here.  We prove it after \Cref{thm.wellposedness_u}.

\begin{proposition}[Comparison principle for the local model]\label{thm.CP_u}
Consider sub- and supersolutions $\underline{u},\overline{u}\in V
$, respectively, to~\eqref{pde.u}; \textit{i.e.} 
\be
\begin{split}
&\underline{u}_t+\chi (A^u(\underline{u}))_x-\underline{u}_{xx}-(\underline{u}-A^u(\underline{u}))\leq 0 \quad\text{ and} \\
&\overline{u}_t+\chi (A^u(\overline{u}))_x-\overline{u}_{xx}-(\overline{u}-A^u(\overline{u}))\geq 0.
\end{split}
\ee 
In the case that  $\chi > 0$, suppose additionally that the shape defect associated with $\overline{u}$, \textit{i.e.} $\overline{w}=-\overline{u}_x-\eta^u(\overline u)$, is nonnegative. 
If $\underline{u}_{\rm in}\leq \overline{u}_{\rm in}$, then $\underline u \leq \overline u$.
\end{proposition}

We now establish the well-posedness of~\eqref{pde.u}.

\begin{proof}[Proof of \Cref{thm.wellposedness_u}]
    We construct the solution as the limit of solutions $u^\e$ to \eqref{e.f.molA}.

\medskip
\noindent
    {\bf \# Step (1): A compact approximating sequence $u^\e$.}
    From \Cref{p.u eps regularity}, we have the energy bound \eqref{e.vhminH1}  uniformly in $\e$. As a consequence of \eqref{e.f.molA}, we also have that, for any $T>0$, $u^{\e }_t \in L^2([0,T],H^{-1}(\R, \zeta_\gamma dx)))$. In particular, we have that the sequence $(u^{\e})_{\e }$ is bounded in 
 \be
    \left\{ g \in L^2([0,T],H^{1}_{\text{\rm loc}}(\R)): 
        g_t \in L^2([0,T],H^{-1}_{\text{\rm loc}}(\R))
    \right\}.
 \ee
 Through the standard diagonal application of the Aubin-Lions lemma on an increasing sequence of compacts, we may extract a converging subsequence as $\e\to 0$, which converges strongly to a limit denoted $u\in C([0,T],L^2_{\rm loc}(\R))$ and converges weakly to the same $u\in L^2([0,T],H^1_{\text{\rm loc}}(\R))$.

\medskip
\noindent
    {\bf \# Step (2): Regularity of $u$.}
When $\chi>0$, the equicontinuity bound \eqref{bd.equicontinuity} from \Cref{c.regularity} allows us to apply a classical Arzel\`{a}-Ascoli argument. Passing to a subsequence, we obtain a locally uniform limit that satisfies the same equicontinuity bound. Since this limit must coincide with the $L^2$-limit, we can conclude that
$u\in C^{\nicefrac{1}{2},1}_{\rm loc}((0,+\infty)\times \R)$. When $\chi = 0$, we note that $u^\eps$ solves
\be
    u_t^\eps - u^\eps_{xx}
    = u^\eps - A_\eps(u^\eps).
\ee
Since the terms on the right hand side are bounded, parabolic regularity theory implies that $u^\eps$, $u^\eps_t$, and $u^\eps_{xx}$ are locally bounded in $L^p$ uniformly in $\eps$ for any $p>1$. By Sobolev embedding, it follows that $u^\e \in C^{1-\nicefrac{1}{p},2-\nicefrac{1}{p}}_{\rm{loc}}([0,+\infty)\times \R)$ uniformly in $\eps$ and thus this regularity is preserved in the limit as $\eps\to 0$.

\medskip
\noindent
    {\bf \# Step (3): Nonnegativity of $w:=-u_x-\eta^u(u)$.}
Recall from \Cref{p.u eps regularity} that $w^\eps \geq 0$.  Define the measurable function
\be
    \overline  x(t) = \sup\{ x : u(t,x) = 1\}.
\ee
where we allow for the possibility that $\overline  x(t) = -\infty$.  We rule this out in the next step.

From $w^\e \geq 0$, we have $u^\e_x\leq 0$ and thus $u_x\leq 0$. Thus, when $x < \overline  x(t)$, we have that $u=1$ and $w = 0$, as $u$ has values in $[0,1]$.  Let us next consider the case where $x > \overline  x(t)$.  Fix any compact subset of $\{u < 1\} = \{x> \overline  x(t)\}$.  On this set $u^\eps$ is smooth (by parabolic regularity theory) and bounded by $1-\delta$ for some $\delta>0$ (by compactness of this set).  It, hence, converges in $C^1$ to $u$.  Additionally, $\eta^\eps$ converges to $\eta$ in $C^1$ on $[0,1-\delta]$ by \Cref{l.approx_eta_loc}.  Hence, the composition $\eta^\eps(u^\eps)$ converges to $\eta^u(u)$ in $C^1$ on this compact set  It follows that $w^\eps$ converges to $w$ on this compact set.  Using the arbitrariness of the compact set, we have that $w\geq 0$ on $\{u < 1\}$.

It follows that $w\geq 0$ for all $x\neq \overline  x(t)$.  Since $\{(t,x) : x = \overline  x(t)\}$ is a measure zero set, we conclude that $w \geq 0$ almost everywhere.

\medskip
\noindent
    {\bf \# Step (4): Regularity of the front $\overline  x$.}
    We show only the proof for $\chi \leq 1$.  The modifications for $\chi > 1$, however, are straightforward.
    
    Let us start by showing that $\overline x$ is finite. 
    Since $w_{\rm in}\geq 0$, we must have $\int u_{\rm in}(z) e^z dz<+\infty$.  It is then easy to show that $\int \tilde{u}(t,z) e^z dz<+\infty$ (see the proof of \Cref{l.expmoment}). From this, we find that $\overline x(t) < \infty$ for all $t$.  We, thus, focus on showing that $\overline x(t) > -\infty$.  Consider $\underline{u}$ the solution to the Fisher-KPP equation with initial data $u_{\rm in}$. Since $u^\eps$ is a supersolution to the Fisher-KPP equation for all $\eps$, we have $\underline{u}\leq  u^\eps$, which remains true in the limit. Thus, there is $b_t$, which can be lower bounded by the analogous level set of $\underline{u}$ and in particular $b_t>-\infty$, such that
\be
    u(t,b_t) = ((1-\chi) + 1) e^{-1}=u_{\rm tw}(1),
\ee
where we recall that $u_{\rm tw}$ is the traveling wave given in \Cref{table.tw}.
Then, since $w\geq 0$, we thus have by steepness, for all $x < b_t$,
\be
    u(t,x)
    \geq
    u_{\rm tw}(1 - (b_t-x)).
\ee
It follows that $u(t,x) = 1$ whenever $x \leq b_t - 1$. Thus, $\overline  x(t) \geq b_t - 1$ and $\overline  x(t)$ is finite.

We now show the H\"older regularity of $\overline  x$. 
Fix any $t>0$ and $h \neq 0$ such that $t+h > 0$.  Let us suppose, without loss of generality, that $\overline x(t) < \overline x(t+h)$.  Denote
\be
    \Delta_h \overline x = |\overline x(t+h) - \overline x(t)|.
\ee
We then have that
\be\label{e.c012202}
\begin{split}
    |u(t,\overline x(t))
        - u(t, \overline x(t+h))|
    &= 1 - u(t,\overline x(t+h))
    \geq 1 - u_{\rm tw}\left( \Delta_h \overline x\right)
    \\&
    = 1 - \left((1-\chi) \Delta_h \overline x + 1\right) e^{-\Delta_h \overline x},
\end{split}
\ee
where we used again that $u(t,\cdot)$ is steeper than $u_{\rm tw}$.

Let us now fix $\chi >0$.  Using the regularity in \Cref{c.regularity} and the definition of $\overline  x$, we find that
\be\label{e.c012201}
    |u(t,\overline x(t))
        - u(t, \overline x(t+h))|
    = |1 - u(t, \overline x(t+h))|
    = |u(t+h, \overline x(t+h)) - u(t,\overline x(t+h))|
    \leq C_T \sqrt h.
\ee
Combining these, we see that
\be\label{e.c010901}
    1 - \left((1-\chi) \Delta_h \overline x + 1\right) e^{-\Delta_h \overline x}
        \leq C_T \sqrt h.
\ee
Setting $g(p):= 1 - \left((1-\chi) p + 1\right) e^{-p}$ and observing that $g$ is increasing for $p\geq 0$, we obtain that
\be 
\Delta_h \overline x
    \leq g^{-1}\left(C_T \sqrt h\right).
\ee
Finally, using that $g^{-1}(0)=0$ and $\left(g^{-1}\right)'$ is locally bounded on $[0,1)$ when $\chi \in (0,1)$, we can conclude that $\overline  x \in C^{\sfrac12}_{\rm loc}((0,+\infty))$, which completes the proof of regularity of $\overline  x$ when $\chi >0$.

When $\chi = 0$, we replace~\eqref{e.c012201} by
\be
    |u(t,\overline x(t))
        - u(t, \overline x(t+h))|
    = |1 - u(t, \overline x(t+h))|
    = |u(t+h, \overline x(t+h)) - u(t,\overline x(t+h))|
    \leq C_T h^{1 - 2\delta}
\ee
for any $\delta>0$.  Then, observing that $\nicefrac{g^{-1}(s)}{\sqrt{s}}$ is locally bounded on $[0,1)$ when $\chi =0$, we obtain that $\Delta_h \overline x=\mathcal{O}\left(h^{\nicefrac{1}{2}-\delta}\right)$ and from there, we can conclude that $\overline  x \in C^{\sfrac12-\delta}_{\rm loc}((0,+\infty))$ for any $\delta >0$.

\medskip
\noindent
    {\bf \# Step (5): Identification of the limit $u$ as a weak solution to \eqref{pde.u}.}
 Next,  we need to show that $u$ is a solution to \eqref{pde.u}. Since the linear terms converge in the sense of distributions, it remains only to show that
 \be
    A^\e(u^\e)\rightharpoonup A(u)
    \quad\text{ in } L^2([0,T], L^2_\text{\rm loc}(\R)).
\ee
 Clearly, $A^\e(u^\e)$ is bounded in $L^2([0,T], L^2_\text{\rm loc}(\R))$.  After passing to a subsequence, we thus find a weak limit $g\in L^2([0,T], L^2_\text{\rm loc}(\R))$.  We immediately see that, weakly,
 \be\label{pde.u g}
    u_t + \chi g_x
    = u_{xx} + u - g.
 \ee
 Our goal is to show that $g = A(u)$.

As noted in Step (2), $u^\e$ converges locally uniformly to $u$. Hence, $A^\eps(u^\eps)$ converges to $0 = A(u)$ on compact subsets of $\{u < 1\}$, as on such sets $A^\eps(u^\eps)$ identically vanishes for $\eps$ sufficiently small. Thus, $g = A(u)$ when $x > \overline  x(t)$.

Next, we consider the case $x < \overline  x(t)$.  Due to the regularity of $\overline  x(t)$ established in Step (4), $\{x < \overline  x(t)\}$ is an open set in $(t,x)$.  We, thus, see that $u\equiv 1$ and, hence, $u_t = u_{xx} = 0$ on $\{x < \overline  x(t)\}$.  It immediately follows from~\eqref{pde.u g} that 
\be\label{e.g_x}
    \chi g_x
    = u - g
    \qquad \text{ in } (-\infty, \overline  x(t)).
\ee
When $\chi = 0$, we deduce that $g =1= A(u)$ immediately from this.  Otherwise,~\eqref{e.g_x} yields
\be
    g(t,x) = 1 - \theta(t) e^{- \sfrac{x}\chi}
\ee
for some function $\theta(t)$ depending only on $t$.  Since $g$ takes values in $[0,1]$ (recall it is the limit of $A^\eps(u^\eps)$)), this may only occur if $\theta(t) = 0$, otherwise $g$ is unbounded as $x\to-\infty$.  We again deduce that $g = 1 = A(u)$.

From the above, we have shown that $g = A(u)$ on $\{x \neq \overline  x(t)\}$.  Hence, $g = A(u)$ almost everywhere.  This concludes the proof that $u$ solves~\eqref{pde.u}.

\medskip
\noindent
    {\bf \# Step (6): Regularity $w\in L^2_{\rm loc}([0,+\infty), H^1_{\rm loc}(\R))$.} Using \eqref{e.w_eps_dx}, we have that $w^\eps$ is uniformly bounded in $L^2_{\rm loc}([0,+\infty), H^1_{\rm loc}(\R))$ and thus admits a weakly converging subsequence. Furthermore $\eta^\e(u^\e)$ converges pointwise to $\eta(u)$ on $\{x \neq \overline  x(t)\}$ and thus almost everywhere. Since $\eta^\e,\eta$ are bounded, we therefore have that $\eta^\e(u^\e)$ converges to $\eta(u)$ in $L^2_{\rm loc}([0,+\infty)\times \R)$ by the Dominated Convergence Theorem. Therefore, $w^\e=-u^\e_x-\eta^\e(u^\e)$ converges in the sense of distributions to $w=-u_x-\eta(u)$ and by uniqueness of the limits, $w\in L^2_{\rm loc}([0,+\infty), H^1_{\rm loc}(\R))$.

\medskip
\noindent
    {\bf \# Step (7): Uniqueness.} This follows directly from \Cref{thm.CP_u}.
\end{proof}

\subsubsection{Proving the comparison principle}

We now establish the comparison principle.

\begin{proof}[Proof of \Cref{thm.CP_u}]
In the $\chi=0$ case, the irregular advection term disappears, and we can readily obtain an energy estimate like~\eqref{bd.energy_est_unique} that we obtained in the nonlocal case. We omit the details.

For the remainder of the proof, we, thus, only consider the $\chi >0$ case.  Additionally, let us change notation with $u_1$ denoting the subsolution $\underline u$ and $u_2$ denoting the supersolution $\overline u$.  This allows us to decorate the functions with tildes more easily and to compactify statements with the subscript $i$.

In the moving frame $(t,z)=(x-ct)$ with $c=c_*(\chi)$ as in~\eqref{e.c_star_chi} and the notation $\tilde{u}_i(t,z)=u_i(t,x-ct)$ and $\tilde{w}_i(t,z)=w_i(t,x-ct)$, following \Cref{l.w_eqn}, we have for $i\in \{1,2\}$ 
    \be 
    \label{pde.uw}
\tilde{u}_{i,t} = -\tilde{w}_{i,z} - R(\tilde{u}_{i})\tilde{w}_{i},
    \ee
    where $R=c-Q'$ and $Q=\eta+\chi A$.
    
For the sake of readability, we introduce the following notation:
\be 
\label{def.Delta_not}
	\Delta f := f_1-f_2.
\ee
We now compute:
\be 
\label{energy.w}
\begin{split}
\frac{d}{dt} \frac{1}{2 }\int_\R  &(\Delta \tilde u)_+^2 \zeta_\gamma
    =
    \int_\R (\Delta \tilde u)_+ \Delta \left( -\tilde{w}_z -R\tilde{w} \right)\zeta_\gamma
    \\& 
    =
    \int_\R  \1_{\{\Delta \tilde u>0\}} (\Delta \tilde u)_z \Delta \tilde{w}\zeta_\gamma
    +\int_\R  (\Delta \tilde u)_+  \Delta \tilde{w}(\zeta_\gamma)_z 
    - \int_\R  (\Delta \tilde u)_+ \Delta \left( R\tilde{w} \right)\zeta_\gamma
    \\&
    = -  \int_\R \1_{\{\Delta \tilde u>0\}}(\Delta \tilde w)^2\zeta_\gamma
    - \int_\R \1_{\{\Delta \tilde u>0\}} \Delta \eta  \Delta \tilde{w} \zeta_\gamma
    + \int_\R (\Delta \tilde{u})_+  \Delta \tilde{w} (\zeta_\gamma)_z
    \\&\qquad
    - \int_\R (\Delta \tilde u)_+ \left(R_1\Delta  \tilde{w}+\tilde{w}_2 \Delta R \right)\zeta_\gamma
    \\&
    = -  \int_\R \1_{\{\Delta \tilde u>0\}}(\Delta \tilde w)^2 \zeta_\gamma
    - \int_\R \1_{\{\Delta \tilde u>0\}} \Delta \left( Q-\chi A\right) \Delta \tilde{w}\zeta_\gamma
    + \int_\R (\Delta \tilde{u})_+  \Delta \tilde{w} (\zeta_\gamma)_z
    \\&\qquad
    - \int_\R  R_1 (\Delta \tilde u)_+ \Delta  \tilde{w}\zeta_\gamma
    - \int_\R \tilde{w}_2 (\Delta \tilde u)_+ \Delta R\zeta_\gamma
    \\&
    =
    -  \int_\R \1_{\{\Delta \tilde u>0\}}(\Delta \tilde w)^2\zeta_\gamma
    - \int_\R  \1_{\{\Delta \tilde u>0\}}\Delta Q  \Delta \tilde{w} \zeta_\gamma
    +\chi \int_\R \1_{\{\Delta \tilde u>0\}}  \Delta A \Delta \tilde{w}\zeta_\gamma
    \\&\qquad
    + \int_\R (\Delta \tilde{u})_+  \Delta \tilde{w}(\zeta_\gamma)_z
    - \int_\R  R_1(\Delta \tilde u)_+ \Delta  \tilde{w} \zeta_\gamma
    - \int_\R  \tilde{w}_2 (\Delta \tilde u)_+ \Delta R\zeta_\gamma.
\end{split}
\ee
The first term has a good sign.  The second, fourth, and fifth terms are bounded by 
\be
    (\|Q'\|_\infty + \gamma + \|R\|_\infty) |\Delta \tilde u| |\Delta \tilde w| \zeta_\gamma
\ee
due to \Cref{l.l.eta_loc} and the fact that $|(\zeta_\gamma)_z|\leq \gamma\zeta_\gamma$.  After applying Young's inequality, we get an upper bound of these three terms in the form:
\be
    \frac{1}{2}\left( \Delta \tilde{w} \right)^2 \zeta_\gamma +\frac{\left(\|Q'\|_\infty+\gamma +\|R\|_\infty\right)^2}{2}  \left( \Delta \tilde u \right)^2 \zeta_\gamma. 
\ee
The third term in the last line of~\eqref{energy.w} has a good sign.  Indeed,
\be 
	\Delta  A \Delta \tilde w=(A(\tilde u_1)-A(\tilde u_2))(\tilde w_1-\tilde w_2) = \left\{
	\begin{array}{ll}
	0 & \text{ if } \tilde u_1=\tilde u_2 = 1, \\
	-\tilde w_1 & \text{ if } \tilde u_1<1, \tilde u_2 = 1, \\
	-\tilde w_2 & \text{ if } \tilde u_1=1,\tilde u_2 < 1, \\
	0 & \text{ if } \tilde u_1,\tilde u_2 < 1.
	\end{array}
	\right.
\ee
Hence
\be
    \1_{\{\Delta \tilde u>0\}} \Delta   A \Delta \tilde w
    = -\1_{\{\Delta \tilde u>0\}} \tilde  w_2
    \leq 0.
\ee
The final term in~\eqref{energy.w} has a good sign; indeed, $R$ is increasing and $\tilde w_2$ is nonnegative by assumption.

Putting together all estimates above, we deduce that
\be
\begin{split}
\frac{d}{dt} \frac{1}{2 }\int_\R  (\Delta \tilde u)_+^2 \zeta_\gamma
\leq & -  \frac{1}{2}\int_\R \left( \Delta \tilde{w} \right)^2 \zeta_\gamma +\frac{\left(\|Q'\|_\infty+\gamma+\|R\|_\infty\right)^2}{2}  \int_\R  \left( \Delta \tilde u \right)^2 \zeta_\gamma.
\end{split}
\ee
Hence, using Gr\"onwall's inequality, we obtain 
\be 
\int_\R  \left( \Delta \tilde u(t) \right)_+^2 \zeta_\gamma \leq e^{Ct}\int_\R  \left( \Delta \tilde u_{\rm in} \right)_+^2 \zeta_\gamma ,
\ee
with $C=\frac{\left(\gamma + \|Q'\|_\infty+\|R\|_\infty \right)^2}{2}$.  Since $(\Delta u_{\rm in})_+ =0$, we see that $(\Delta \tilde u(t))_+ = 0$ for all $t$.  The proof is complete.
\end{proof}

\subsubsection{The Approximating Flux $A^\e$. Proof of \Cref{l.approx_eta_loc}}
\label{ss.eta_eps}

Here we investigate the approximation $A^\eps$, defined in~\eqref{def.A.eps}.  In particular, we prove \Cref{l.approx_eta_loc}, which establishes several key properties of $\eta^\eps$ used in the well-posedness theory of $u$.

In the $\chi \geq 1$ case, this lemma is a direct consequence of the fact that
\be 
\eta^\e(s)=\chi \left( s-A^\e(s)\right)
\ee
is the traveling wave profile associated with the wave speed $c= \chi+\frac{1}{\chi}$.

In the $\chi\in [0,1)$ case, we may rewrite \eqref{ode.eta} to obtain
\be 
\label{ode.eta_eps}
\left(\eta^\e\right)'(s)=2-\frac{s-A^\e(s)}{\eta^\e(s)}-\chi \left(A^\e\right)'(s).
\ee
We can observe that $\eta^\eps$ and $\eta^u$ satisfy the same ODE on $[0,1-\e]$, as on that interval $A^\e\equiv 0 \equiv A^u$. Therefore, a quick computation shows that for $\eta^\e$ given by \eqref{e.eta_eps_e_k_eps} solves \eqref{ode.eta_eps} on $[0,1-\e]$. 

On $[1-\e,1]$, \eqref{ode.eta_eps} may be rewritten as
\be 
\label{ode.eta_eps_2} 
    (\eta^\e)'(s)
        =2-\frac{(1-\e)(1-s)}{\e\eta^\e(s)}-\frac{\chi}{\e}.
\ee
One readily checks that
\be\label{e.c011601}
    \eta^\eps(s) = m_\eps(1-s)
\ee
is a solution to~\eqref{ode.eta_eps_2}.

The constant $k^\e$ is then selected in such a way that $\eta^\e$ is continuous at $s=1-\e$, which corresponds to
\be 
\eta^\e((1-\e)^-)=\eta^\e((1-\e)^+),
\ee
which is equivalent to \eqref{d.keps}.

We now verify the existence of $k^\eps$. 
Using \eqref{eta.exp_u} from \Cref{l.l.eta_loc}, leads to
\be 
\label{e.keps_W1}
(1-\e)\left(1+\frac{1}{W_{-1}(-\kappa e^{k^\e}(1-\e))}\right)
=  \psi_*^{\e,+} .
\ee
This equation admits a unique solution. Indeed, on the one hand, let us set 
\be 
g^\e(p):=(1-\e)\left(1+\frac{1}{W_{-1}(-\kappa p)}\right),
\ee
and observe that $\kappa= \frac{1}{1-\chi}e^{-\frac{1}{1-\chi}}\leq e^{-1}$. Therefore, we have that $W_{-1}(-\kappa p)$ is well defined for $p\in (0,1)$ and $g^\e$ is strictly decreasing on $(0,1)$. But, we also have that $g^\e(0^+)=1-\e$ and $g^\e(1^-)=(1-\e)\chi$, as a consequence of the known boundary behavior of $\frac{\eta(s)}{s}$. 

On the other hand, $\psi_*^{\e,+}$ is the positive root of the quadratic
\be 
h^\e(\mu)=\mu^2-(\chi-2\e)\mu -\e(1-\e).
\ee
But $h^\e(1-\e)>0$ and $h^\e((1-\e)\chi)<0$ for $\e>0$ small enough. Thus, we have that
\be
g^\e(1^-)=(1-\e)\chi < \psi_*^{\e,+} <1-\e=g^\e(0^+),
\ee
and, by the intermediate value theorem, we can find a unique $p_\ast^\e\in (0,1)$ such that $g^\e(p_\ast^\e)=\psi_*^{\e,+}$, which leads to a unique 
\be 
k^\e =\log\left( \frac{p_\ast^\e}{1-\e}\right) \in \left( -\infty, -\log(1-\e)\right) 
\ee
solving \eqref{e.keps_W1}.

Let us furthermore observe that, since $g^\e(1^-)=(1-\e)\chi$ and by a straightforward computation
\be 
\lim_{\e \to 0} \frac{\psi^{\e,+}_\ast}{1-\e} = \chi.
\ee
Thus, by inverting $\nicefrac{g^\e }{1-\e}$ and using the continuity of the inverse, we must have that
\be 
\lim_{\e \to 0} p ^\e_\ast=1.
\ee
Consequently,
\be 
\lim_{\e \to 0} k ^\e=0.
\ee
From this we easily deduce the limit in~\eqref{bd.eta_e}.

Now, let us show that $Q^\e$ is concave on $[0,1]$. On $[0,1-\e]$, we have $Q^\e=\eta^\e$ and $\eta^\e$ is clearly concave on that interval, since $\eta^u$ is concave on $[0,1)$. On $[1-\e,1]$, $Q^\e=\eta^\e+\chi A^\e$ is concave as the sum of two linear functions. Furthermore, $Q^\e$ satisfies
\be 
\label{ode.Q_eps}
\left(Q^\e\right)'(s)=2-\frac{s-A^\e(s)}{\eta^\e(s)},
\ee
and therefore $Q^\e\in C^1((0,1))$, which combined with its concavity on $[0,1-\e]$ and on $[1-\e,1]$ shows that $Q^\e$ is concave on $[0,1]$.

Finally, we must show that $\eta^\e(s)\leq \eta^u(s)$ for $s\in [0,1]$. 
The first step is to establish that 
\be\label{e.c011602}
    \eta^\e(1-\e)\leq \eta^u(1-\e).
\ee
This yields $\eta^\e(s)\leq \eta^u(s)$ on $[0,1-\e]$ because they are both solutions to the same ODE on this domain.

Let us consider~\eqref{e.c011602} when $\chi \in (0,1)$ first.  
By a standard expansion of $\eta^\e(1-\e)$, we find that, for $\chi\in (0,1)$,
\be
\eta^\e(1-\e)
=\psi^{\eps,+}_*
=\chi
    +\e \left(\frac{1}{\chi}-2 \right)
    -\e^2\left(\frac{(\chi-1)^2}{\chi^3}\right) +O(\e^3).
\label{expansion.theta_eps}
\ee
On the other hand, using from \Cref{l.l.eta_loc} the values for $(\eta^u)'(1^-)$ and $(\eta^u)''(1^-)$ given by \eqref{value.eta_u'} and \eqref{value.eta_u''}, we obtain through a Taylor expansion, for $\chi\in (0,1)$,
\be\begin{split}
\eta^u(1-\e)
&
= \eta^u(1^-)-\e \left(\eta^u\right)'(1^-)+\frac{\e^2}{2} \left(\eta^u\right)''(1^-)+O(\e^3)
\\&
=\chi
    + \e \left(\frac{1}{\chi}-2\right)
    -\e^2\left(\frac{(\chi - 1)^2}{2\chi^3}\right) +O(\e^3).
\label{expansion.theta_u}
\end{split}
\ee
Comparing~\eqref{expansion.theta_eps} and~\eqref{expansion.theta_u}, we  deduce~\eqref{e.c011602} for $\e>0$ small enough.

Let us now consider the $\chi = 0$ case.  It is easy to check that
\be 
\eta^\e(1-\e) 
=\psi_*^{\eps,+}
= \sqrt \eps - \eps.
\ee
We claim that
\be
\eta^u(1-\e)
=\sqrt{2\e}+o(\sqrt{\e}).
\label{expansion.theta_u.chi0},
\ee
from which~\eqref{e.c011602} follows. Define $\tau = (\eta^u)^2$.  By~\eqref{ode.eta}, we have, for $s<1$,
\be
    2 \sqrt{\tau(s)}
        - \frac{1}{2} \tau'(s) = s.
\ee
By \Cref{l.l.eta_loc}, $\eta(1^-)=0$, whence $\tau(1^-) = 0$.  Hence, $\tau'(1^-) = 2$.  We deduce that
\be
    \tau(1-\eps)
    = 2\eps + o(\eps).
\ee
from which~\eqref{expansion.theta_u.chi0} follows.  Thus,~\eqref{e.c011602} holds when $\chi = 0$.  This completes the proof that, for all $s\in[0,1-\eps]$,
\be\label{bd.eta_3}
    \eta^\eps(s)
    \leq \eta^u(s).
\ee

Then, on $[1-\e,1]$, we start from:
\be 
    Q^\e(1-\e)
    =\eta^\e(1-\e)
    \leq \eta^u(1-\e)
    =Q^u(1-\e),
\ee
as well as $Q^\e(1)=\chi=Q^u(1)$. On the other hand, $Q^\e$ is linear on $[1-\e,1]$ and $Q^u$ is concave.  It follows that, for $s\in [1-\e,1)$,
\be 
Q^\e(s)\leq Q^u(s).
\ee
from which deduce that
\be 
    \eta^\e(s)
    =Q^\e(s)-\chi A^\e(s)
    \leq Q^\e(s)
    \leq Q^u(s)
    =\eta^u(s).
\ee
This, along with~\eqref{bd.eta_3} and the fact $\eta^u(0)=\eta^\e(0)=0$, completes the proof that, for $s \in [0,1]$,
\be 
\eta^\e(s)\leq \eta^u(s).
\ee

\section{An Exponential Moment and A Preliminary Upper Bound on Front Location}\label{s.prelim}

In this section, we introduce a weighted exponential moment in the speed $c$ moving frame and
show that it is constant in time when $\chi \geq 1$ and nonincreasing when $\chi < 1$.   As a consequence, we obtain a simple upper
bound on the position of the front $\overline x(t)$ for both the nonlocal and local
models. This bound is sharp in the pushed ($\chi>1$) case.  While it is not sharp in the pushmi-pullyu ($\chi = 1$) and pulled ($\chi<1$) cases, it will play a central technical role in the later analysis. In fact, the corresponding exponential moments, which we denote generically by $I(t)$ (see \eqref{defi.It}), will reappear throughout the paper: many of our proofs obtain bounds on the front location by first deriving estimates of $I(t)$ and then translating these into control of $\overline x(t)$. \Cref{c.easyupperbound} below illustrates this mechanism in a particularly simple form.

Shifting to the moving $ct$ moving frame (recall \Cref{ss.notation}) with $c=c_*(\chi)$ defined by~\eqref{e.c_star_chi}, and using \eqref{pde.rho}, \eqref{pde.P}, and \eqref{pde.u}, we find
\be 
\label{pde.mv}
\begin{aligned}
    &\tilde\rho_t -c\tilde{\rho}_z+ \chi (\alpha(\tilde P) \tilde \rho )_z = \tilde\rho_{zz} + \tilde\rho(1 - \alpha(\tilde P)),
    \\
    &\tilde P_t -c\tilde{P}_z+ \chi \left(A^P(\tilde P) \right )_z = \tilde P_{zz} + \tilde{P}- A^P(\tilde P),
    \quad\text{ and}
    \\
    &\tilde u_t -c\tilde{u}_z+ \chi (\alpha(\tilde u) \tilde u )_z = \tilde u_{zz} + \tilde u(1 - \alpha(\tilde u)).
\end{aligned}
\ee
We define the exponential moments
\be 
\label{defi.It}
I_\rho(t):=\int \tilde{\rho}(t,z)e^{\nicefrac{z}{\chi_{\lowvee}}}dz, \quad I_P(t):=\int \tilde{P}(t,z)e^{\nicefrac{z}{\chi_{\lowvee}}}dz, \quad \text{and} \quad I_u(t):=\int \tilde{u}(t,z)e^{\nicefrac{z}{\chi_{\lowvee}}}dz,
\ee
where $
\chi_{\lowvee}: = \max\{ 1, \chi \}$. 
Through a simple integration by parts and noting that $\tilde{P}_z=-\tilde{\rho}$, we have for all $t\geq 0$,
\be\label{e.I_rho_I_P}
I_\rho(t)=\frac{1}{\chi_{\lowvee}}I_P(t).
\ee

\begin{proposition}
\label{l.expmoment}
Consider $\tilde{\rho},\tilde{P}$ and $\tilde{u}$ solutions to \eqref{pde.mv} and $\chi_{\lowvee}=\max\{chi,1\}$. Suppose that
\be 
    I_\rho(0)=\int \rho_{\rm in}(z)e^{\nicefrac{z}{\chi_{\lowvee}}}dz < +\infty
    \quad 
    \text{ and }
    \quad 
    I_u(0)= 
    \int u_{\rm in}(z)e^{\nicefrac{z}{\chi_{\lowvee}}} < +\infty.
\ee
Then, the following dichotomy holds. If $\chi\geq 1$, then
\be 
I_\rho(t) = I_\rho(0), \quad I_P(t) = I_P(0), \quad \text{ and }\quad I_u(t) = I_u(0),
\ee
whereas if $\chi \in [0,1)$, then
\be 
I_\rho(t) \leq I_\rho(0), \quad I_P(t) \leq I_P(0), \quad \text{ and }\quad I_u(t) \leq I_u(0).
\ee
\end{proposition}

\begin{proof}
Differentiating $I_\rho$ in time and integrating by parts leads to 
    \be
\begin{split}
    \dot{I_\rho}
    &=\int e^{\nicefrac{z}{\chi_{\lowvee}}} \left( \tilde {\rho}_{zz} + c \tilde \rho_z -\left( \chi \alpha \tilde \rho\right)_z+(1-\alpha)\tilde \rho \right) dz\\
    &=\frac{I_\rho}{\chi_{\lowvee}^2}-\frac{cI_\rho}{\chi_{\lowvee}} + I_\rho
        - \left(1 - \frac{\chi}{\chi_{\lowvee}}\right)\int e^{\nicefrac{z}{\chi_{\lowvee}}} \alpha \tilde \rho 
    = - \left(1 - \frac{\chi}{\chi_{\lowvee}}\right)\int e^{\nicefrac{z}{\chi_{\lowvee}}}\alpha \tilde \rho dz ,
\end{split}
\ee
where we used that $\chi_{\lowvee}^{-2} - c \chi_{\lowvee}^{-1} + 1 = 0$ with $c=c_*(\chi)$ as in~\eqref{e.c_star_chi}. Then, the conclusion follows by noting that for $\chi\geq 1$, we have $\sfrac{\chi}{\chi_{\lowvee}}=1$ leading to $\dot{I}_\rho=0$, and for $\chi \in [0,1)$, we have $\sfrac{\chi}{\chi_{\lowvee}}<1$ leading to $\dot{I}_\rho\leq 0$.

The same computations apply to $I_u$.  Moreover, the result for $I_P$ follows from~\eqref{e.I_rho_I_P}.
\end{proof}

Having established the conservation (resp. the monotonicity) of the exponential moment, we now
show how bounds on this quantity can be converted into a quantitative control
on the position of the front $\overline x(t)$ (recall~\eqref{defi.xt.loc} and \eqref{defi.xt.nonloc}), starting with a particularly simple but robust upper bound.

\begin{corollary}\label{c.easyupperbound}
    Suppose the assumptions of \Cref{l.expmoment} hold.  In the nonlocal case
    \be
        \overline x(t)
        \leq
        ct + \chi_{\lowvee}\log I_\rho(0).
    \ee
    In the local case,
    \be
        \overline x(t)
        \leq
        ct + \chi_{\lowvee} \log \frac{I_u(0)}{\chi_{\lowvee}}.
    \ee
\end{corollary}
\begin{proof}
    For the nonlocal model, let us recall that
    \be\label{e.P_bar_x_in_lemma}
    1
    =P(t,\overline x(t))
    =\int_{\overline x(t) - ct}^{+\infty}\tilde\rho(t,z)dz.
\ee
Hence,  we have 
\be 
    I_\rho(0)
    \geq I_\rho(t)
    = \int e^{\frac{z}{\chi_{\lowvee}}}\tilde{\rho}(t,z)dz
    \geq e^{\frac{\overline x(t) - ct}{\chi_{\lowvee}}} 
        \int_{\overline x(t) - ct}^{+\infty} {\tilde \rho}(t,z)dz
    =e^{\frac{\overline x(t) - ct}{\chi_{\lowvee}}},
\ee 
which clearly implies the conclusion.

For the local model, we do not have~\eqref{e.P_bar_x_in_lemma}, of course, but instead we may use that, for $z\leq \overline x(t) -ct$ we have $\tilde u(t,z)=1$.  Thus,
\be
    I_u(0)
    \geq I_u(t)
    \geq \int_{-\infty}^{\overline x(t)-ct} e^{\frac{z}{\chi_{\lowvee}}} dz
    = \chi_{\lowvee} e^{\frac{\overline x(t) - ct}{\chi_{\lowvee}}},
 \ee
 from which the result follows.
\end{proof}

\section{The pushed case: $\chi > 1$}\label{s.pushed}

The upper bound on $\overline x(t)$ in \Cref{thm-asymptotics} 
when $\chi >1$ was established in \Cref{c.easyupperbound}.  Accordingly, in order to prove \Cref{thm-asymptotics} when $\chi >1$, it suffices to prove the lower bound.  We begin with two technical lemmas.

\begin{lemma}
\label{c.rhoLinfty}
Fix $\chi \geq 0$, $c=c_*(\chi)$ as in~\eqref{e.c_star_chi}.
Assume that $\hat{W}_{\rm in}(x):=e^{\nicefrac{x}{\chi_{\lowvee}}}W_{\rm in}(x)\in L^\infty(\R)\cap L^1(\R)$ with $\chi_{\lowvee}=\max\{\chi,1\}$.
Then, there exists a constant $C>0$ such that the solution $\rho$ to \eqref{pde.P} satisfies for $t\geq 0$,
\be 
\label{bd.rhoLinfty}
\sup_{x \leq \overline{x}(t)}\rho(t,x) \leq \| \rho_{\rm in} \|_{L^\infty(\R)} + \frac{1}{c-\chi} + \sup_{s\in [0,t]} \left\{e^{-\frac{\overline{x}(s)-cs}{\chi_{\lowvee}}}\min\left\{ \left\|\hat{W}_{\rm in}\right\|_{L^\infty(\R)},  \frac{C}{\sqrt{s}}\left\|\hat{W}_{\rm in}\right\|_{L^1(\R)}\right\}\right\}.
\ee
\end{lemma}

Although stated for $\chi\geq 0$, we will only use this lemma here for $\chi>1$ and later for the pushmi-pullyu case $\chi=1$ in \Cref{s.pp_lower}.

\begin{lemma}
\label{l.upperboundgpush}
Suppose that $g:[0,+\infty)\to \R$ is locally bounded and satisfies the inequality
\be 
    g(t)
    \leq C\left(1+\sup_{s\in [0,t]} \frac{g(s)}{\sqrt{1+s}}\right),
\ee
for some positive constant $C$, then 
\be 
\sup_{t \geq 0} g(t) <+\infty.
\ee
\end{lemma}

We postpone the proofs of these lemmas to the end of the section and now turn to the lower bound on $\overline x(t)$.

\begin{proposition}[Lower bound]\label{l.pushed_lower}
	Let $\chi >1$ and $c=c_*(\chi)$ as in~\eqref{e.c_star_chi}.  Suppose that $P$ (resp. $u$) solves \eqref{pde.P} (resp. \eqref{pde.u}) with the initial datum $P_{\rm in}$ (resp. $u_{\rm in}$).  Assume further that 
    \begin{itemize}
        \item (nonlocal case) $W_{\rm in} \geq 0$ and $\hat{W}_{\rm in}(x):=e^{\sfrac{x}{\chi}}{W}_{\rm in}(x) \in  L^\infty(\R)\cap L^1(\R) $;
        \item (local case) $w_{\rm in}\geq 0$.
    \end{itemize}
	Then, there is $C$, depending only on the initial data and $\chi$, such that, for $t\geq 0$,
	\be
		\overline{x}(t) \geq ct - C.
	\ee
\end{proposition}

\begin{proof}
We show the proof for $P$ and omit the proof for $u$ as it is analogous and simpler. The argument relies on a refinement of the exponential-moment method introduced in \Cref{s.prelim}. 
We use $I(t)=I_P(t)$ as defined in \eqref{defi.It}.
Recall that the tilde denotes a shift to the moving frame: $\tilde{P}(t,z)=P(t,z+ct)$. From \Cref{l.expmoment}, we know that for $t\geq 0$
\be
\label{e.2511115}
I(t)=I(0).
\ee
We bound this integral from above by splitting it over the intervals $(-\infty,\overline{x}(t)-ct]$, $[\overline{x}(t)-ct,+\infty)$, \textit{i.e.} on $\{P\geq 1\}$ and $\{P<1\}$, respectively.

Recall from \Cref{prop.W.pos} that $W\geq 0$.
Hence, 
\be\label{e.P_diff_inequality}
 - \tilde P_z - \chi \tilde P
 = -\tilde P_z - \eta(P)
 = W
 \geq 0
 \quad\text{ on } [\overline x(t) - ct, \infty).
\ee
The first equality uses that $\eta(P) = \chi P$; see~\eqref{e.eta_and_Q_expressions}
The second equality is simply the definition~\eqref{def.w} of the shape defect function. 
Using~\eqref{e.P_diff_inequality} and by observing that $\tilde{P}(t,\overline{x}(t)-ct)=1$, we get for $z\geq \overline{x}(t)-ct$:
\be
    \tilde{P}(t,z)
    \leq e^{-\chi(z-(\overline{x}(t)-ct))}, 
\ee
which leads to 
\be 
\label{e.2511112}
    \frac{1}\chi \int^{+\infty}_{\overline{x}(t)-ct}  e^{\nicefrac{z}{\chi}} \tilde P(t,z) dz 
    \leq
    \frac{e^{\chi (\bar x(t) - ct)}}\chi \int^{+\infty}_{\overline{x}(t)-ct}  e^{- z \left( \chi - \nicefrac{1}{\chi}\right)} dz
    =\frac{1}{\chi^2 - 1} e^{\nicefrac{\overline x(t) - ct}{\chi}}.
\ee
This estimate crucially relies on the condition $\chi>1$, which is specific to the pushed regime.

On the interval $(-\infty,\overline{x}(t)-ct]$, we proceed first with an integration by parts
\be 
\label{e.2511113}
\frac{1}\chi \int_{-\infty}^{\overline{x}(t)-ct} e^{\nicefrac{z}{\chi}} \tilde P(t,z) dz
=  e^{\nicefrac{\overline{x}(t)-ct}{\chi}} + \int_{-\infty}^{\overline{x}(t)-ct} e^{\nicefrac{z}{\chi}} \tilde \rho (t,z) dz.
\ee
At this point, we require a bound on $\tilde \rho$.  Applying \Cref{c.rhoLinfty} yields for all $z \leq \overline x(t) - ct$,
\be\label{e.c030201}
    \tilde{\rho}(t,z)
    \leq C + C\sup_{s\in [0,t]} \frac{e^{-\nicefrac{\overline{x}(s)-cs}{\chi}}}{\sqrt{1+s}}
\ee
for some large constant $C$ depending only on $\chi$ and the initial data. As an aside, this is the only difference with the proof for $u$; indeed, we may use the obvious bound $u\leq 1$ ``for free'', so we do not need to apply \Cref{c.rhoLinfty}.

From~\eqref{e.c030201} and \eqref{e.2511113}, we find
\be 
\label{e.2511114}
    \frac{1}\chi
    \int_{-\infty}^{\overline{x}(t)-ct} e^{\nicefrac{z}{\chi}} \tilde P(t,z) dz
    \leq  e^{\nicefrac{\overline{x}(t)-ct}{\chi}}
        + \chi Ce^{\nicefrac{\overline{x}(t)-ct}{\chi}}\left( 1 + 
        \sup_{s\in [0,t]} \frac{e^{-\nicefrac{\overline{x}(s)-cs}{\chi}}}{\sqrt{1+s}}\right).
\ee
Combining \eqref{e.2511115}, \eqref{e.2511112}, and \eqref{e.2511114} leads to the inequality
\be 
    e^{-\nicefrac{\overline{x}(t)-ct}{\chi}} I(0)
    \leq C 
        +C\sup_{s\in [0,t]} \frac{e^{-\nicefrac{\overline{x}(s)-cs}{\chi}}}{\sqrt{1+s}}.
\ee
At this point, we may apply \Cref{l.upperboundgpush} with  $g(t)= e^{-\nicefrac{\overline{x}(t)-ct}{\chi}}$ to find
\be 
    e^{-\nicefrac{\overline{x}(t)-ct}{\chi}} \leq C
\ee
which, after taking logarithms, concludes the proof.
\end{proof}

We now prove the two technical lemmas.  We begin with \Cref{c.rhoLinfty}, which establishes an estimate on $\rho$ in terms of $\hat W$.

\begin{proof}[Proof of \Cref{c.rhoLinfty}] 
In the region $\{ x \leq \overline{x}(t)\}$, $\rho$ satisfies
\be 
\rho_t+\chi \rho_x=\rho_{xx}.
\ee
The maximum principle yields
\be 
\label{bd.rho_go_aux}
\sup_{x \leq \overline{x}(t)}\rho(t,x)\leq \|\rho_{\rm in}\|_{L^\infty(\R)} + \sup_{s\in [0,t]} \rho(s,\overline{x}(s)).
\ee
Next, observe that
\be 
\rho(s,\overline{x}(s))=W(s,\overline{x}(s)) +\eta(P(s,\overline{x}(s)) )=\tilde{W}(s,\overline{x}(s)-cs) +\eta(1) = e^{-\frac{\overline{x}(s)-cs}{\chi}}\hat{W}(s,\overline{x}(s)-cs) +\eta(1) .
\ee
From \Cref{l.l.eta_nonloc}, we have $\eta(1)=\frac{1}{c-\chi}$. Using this with \Cref{thm.W.decay} leads to
\be 
\sup_{s\in [0,t]} \rho(s,\overline{x}(s)) \leq \frac{1}{c-\chi} + \sup_{s\in [0,t]} \left\{e^{-\frac{\overline{x}(s)-cs}{\chi}}\min\left\{ \left\|\hat{W}_{\rm in}\right\|_{L^\infty(\R)},  \frac{C}{\sqrt{s}}\left\|\hat{W}_{\rm in}\right\|_{L^1(\R)}\right\}\right\},
\ee
which combined with \eqref{bd.rho_go_aux} concludes the proof.
\end{proof}

We now establish \Cref{l.upperboundgpush}, the technical lemma that provides an upper bound on a function $g$ of time that satisfies a particular  recursive bound. 
\begin{proof}[Proof of \Cref{l.upperboundgpush}]
Setting 
\be
    G(t)=\sup_{s\in [0,t]}g(s)
    \quad\text{ and }\quad
    H(t)=\sup_{s\in [0,t]}\frac{g(s)}{\sqrt{1+s}},
\ee
we find that
\be 
\label{e.2511116}
    G(t)
    \leq C(1+H(t)).
\ee
Without loss of generality, we may assume that $C\geq 1$.  Pick $t_0$ such that
\be\label{e.t_0}
    \frac{C}{\sqrt{1+t_0}} = \frac{1}{2}.
\ee
Observe that
\be 
    H(t)
    =\max\left\{
        H(t_0),
        \sup_{s\in [t_0,t]}\frac{g(s)}{\sqrt{1+s}}
    \right\}
    \leq
    \max\left\{
        H(t_0),
        \frac{G(t)}{\sqrt{1+t_0}}
    \right\}.
\ee
Plugging this into (\ref{e.2511116}) leads to 
\be 
\label{e.2511116-bis}
G(t)\leq C\left(1+\max\left\{
H(t_0),\frac{G(t)}{\sqrt{1+t_0}}\right\}\right).
\ee

We now consider the two cases based on the maximum above.  First, assume that
\be
    \max\left\{
        H(t_0),
        \frac{G(t)}{\sqrt{1+t_0}}
    \right\}
        = \frac{G(t)}{\sqrt{1+t_0}}.
\ee
Then, due to~\eqref{e.t_0},
\be
    G(t)
        \leq C+ \frac{C}{\sqrt{1+t_0}} G(t)
        = C+ \frac{1}{2} G(t),
\ee
which clearly yields an upper bound of $2C$ for $G(t)$.

On the other hand, if
\be
    \max\left\{
        H(t_0),
        \frac{G(t)}{\sqrt{1+t_0}}
    \right\}
        = H(t_0),
\ee
then we obtain an immediate upper bound on $G$ of $C(1 + H(t_0))$ from \eqref{e.2511116-bis}.

In summary, for $t\geq t_0$
\be
    G(t)
    \leq \max\left\{ 2C, C(1+H(t_0))\right\},
\ee
which concludes the proof after applying the local boundedness of $G$.
\end{proof}

\section{The pushmi-pullyu case: $\chi = 1$}\label{s.pp}

We now specialize \Cref{thm-asymptotics} to the pushmi–pullyu regime.

\begin{theorem}\label{thm.asymptotics_pushmipullyu}
    	Let $\chi = 1$.  Suppose that  $P$ (resp. $u$) solves~\eqref{pde.P} (resp.~\eqref{pde.u}) with nontrivial initial data whose shape defect function $W_{\rm in}$ (resp.~$w_{\rm in}$) is nonnegative and such that, for some $\gamma >0$,
        \begin{itemize}
            \item (nonlocal case) $\int P_{\rm in}(x)e^xdz<+\infty$ and $W_{\rm in} \leq \gamma^{-1} e^{- \gamma x_+^2}$;
            \item (local case) 
            $ u_{\rm in} \leq \min\{1, \gamma^{-1} e^{-\gamma x_+^2}\}$.
        \end{itemize}
        Then, there is $C$ such that, for $t\geq 0$,
	\be
		2t - \frac{1}{2} \log t - C\leq \overline{x}(t) \leq 2t - \frac{1}{2} \log t + C.
	\ee
\end{theorem}

Let us briefly discuss the assumptions.  First, we observe that the Heaviside function satisfies these assumptions.  Specifically, we may take $\rho_{\rm in} = u_{\rm in} = \1_{(-\infty,0)}$. Since this is a typical initial data used in front propagation problems, it is reassuring that it is admissible here.

Second, $W_{\rm in}$ being bounded is equivalent to $\rho_{\rm in}$ being bounded.  The Gaussian bound on the right, however, is not equivalent.  Indeed, if $\rho_{\rm in}$ is equal to the traveling wave, then $W_{\rm in}$ is zero (and, thus, bounded by a Gaussian), but $\rho_{\rm in}$ is not bounded by a Gaussian.  Of course, this case is ruled out by the finite exponential moment of $P_{\rm in}$.  This is important because such a case, {\em i.e.}, $\rho_{\rm in} = \rho_{\rm tw}$, would not yield a logarithmic delay.  On the other hand, a Gaussian bound on $\rho_{\rm in}$ yields a Gaussian bound for $W_{\rm in}$.  

Finally, we note that the upper and lower bounds hold under more general assumptions; see \Cref{p.pp_lower} and \Cref{p.pp_upper} below.  This reflects the different strategies employed in each proof.

\subsection{Lower bound in the pushmi-pullyu case: $\chi =1$}
\label{s.pp_lower}

We begin with the lower bound as it is simpler and shorter than the upper bound.  
Our goal in this section is to prove the following.
\begin{proposition}[Lower bound]\label{p.pp_lower}
	Let $\chi =1$.  Suppose that $P$ (resp. $u$) solves~\eqref{pde.P} (resp.~\eqref{pde.u}) with nontrivial initial data $P_{\rm in}$ (resp. $u_{\rm in}$).
  Assume further that 
    \begin{itemize}
        \item (nonlocal case) $\int P_{\rm in}(x)e^xdx<+\infty$, $W_{\rm in} \geq 0$, and $\hat{W}_{\rm in}(x):=e^{x}{W}_{\rm in}(x) \in  L^\infty(\R)\cap L^1(\R) $;
        \item (local case) $\int u_{\rm in}(x)e^xdx<+\infty$ and $w_{\rm in}\geq 0$.
    \end{itemize}   
    Then there is $C$ such that, for $t\geq 0$,
	\be
		\overline{x}(t) \geq 2t - \frac{1}{2} \log(1+ t) - C.
	\ee
\end{proposition}

The main computations in this proof follow the work of \cite{an2021,AnHendersonRyzhik_Quantitative}.  In particular, we shift to an appropriate moving frame and derive a differential inequality for certain exponential moments of $u$ similar to the Fabes-Stroock approach to heat kernel bounds~\cite{FabesStroock}. 
Compared to the previous works, the main difficulty is due to the unboundedness of $\tilde P$ and the lack of an immediate upper bound on $\tilde \rho$ behind the front. These complications are due to the nonlocality in our setting.

The main technical lemma allowing us to overcome these additional issues is the following.

\begin{lemma}
\label{l.upperboundgpushmi}
Suppose that $f:[0,+\infty)\to \R$ is locally bounded and satisfies the inequality
\be 
\label{e.25111110}
e^{-f(t)}\leq C\sqrt{1+t}-Df(t)+C\sup_{s\in [0,t]} \frac{e^{-f(s)}}{\sqrt{1+s}},\ee
for $C,D>0$. Then, there exists a constant $K$ such that
\be 
f(t) \geq - \frac{1}{2}\log(1+t)-K.
\ee
\end{lemma} 

We postpone the proof of \Cref{l.upperboundgpushmi} to the end of this subsection and are now ready to prove the lower bound on $\overline x(t)$.

\begin{proof}[Proof of \Cref{p.pp_lower}]
We only show the proof in the nonlocal case as the proof in the local case is analogous but simpler for the same reasons already mentioned in \Cref{s.pushed} for the pushed case $\chi>1$.

As in previous proofs, we shift to the moving frame, denoting functions in this frame by a tilde; \textit{e.g.} 
\be
    \tilde P(t,z) := P(t,z + 2t).	
\ee
Recall~\eqref{pde.P}.  Thus,
\be\label{e.c111901}
    \tilde P_t - 2\tilde P_z + (A^P(\tilde P))_z
        = \tilde P_{zz} + \tilde P - A^P(\tilde P).
\ee
As before, we consider the exponential moment $I(t)=I_P(t)$ as defined in \eqref{defi.It}. 
From \Cref{l.expmoment}, we know that for $t\geq 0$,
\be
I(t)=I(0).
\ee

The estimate here is more involved than the proof of \Cref{l.pushed_lower}, which establishes the lower bound in the pushed case ($\chi>1$).  Indeed, we bound the exponential moment from above by estimating $\tilde{P}$ on the {\em three} intervals $(-\infty, \overline{x}(t)-2t]$, $[ \overline{x}(t)-2t,N\sqrt{t}]$, and $[N\sqrt{t},+\infty))$ for a large constant $N>0$ to be chosen.  The third interval plays the crucial role in incorporating the long-range Gaussian behavior of $e^z \tilde P$.

Note that, by \Cref{c.easyupperbound}, $\overline{x}(t)-2t$ is bounded above, and hence 
\be
    \overline{x}(t)-2t \leq N\sqrt{t}
\ee
for all $t$ large once $N$ is fixed.  We use this below.

We begin with the interval $(-\infty,\overline{x}(t)-2t]$.  Here, we argue as in \Cref{l.pushed_lower}. In particular, after applying \Cref{c.rhoLinfty} and arguing as in~\eqref{e.2511114}, we find
\be 
\label{e.2511118}
    \int_{-\infty}^{\overline{x}(t)-2t} e^{z} \tilde P(t,z) dz
    \leq 
    Ce^{\overline{x}(t)-2t}\left( 1 + \sup_{s\in [0,t]} \frac{e^{-(\overline{x}(s)-2s)}}{\sqrt{1+s}}\right)
\ee
for $C>0$ sufficiently large depending only on $\chi=1$ and the initial data.

Next, we consider the estimate on the interval $[ \overline{x}(t)-2t,N\sqrt{t}]$.  Here, we use that $\tilde{P}(t,z)\leq e^{-(z-(\overline{x}(t)-2t))}$, since $W\geq 0$.  Hence,
\be 
\label{e.2511117}
    \int_{ \overline{x}(t)-2t}^{N\sqrt{t}} e^z \tilde P (t,z) dz 
    \leq
    e^{\overline x(t) - 2t} \int_{ \overline{x}(t)-2t}^{N\sqrt{t}} dz
    =
    e^{ \overline x(t)-2t}(N\sqrt{t}-(\overline x(t) - 2t)).
\ee

 Finally, we state a bound on the last interval $[N\sqrt{t},+\infty)$.  Let us introduce another exponential moment and claim that, for any $m \in(0,1)$,
\be\label{e.I_m}
    I_m(t)
        := \int (e^{mz} + e^{-mz}) e^z \tilde \rho(t,z) dz
        \leq C e^{C m^2 t}.
\ee
We postpone this momentarily and show how to conclude the proof with it.  Indeed, using~\eqref{e.I_m}, we find
\be
    \begin{split}
    \int^\infty_{N \sqrt t} e^z \tilde P(t,z ) dz
        =&\left[ e^z \tilde P(t,z )\right]^\infty_{N \sqrt t}+\int^\infty_{N \sqrt t} e^z \tilde \rho(t,z ) dz
        \leq \int_{N\sqrt t}^\infty e^{\frac{z - N \sqrt t}{\sqrt t}} e^z \tilde \rho(t,z) dz
        \\&
        = e^{-N} \int_{N\sqrt t}^\infty e^{\frac{z}{\sqrt t}} e^z \tilde \rho(t,z) dz
        \leq e^{-N} I_{\frac{1}{\sqrt t}}(t)
        \leq e^{-N} C.
    \end{split}
\ee
The justification of the integration by parts is argued similarly as in the proof of \Cref{l.pushed_lower} in the pushed case $\chi>1$, since $\exp\{z\} \tilde P(t,z )|_{z=+\infty}=0$, as a consequence of the finiteness of the exponential moment and $W\geq 0$. 
Choosing $N> \max\left\{1,- \log\left(\frac{I(0)}{2C}\right)\right\}$, we find
\be\label{e.c010602}
    \int^\infty_{N \sqrt t} e^z \tilde P(t,z ) dz
        \leq \frac{I(0)}{2}.
\ee

Combining \eqref{e.2511118}, \eqref{e.2511117}, and \eqref{e.c010602}, absorbing the $\sfrac{I(0)}{2}$ from the right hand side into the left hand side, and multiplying by $\exp\{-(\overline x(t) - 2t)\}$ yields
\be 
\label{e.2511119}
 e^{-(\overline{x}(t)-2t)}\frac{I(0)}{2} \leq C\left( 1 + \sup_{s\in [0,t]} \frac{e^{-(\overline{x}(s)-2s)}}{\sqrt{1+s}}\right) +(N\sqrt{t}-(\overline x(t) - 2t)).
\ee
We now invoke \Cref{l.upperboundgpushmi}.  Indeed, we see that \eqref{e.2511119} satisfies inequality \eqref{e.25111110} of \Cref{l.upperboundgpushmi} with $f(t):=\overline{x}(t)-2t$ and thus we obtain
\be 
    \overline{x}(t)
    \geq 2t- \frac{1}{2}\log(1+t)-K,
\ee
which concludes the proof up to the proof of \eqref{e.I_m}.
	 
We now establish~\eqref{e.I_m} following the ideas in~\cite[Section~5.3]{an2021}.  Integrating by parts, we find
	 \be
	 	\begin{split}
			\dot I_m
				&= \int e^z (e^{mz} + e^{-mz}) [((2-\alpha)\tilde \rho)_z + \tilde \rho_{zz} + \tilde \rho(1 - \alpha)] dz
			\\&
				= - \int e^z ( (m+1) e^{mz} - (m-1) e^{-mz})(2-\alpha)\tilde \rho dz
			\\&\qquad
				+ \int e^z ((m+1)^2 e^{mz} + (m-1)^2 e^{-mz}) \tilde \rho dz
			\\&\qquad
				+ \int e^z (e^{mz} + e^{-mz})\tilde \rho(1 - \alpha)] dz
			\\&
				= m^2 I_m + m \int e^z (e^{mz} - e^{-mz}) \tilde \rho(t,z) \alpha(\tilde P) dz.
		\end{split}
	 \ee
	 The proof is finished if we show that the last term is bounded by $Cm^2 I_m$, which we do now:
	 \be
	 	\begin{split}
			m\int e^z (e^{mz} - e^{-mz}) \tilde \rho(t,z) \alpha dz
				&= m\int \tanh(mz) e^z (e^{mz} + e^{-mz}) \tilde \rho(t,z) \alpha(\tilde P) dz
				\\&
				\leq m^2 \int z_+\alpha(\tilde P) e^z (e^{mz} + e^{-mz}) \tilde \rho(t,z)  dz.
		\end{split}
	 \ee
	 Since we know that $\overline{x}(t)-2t\leq C$, we have that
	\be
		z_+ \alpha(P)
			= z_+ \1_{(-\infty, \overline x(t)-2t]}(z)
			\leq (\overline{x}(t)-2t)_+\leq C_+. 
	\ee
	This yields
	\be 
	\dot{I}_m\leq m^2(1+C)I_m
	\ee
	and establishes~\eqref{e.I_m}.
    The proof is complete.
\end{proof}

It remains to prove the technical lemma.	
\begin{proof}[Proof of \Cref{l.upperboundgpushmi}]
Let
\be
    h(t)=\frac{e^{-f(t)}}{\sqrt{1+t}}.
\ee
We are finished if we show that $h$ is bounded.  We do this using \Cref{l.upperboundgpush}.  It is, thus, sufficient to show that $h$ satisfies the hypothesis of \Cref{l.upperboundgpush}.

Note that, $\log(p)\leq p$ for $p>0$, as well as that for any constant $C>0$, there is a constant $\bar C>0$ such that for all $p>0$, $C\log p \leq \bar C + \sfrac{p}{2}$.  Observe that, after dividing by $\sqrt{1+t}$, we have
\be
\begin{split}
    h(t)
    &
    \leq C + D \frac{\log\left(\sqrt{1+t} h(t)\right)}{\sqrt{1+t}}
    + \frac{C}{\sqrt{1+t}} \sup_{s\in[0,t]} h(s)
    \\&
    \leq C + D \frac{\log\left(\sqrt{1+t}\right)}{\sqrt{1+t}} + \frac{D}{\sqrt{1+t}} \log h
    + C \sup_{s\in[0,t]} \frac{h(s)}{\sqrt{1+s}}
    \\&
    \leq C + D+ \bar D + \frac{h(t)}{2}
    + C \sup_{s\in[0,t]} \frac{h(s)}{\sqrt{1+s}}
    .
\end{split}
\ee
Rearranging this, we find that the hypothesis of \Cref{l.upperboundgpush} is satisfied, concluding the proof.

\end{proof}

\subsection{Upper bound in the pushmi-pullyu case: $\chi = 1$}

Our goal in this section is to prove the following matching upper bound on the front location.
\begin{proposition}[Upper bound]\label{p.pp_upper}
Let $\chi = 1$.  Suppose that  $P$ (resp. $u$) solves~\eqref{pde.P} (resp.~\eqref{pde.u}) with nontrivial initial data such that, for some $\gamma >0$,
\begin{itemize}
    \item (nonlocal case) $\int P_{\rm in}(x)e^xdx<+\infty$ and $0 \leq W_{\rm in} \leq \gamma^{-1} e^{- \gamma x_+^2}$;
    \item (local case) $0 \leq u_{\rm in} \leq \min\{1, \gamma^{-1} e^{-\gamma x_+^2}\}$.
\end{itemize}
Then there is $C$ such that, for $t\geq 0$,
	\be
    \label{e.26021201}
		\overline{x}(t) \leq 2t - \frac{1}{2} \log(1+ t) + C.
	\ee
\end{proposition}

The proofs in both the local and nonlocal cases rely, however, on different constructions. In the local case, we are able to construct a direct supersolution for \eqref{pde.u} on $u$.  An analogous construction does not work in the nonlocal case because $P$ is not bounded behind the front; it has linear growth as $x \to -\infty$.  
The technical reasons for the failure of this construction are discussed further at the end of \Cref{s.pp_local_upper}. Instead, we bypass this difficulty by constructing a supersolution for the shape defect function $W$ and then using that upper bound to deduce the result on $P$.  We note that it is technically difficult to justify a similar supersolution for the shape defect function $w$ in the local case due to the singular nature of the PDE that it satisfies.

We first treat the local case in \Cref{s.pp_local_upper}. At the end of that subsection, we briefly explain why the same construction cannot be adapted to the nonlocal setting. We then turn to the nonlocal case in \Cref{s.pp_nonlocal_upper}, where the argument proceeds through the shape defect function $W$.

\subsubsection{The Local Case}
\label{s.pp_local_upper}

We find an upper bound on $u$ in order to bound the front when $\chi = 1$.   Let us make a quick note on the heuristics behind this construction.  Going back to~\cite{ebert2000}, see also the seminal work~\cite{HNRR_JEMS,hamel2013} , it is understood that solutions to reaction diffusion equations look, roughly, like
\be\label{e.heuristic_asymptotics}
   \breve u(t,z)
    \approx u_{\rm tw}(z) e^{-\frac{z^2}{4t}}
    = \min\{1, e^{-z}\} e^{-\frac{z^2}{4t}}
\ee
in the correct traveling frame (\textit{e.g.} here with $\chi=1$, we would consider the frame $\breve u(t,z)=u(t,z+2t-\nicefrac{1}{2}\log(t+t_0))$). With some careful perturbations, we can make a supersolution out of this heuristic, which bounds $u$ from above.  Of course, this immediately yields an upper bound on $\overline x$, as desired.

\begin{proof}[Proof of \Cref{p.pp_upper} in the local case] 
We begin by passing to the moving frame
\be
    \breve u(t,z)
    =
    u\left(t, z + 2t - \tfrac{1}{2}\log (t+t_0)\right)
\ee
for $t_0\gg 1$ to be determined. Then~\eqref{pde.u} becomes
\be 
\label{e.umv12}
    \breve{u}_t
    =
    \breve{u}_{zz}
        +\left( 2 -\frac{1}{2(t+t_0)}\right)\breve{u}_z
        - \left( \alpha(\breve u)\breve{u}\right)_z
        + \breve u (1 - \alpha(\breve u)).
\ee 
Our proof is based on the construction of a supersolution $\overline{u}$, which will be the concatenation of a supersolution of \eqref{e.umv12},  when $\overline{u}< 1$ and the constant function $\equiv 1$. We then invoke the comparison principle, \Cref{thm.CP_u}. Let us stress that for $\overline u$ to be a supersolution in the sense of \Cref{thm.CP_u}, we need not only that
\be 
    \overline{u}_t
    \geq
    \overline{u}_{zz}
        +\left( 2 -\frac{1}{2(t+t_0)}\right)\overline{u}_z
        - \left( \alpha(\overline u)\overline{u}\right)_z
        + \overline u (1 - \alpha(\overline u)),
\ee 
but also that the associated shape defect function $\overline{w}=-\overline{u}_z-\eta^u(\overline u)$ be nonnegative. In order to show the nonnegativity of the shape defect $\overline{w}$, we must show that in the region, where $\overline{u}<1$, we have $-\overline{u}_z-\eta^u(\overline u)=-\overline{u}_z-\overline u\geq 0$ and pay special attention to the kink at the boundary between these two regions.

\medskip
\noindent
{\bf \# Step (1):} Let us start by determining the shape of our supersolution $\overline u$, when $\overline u <1$.  
In order to make our computations more manageable, we write our candidate supersolution in component parts.  For positive constants $\beta, K \gg 1$ to be determined, let
\be\label{e.pp_supersolution}
\begin{split}
&E(z):= e^{-z},
\quad
G(t,z):=e^{-\frac{z^2}{4(t+t_0)}},
\quad
H(t,z):=e^{\frac{1}{\sqrt{t+t_0}}\left( \frac{z^2}{t+t_0} - K \right)},
\\
&
\text{and } \quad
    R (t,z):=\beta E(z)G(t,z)H(t,z).
\end{split} 
\ee
Notice that this form of supersolution corresponds, modulo to the term $H$, to the shape given by ~\eqref{e.heuristic_asymptotics}.
We use the suggestive letters $E$ for exponential, $G$ for Gaussian, and $H$ for helper (it plays no role in the asymptotics but helps $R$ to be a supersolution).  Additionally, we use $R$ for right, as it will be shown to be a supersolution on the right.

For simplicity, let us denote the linear operator when $z \geq \overline x(t)-(2t - (\sfrac12)\log(t+t_0))$ as
\be 
\label{def.Lbeta}
    \cL_{\sfrac12} f
    := f_t - f_{zz} -\left(2-\frac{1}{2(t+t_0)}\right) f_z - f,
\ee
and we will show that $\cL_{\sfrac{1}{2}}R \geq 0$ under certain conditions on $t_0,K$. 

Let us begin by noting some useful identities:
\be
    -E_{zz}-2 E_z - E =0
    \quad\text{ and }\quad
    G_t - G_{zz}=\frac{1}{2(t+t_0)}G.
\ee
Additionally, we have the logarithmic derivative identities
\be
\begin{split}
    &\frac{\partial(EGH)}{EGH}
    =\frac{\partial E}{E}
        +\frac{\partial G}{G} 
        +\frac{\partial H}{H}
    \quad\text{ and}
    \\&
    \frac{\partial^2(EGH)}{EGH}
    =\frac{\partial^2 E}{E}
        +\frac{\partial^2 G}{G}
        +\frac{\partial^2 H}{H}
        +2\frac{\partial E}{E}\frac{\partial G}{G}
        +2\frac{\partial G}{G}\frac{\partial H}{H}
        +2\frac{\partial H}{H}\frac{\partial E}{E}.
\end{split}
\ee
Using the above, we find
\be
\begin{split}
    \frac{\cL_{\sfrac12} R}{R}
    &= \frac{1}{R}\left(  R_t - R_{zz} -\left(2-\frac{1}{2(t+t_0)}\right) R_z - R\right)
    \\
        &= \left\{ -\frac{ E_{zz}}{E } -\left(2-\frac{1}{2(t+t_0)}\right)\frac{ E_z}{E }  - 1 \right\} + \left\{\frac{ G_t}{G } -\frac{ G_{zz}}{G } -\left(2-\frac{1}{2(t+t_0)}\right)\frac{ G_z}{G } \right\}
        \\&\qquad
        +\left\{ \frac{ H_t}{H } -\frac{ H_{zz}}{H } -\left(2-\frac{1}{2(t+t_0)}\right)\frac{ H_z}{H }\right\}  -2\left\{\frac{E_z}{E}\frac{G_z}{G}+\frac{G_z}{G}\frac{ H_z}{H}+\frac{ H_z}{H}\frac{ E_z}{E} \right\}
    \\
        &=  \frac{1}{2(t+t_0)}\frac{ E_z}{E }  + \left\{ \frac{1}{2(t+t_0)} -\left(2-\frac{1}{2(t+t_0)}\right)\frac{ G_z}{G } \right\}
        \\&\qquad
        +\left\{ \frac{ H_t}{H } -\frac{ H_{zz}}{H } -\left(2-\frac{1}{2(t+t_0)}\right)\frac{ H_z}{H }\right\}  -2\left\{\frac{E_z}{E}\frac{G_z}{G}+\frac{G_z}{G}\frac{ H_z}{H}+\frac{ H_z}{H}\frac{ E_z}{E} \right\} .
\end{split}
\ee
Let us observe here that the term $-\nicefrac{R}{R}=-1$ is grouped with the terms in $E$. Next, using
\be\begin{split}
\label{e.EGH}
    &\frac{E_z}{E}
        = -1,
    \quad
    \frac{ G_z}{G}
        = -\frac{z}{2(t+t_0)},
    \quad\frac{H_z}{H}
        = \frac{2z}{(t+t_0)^{\sfrac{3}{2}}}, 
    \\
    &\frac{H_{zz}}{H}
        = \frac{4z^2}{(t+t_0)^3}+\frac{2}{(t+t_0)^{\sfrac{3}{2}}}, 
    \quad\text{ and }\quad
    \frac{ H_t}{H}
        = -\frac{3z^2}{2(t+t_0)^{\sfrac{5}{2}}}+\frac{K}{2(t+t_0)^{\sfrac{3}{2}}},
\end{split}\ee
yields
\be\begin{split}
    \frac{\cL_{\sfrac12} R}{R}
    &
    = \left\{ \left( 2 - \frac{1}{2(t+t_0)}\right) \frac{z}{2(t+t_0)} \right\}  \\
    &\qquad 
    + \left\{ -\frac{3z^2}{2(t+t_0)^{\sfrac{5}{2}}}+\frac{K}{2(t+t_0)^{\sfrac{3}{2}}} -\frac{4z^2}{(t+t_0)^3}-\frac{2}{(t+t_0)^{\sfrac{3}{2}}} -\left(2-\frac{1}{2(t+t_0)}\right)\frac{2z}{(t+t_0)^{\sfrac{3}{2}}}\right\} 
    \\
    &\qquad
        -2\left\{\frac{z}{2(t+t_0)}-\frac{z}{2(t+t_0)} \frac{2z}{(t+t_0)^{\sfrac{3}{2}}}-\frac{2z}{(t+t_0)^{\sfrac{3}{2}}}\right\}.
        \end{split}
    \ee
Then, regrouping all matching terms in powers of $z$ and $t+t_0$, we find
\be 
    \frac{\cL_{\sfrac12} R}{R}
    = 
        \frac{z^2}{2(t+t_0)^{\sfrac52}} \left(1 - \frac{8}{(t+t_0)^{\sfrac12}}\right)
        - \frac{z}{4(t+t_0)^2}
        + \frac{K-4}{2(t+t_0)^{\sfrac32}}
        + \frac{z}{(t+t_0)^{\sfrac52}}.
\ee
Let us now choose
\be
\label{bd.K_t0}
    K = 4 + \nicefrac{\tilde K}{2}
    \quad\text{ with }
    \tilde K \geq 1
    \text{ and }
    t_0 > 16^2,
\ee
so that
\be\label{e.LcalFoverF}
\begin{split}
    \frac{\cL_{\sfrac12} R}{R}
    \geq&
    \frac{z^2}{4(t+t_0)^{\sfrac52}}
        - \frac{z}{4(t+t_0)^2}
        + \frac{\tilde K}{4(t+t_0)^{\sfrac32}}
        + \frac{z}{(t+t_0)^{\sfrac52}} \\
        \geq &
        \frac{1}{4(t+t_0)^{\sfrac32}}\left(\frac{z^2}{t+t_0}
    -\frac{z}{\sqrt{t+t_0}}+1+\frac{4z}{t+t_0}
        \right)
        .
\end{split}
\ee
If $z \geq \sqrt{t+t_0}$, then the sum of the first two terms on the right is positive.  If $z \leq \sqrt{t+t_0}$, then the sum of the second and third terms is positive.  Hence, in all cases, we have that
\be
    \cL_{\sfrac12} R \geq 0,
\ee
as desired.  In other words, $R$ is a supersolution~\eqref{pde.u} when $R<1$. 

\medskip
\noindent
{\bf \# Step (2):} Let us now show that $R$ has a nonnegative shape defect when $R<1$. To do so, we need to show that 
\be 
-R_z-\eta^u(R)=-R_z-R=R\left(-\frac{R_z}{R}-1\right)\geq 0.
\ee
Through a quick computation, we find that
\be
-\frac{R_z}{R}-1 = -\frac{G_z}{G}-\frac{H_z}{H}=\frac{z}{(t+t_0)}\left(\frac{1}{2}-\frac{2}{\sqrt{t+t_0}} \right)\geq 0,
\ee
whenever $z\geq 0$, since $t_0\geq 16$.

\medskip
\noindent
{\bf \# Step (3):} We concatenate a plateau $\equiv 1$ on the left with $R$ on the right to obtain a supersolution $\overline{u}$.  
Setting $t_0>K^2$, we obtain that $H\geq e^{-1}$ and thus choosing $\beta>e$ leads to the maximum of $R(t,\cdot)$ being larger than one for any $t\geq 0$.  Hence, there exists a unique, smooth $ \overline  z_{\rm super}(t) \in (0, \log \beta)$ such that $R(t,\overline  z_{\rm super}(t)) = 1$.   We define
\be
    \overline{u}(t,z)
    = \left\{
        \begin{array}{ll}
            R(t,z) & \text{ if }z\geq  \overline  z_{\rm super}(t) \\
            1& \text{ if }z< \overline  z_{\rm super}(t).
        \end{array}
        \right.
\ee

\medskip
\noindent
{\bf \# Step (4):} Let us check that $\overline u$ is a supersolution with nonnegative shape defect. 
Let us observe that the distribution
\be
    \mathcal{P}(\overline{u}):=\overline u_t -\left( 2 -\frac{1}{2(t+t_0)}\right)\overline{u}_z
        + (A^u(\overline u)-\overline u_z)_z
        - \overline u + A^u(\overline u)
\ee
is smooth away from $z=\overline z_{\rm super}(t)$ and has only a single kink there.  Hence, $\mathcal{P}(\overline{u})$ is a signed Radon measure. Therefore, we may decompose it in its absolutely continuous and singular parts:
\be 
\mathcal{P}(\overline{u}) = \left[ \mathcal{P}(\overline{u})\right]_\text{ac} +\left[ \mathcal{P}(\overline{u})\right]_\text{s}.
\ee
From Step (1), we have that $\overline u$ is a supersolution to~\eqref{pde.u} when $z >  \overline z_{\rm super}(t)$, as there $u< 1$.  It is also clearly a supersolution when $z < \overline z_{\rm super}(t)$, as there $u\equiv 1$. This establishes that
\be 
\left[ \mathcal{P}(\overline{u})\right]_\text{ac} \geq 0.
\ee
Next, we observe that the only source of a singular term is the $(A^u(\overline u)-\overline u_z)_z$ term at  $z=\overline z_{\rm super}(t)$. Recall that $A^u(\overline u)=1$, when $z<\overline z_{\rm super}(t)$ and $A^u(\overline u)=0$, when $z>\overline z_{\rm super}(t)$. Thus, we obtain that
\be
\left[ \mathcal{P}(\overline{u})\right]_\text{s}
=\left(\left(A^u(\overline u(t,\overline  z_{\rm super}(t)^+))-\overline u_z(t,\overline  z_{\rm super}(t)^+)\right) - \left( A^u(\overline u(t,\overline  z_{\rm super}(t)^-))-\overline u_z(t,\overline  z_{\rm super}(t)^-)\right)
\right)\delta_{\overline  z_{\rm super}(t)}.
\ee
Observe that
\be
\begin{split}
    &\left(A^u(\overline u(t,\overline  z_{\rm super}(t)^+))-\overline u_z(t,\overline  z_{\rm super}(t)^+)\right)
    -
    \left( A^u(\overline u(t,\overline  z_{\rm super}(t)^-))-\overline u_z(t,\overline  z_{\rm super}(t)^-)\right)
    \\&
    =
    \left(0-\overline{u}_z(t,\overline  z_{\rm super}(t)^+)\right)
    -
    \left(1-0\right)
    =
        -R_z(t, \overline  z_{\rm super}(t))-1
    =
        \frac{\overline  z_{\rm super}(t)}{4(t+t_0)} - \frac{\overline z_{\rm super}(t)}{(t+t_0)^{\sfrac32}}
    >0.    
\end{split}
\ee
The last inequality follows because $\overline  z_{\rm super}(t)>0$ and $t_0\geq 16$. Thus,
\be 
    \overline u_t-\left( 2 -\frac{1}{2(t+t_0)}\right)\overline{u}_z
        + (A^u(\overline u))_z
        - \overline u_{zz}
        - \overline u + A^u(\overline u)
    = \mathcal{P}(\overline{u})
    \geq 0.
\ee
Furthermore, from Step (2), we have that the shape defect of $\overline u$ is nonnegative when $z > \overline z_{\rm super}(t)$, and it is also clearly nonnegative when $z<\overline z_{\rm super}(t)$. This establishes that $\overline{u}$ is a supersolution with nonnegative shape defect function and the comparison principle \Cref{thm.CP_u} applies.

\medskip
\noindent
{\bf \# Step (5):}  
Let us point out that, due to the assumptions on $u_{\rm in}$, we may increase $t_0$ and then shift $\overline u$ by a large constant so that
\be
    \breve u_{\rm in}
        = u_{\rm in}\left(\cdot - \tfrac12\log(t_0)\right)
        \leq \overline u(0, \cdot - C).
\ee
The largeness of $t_0$ is required so that $\bar u(0, \cdot)$ has slower Gaussian decay than $\breve u_{\rm in}$. 

By the comparison principle, \Cref{thm.CP_u}, which applies as $\overline{u}$ is a supersolution with nonnegative shape defect, we conclude that 
\be
    \breve u(t,z) \leq \overline u(t,z - C)
\ee
for all $t>0$ and $z\in \R$. 
As a result, we immediately see that there is $C>0$ such that, for all $z\geq C$
\be
    \breve u(t,z) \leq \bar u(t,z)
    < 1.
\ee
This concludes the proof.

\end{proof}

Let us now briefly comment on why an analogous construction fails in the nonlocal case. The main issue stems from the fact that one expects $\rho \approx 1$ for $x \leq \overline x(t)$, whence $P \approx 1+ (\overline x(t) - x)$.   One would then naturally alter the supersolution above by concatenating $R$ with a linear function (instead of the constant function $\equiv 1$).  In order to apply a comparison principle, one requires a concave cusp, meaning that the slope of the linear function must be smaller than 
\be
    R_z(t,\overline  z_{\rm super}(t))
    =-1 -\frac{\overline  z_{\rm super}(t)}{4(t+t_0)}+\frac{\overline  z_{\rm super}(t)}{(t+t_0)^\frac{3}{2}}.
\ee
In other words, any linear supsersolution function of the form 
\be
    f(t,z)=-(1+\theta(t))(z-\overline  z_{\rm super}(t))+1
\ee
must satisfy
\be\label{e.c112102}
    \theta(t) \leq \frac{\overline  z_{\rm super}(t)}{4 (t+t_0)}.
\ee
to be a supersolution.  We now show that this condition is incompatible with the supersolution property. A quick computation yields
\be
    \dot{\overline{z}}_{\rm super}(t)
    =O\left(\frac{1}{(t+t_0)^\frac{3}{2}} \right),
\ee
from which we can deduce that
\be
    f_t - f_{zz} -\left(1-\frac{1}{2(t+t_0)}\right) f_z -1
    =-\dot{\theta}(z-\dot{\overline{z}}_{\rm super})+\theta -\frac{1}{2(t+t_0)}+o \left( \frac{1}{t+t_0}\right).
\ee
From this expression it becomes clear that $\dot{\theta}$ must be positive since $-(z-\dot{\overline{z}}_{\rm super})\geq 0$ and can be arbitrarily large as $z\to -\infty$.  Additionally, we see that it must be that
\be
    \theta(t) \geq \frac{1}{2(t+t_0)}.
\ee
In other words, $\theta(t)$ is positive and increasing, which contradicts the constraint \eqref{e.c112102} for large $t$.  This violates the inequality~\eqref{e.c112102} for large times.  Hence, such an approach to constructing a supersolution cannot succeed.

\subsubsection{Front Asymptotics in the Nonlocal Case.}
\label{s.pp_nonlocal_upper}

We break the proof into two steps.  First, we state an upper bound on the shape defect function in the moving frame:
\be
    \breve W(t,z)
    = W\left(t, z + 2t - \tfrac12 \log(t+t_0)\right),
\ee
where $t_0\geq1$ is chosen below.
Then, we show how to deduce the bound on $\breve P$, defined similarly as $\breve W$, from the bound on $\breve W$.

We state our main estimate first.
\begin{lemma}\label{l.pp_W_upper}
    Under the assumptions of \Cref{p.pp_upper}, there is a constant $C>0$ such that, for $t,z\geq 1$,
    \be
        \breve W(t,z)
        \leq \frac{C z}{t} e^{-z - \frac{z^2}{8t}}.
    \ee
\end{lemma}

Actually, we prove a slightly more precise estimate, but the extra precision is not needed to deduce \Cref{p.pp_upper}.  Additionally, the estimate we obtain provides a rate of convergence $\breve P \to P_{\rm tw}$; see \Cref{sss.convergence_pp}.  Let us now show how to apply \Cref{l.pp_W_upper} in order to deduce the upper bound on $\overline x(t)$.  We prove \Cref{l.pp_W_upper} in \Cref{ss.l.pp_W_upper}.

\begin{proof}[Proof of \Cref{p.pp_upper} in the nonlocal case]
Fix any $t>1$.  We may assume that
\be\label{e.bar_x_assumption}
    \overline x(t) \geq 2t - \frac12 \log(t+t_0) + 1,
\ee
since otherwise there is nothing to prove.

Next, observe that, for $z \geq \overline x(t)-2t+\frac{1}{2}\log(t+t_0)$,
\be\label{e.c112101}
	-(e^z\breve{P}(t,z))_z = e^z\breve{W}(t,z).
\ee
This follows from~\eqref{e.eta_and_Q_expressions} and the definition~\eqref{d.w} of $W$.  
Using~\eqref{e.c112101} and \Cref{l.pp_W_upper}, we see that, for $z \geq \overline x(t)-2t+\frac{1}{2}\log(t+t_0)$,
\be\label{e.c011201}
	\begin{split}
	e^z \breve{P}(t,z)
    &= - \int_z^\infty \left( e^\zeta \breve P(t,\zeta)\right)_\zeta d\zeta
    =  \int_z^\infty e^\zeta \breve{W}(t,\zeta) d\zeta
	\leq 
        \int_z^\infty \frac{C\zeta}{t} e^{- \frac{\zeta^2}{8t}} d\zeta
    \leq C.
	\end{split}
\ee
The first equality requires that the boundary term $(e^z \breve P)|_{z=+\infty}$ vanishes.  This is where we use the assumption that $P_{\rm in}$ has bounded exponential moment.  Indeed, by \Cref{l.expmoment}, $\breve P(t)$ has an exponential moment as well for all later time $t\geq0$. Recalling~\eqref{e.c112101}, which implies that $e^z \breve P$ is decreasing in $z$.  Hence, $e^z \breve P$ may only have a bounded integral if the boundary term $(e^z \breve P)|_{z=+\infty}$ in the integration by parts in the first equality of~\eqref{e.c011201} vanishes. 

Evaluating at $z=\overline x(t) - \left(2t - \frac{1}{2} \log(t+t_0)\right)$ yields
\be
    e^{\overline x(t) - \left(2t - \frac{1}{2} \log(t+t_0)\right)}
    = e^{\overline x(t) - \left(2t - \frac{1}{2} \log(t+t_0)\right)}
        \breve P\left(t, \overline x(t) - \left(2t - \frac{1}{2} \log(t+t_0)\right)\right)
    \leq C.
\ee
After taking the logarithm of both sides and adjusting the constant $C$, we find the desired upper bound \eqref{e.26021201}.
\end{proof}

\subsubsection{The bound on the shape defect function in the nonlocal case: proving \Cref{l.pp_W_upper}}\label{ss.l.pp_W_upper}

\begin{proof}[Proof of \Cref{l.pp_W_upper}]
We change to the moving frame:
\be
    \breve W(t,z)
        = W\left(t, z + 2t - \frac{1}{2} \log(t+t_0)\right).
\ee
We denote $\tilde P$ similarly. From~\eqref{e.w_eqn} (cf.~\eqref{e.tildeWchipush}), we find
\be 
\label{e.Wmv12}
    \breve{W}_t
    =
    \breve{W}_{zz}
    + \left( 2 -\frac{1}{2(t+t_0)}\right) \breve{W}_{z}
    -\alpha(\breve{P}) \breve{W}_z +(1-\alpha(\breve{P}))\breve{W}.
\ee
Moving all terms to the left-hand side, we may write this as
\be
    \cK \breve W = 0,
\ee
where we have defined
\be\label{def.K}
    \cK  := \cL_{\sfrac12} + \alpha(\breve P) \left( \partial_z + 1\right).
\ee
We recall the definition~\eqref{def.Lbeta} of $\cL_{\sfrac12}$. 
Lastly, before beginning our proof let us note that, due to \Cref{p.pp_lower} there is a constant $C$, depending on the initial data, such that, for all $t\geq 0$,
\be
    \overline x(t)
    \geq
    2t - \frac{1}{2} \log(t+1) - C
    = 2t - \frac{1}{2} \log(t + t_0) 
        + \left(\log\Big(\frac{t+t_0}{t+1}\Big) - C\right).
\ee
Hence, up to a translation of the initial data that is independent of $t_0$ (as long as $t_0\geq 1$), we may assume that
\be
\label{e.26021207}
    \overline x(t)
    \geq 2t - \frac{1}{2} \log(t+t_0) + 4.
\ee
As a consequence, for all $z \leq 4,t\geq 0$,
\be\label{e.tilde_P_large}
    \breve P(t,z) \geq 1
    \qquad\text{ and, hence, }
    \alpha(\breve P(t,z)) = 1.
\ee
This plays an important role in the sequel.

We proceed as follows: (1) we construct a supersolution $R$ to~\eqref{e.Wmv12} on the domain $z>0$; (2) we construct a supersolution $L$ to \eqref{e.Wmv12} when $z<4$; (3) for $z\in [1,2]$, we prove the ordering $R(t,z)\geq L(t,z)$; (4) for  $z\geq 3$, we prove the opposite ordering $R(t,z)\leq L(t,z)$; (5) define
\be 
\label{def.Wsuper}
    \overline{W}(t,z)
    := \left\{ \begin{array}{ll}
        L(t,z) & z \leq 1\\
        \min\{L(t,z),R(t,z)\} & z > 1,
        \end{array}
    \right.
\ee 
which by the steps above then defines a supersolution to~\eqref{e.Wmv12} on the entire domain.

\smallskip
\noindent
{\bf \# Step (1):} Let $E$, $G$, and $H$ be as in the local case, see~\eqref{e.pp_supersolution}, and define
\be
    R(t,z)
    :=
     \beta \frac{z}{t+t_0} E(z) G(t,z) H(t,z).
\ee
From this and straightforward expansions, we find that for $z>0$
\be\begin{split}
    \frac{\mathcal{K} R}{R}
    & = -\frac{1}{t+t_0} -\frac{2}{z}\left( \frac{E_z}{E}+\frac{G_z}{G}+\frac{H_z}{H}  \right)-\frac{1}{z}\left(2-\frac{1}{2(t+t_0)}\right)
    \\&\qquad
    +\frac{\cL_{\sfrac12}(EGH) }{EGH}+ \alpha(\breve{P}) \left( \frac{1}{z}+\frac{E_z}{E}+\frac{G_z}{G}+\frac{H_z}{H}+1 \right).
\end{split}\ee
Recall the values computed in~\eqref{e.EGH}, so that
\be
    - \frac{2}{z} \left(\frac{E_z}{E}+\frac{G_z}{G}+\frac{H_z}{H}  \right)
    = - \frac{2}{z} \left( -1 - \frac{z}{2(t+t_0)} + \frac{2z}{(t+t_0)^{\sfrac32}}\right)
    = \frac{2}{z} + \frac{1}{t+t_0} - \frac{4}{(t+t_0)^{\sfrac32}}.
\ee
Then, using~\eqref{e.EGH} and the lower bound for $\frac{\cL_{\sfrac12}(EGH) }{EGH}$ given by~\eqref{e.LcalFoverF} and rearranging the terms, we obtain, for $z >0$,
\be
\label{e.26021205}
\begin{split}
    \frac{\mathcal{K}R}{R}
    &\geq \frac{z^2}{4(t+t_0)^\frac{5}{2}}
    -
    \frac{z}{4(t+t_0)^2}
    +
    \frac{\tilde K - 4}{(t+t_0)^{\sfrac32}}
    +
    \frac{z}{(t+t_0)^{\frac{5}{2}}}
    \\&\qquad
    +
    \frac{1}{2z(t+t_0)}
    +
    \alpha(\breve{P}) \left( \frac{1}{z}-\frac{z}{2(t+t_0)}+\frac{2z}{(t+t_0)^\frac{3}{2}} \right).
\end{split}\ee
As before, the sum of the first four terms is nonnegative as long as $\tilde K \geq 4+\frac{1}{4}$.  Additionally, the fifth term is positive.  Thus,
\be
\label{e.26021208}
    \frac{\mathcal{K}R}{R}
    \geq 
    \alpha(\breve{P}) \left( \frac{1}{z}-\frac{z}{2(t+t_0)}+\frac{2z}{(t+t_0)^\frac{3}{2}} \right).
\ee
But, recall from \Cref{c.easyupperbound} that $\overline x(t) \leq 2t + C$. Let us note that this constant $C$ depends on the shift of the initial datum that we carried out in \eqref{e.26021207}, but this latter shift is independent of $t_0$. Because of this upper bound, we have in the shifted frame that
\be\label{e.c112103}
    \alpha(\breve P(t,z)) = 0
    \quad\text{ whenever } z \geq \frac{1}{2}\log(t+t_0) + C,
\ee
and thus, by \eqref{e.26021208}, $\nicefrac{\cK R}{R}\geq 0$, when $z \geq \frac{1}{2}\log(t+t_0) + C$. Let us now assume that $0< z<\frac{1}{2}\log(t+t_0) + C$. Using \eqref{e.26021208}, we have
\be
 \frac{\mathcal{K}R}{R}
    \geq \frac{\alpha(\breve{P})}{z} \left( 1-\frac{z^2}{2(t+t_0)}\right) \geq \frac{\alpha(\breve{P})}{z} \left( 1-\frac{\left(\frac{1}{2}\log(t+t_0) +C\right)^2}{2(t+t_0)}\right)
    \geq 0,
\ee
for $t_0$ large enough, depending on $C$. By the above observation $C$ is independent of $t_0$ and thus we may choose $t_0$ as a function of $C$, \textit{i.e.} as a function of the initial datum.

Thus, we have proved that $R$ is a supersolution when $z>0$, which completes step (1).

\smallskip
\noindent
{\bf \# Step (2):} Define
\be 
    L(t,z)
    :=
    \frac{\nicefrac{\beta}{\gamma} }{t+t_0}E(z) e^{\kappa z}= \frac{\nicefrac{\beta}{\gamma} }{t+t_0}e^{-(1-\kappa)z},
\ee 
for any fixed $\kappa \in (0,1)$. A direct computation leads to
\be
    \frac{\mathcal{K}L}{L}
    =
     \kappa\alpha(\breve{P})-\kappa^2-\frac{3-\kappa}{2(t+t_0)}.
\ee 
Recall, from~\eqref{e.tilde_P_large}, that $\alpha(\breve P) = 1$ for all $z\leq 4$.  Hence, 
for $t_0$ large enough, we obtain, for every $t\geq 0$ and $z< 4$,
\be
    \frac{\mathcal{K}L}{L}
    = \kappa(1-\kappa)-\frac{3-\kappa}{2(t+t_0)}
    > 0.
\ee
Hence $L$ is a supersolution for $z<4$.
Here, $t_0$ inherits, in addition to its dependence on the initial datum (from Step (1)), a dependence on $\kappa$.

\smallskip
\noindent
{\bf \# Step (3):}
We observe that, for $z \in (0,4)$,
\be\label{E.R_over_L}
    \frac{R(t,z)}{L(t,z)}
    = \gamma ze^{-\kappa{z}}G(t,z)H(t,z)
    =\frac{\gamma}{\kappa} g(\kappa z) e^{-\frac{z^2}{4(t+t_0)}}e^{\frac{1}{\sqrt{t+t_0}}\left( \frac{z^2}{t+t_0} - K \right)},
\ee 
where we define
\be 
g(z):=ze^{-z}.
\ee
We observe that $g$ increases on $(0,1)$ and then decreases on $(1,\infty)$; its maximum is achieved at $z=1$ with value $g(1)=e^{-1}$. Now, let us select 
\be 
\kappa \in \left( \nicefrac{\log(3)}{2}, 1 \right).
\ee
Because of this choice, a quick computation shows that $g(\kappa)>g(3\kappa)$. Moreover, since $2\kappa>1$, we also have that $g(2\kappa)>g(3\kappa)$. We summarize both facts in the bound
\be 
\label{e.26021203}
\min\{ g(\kappa),g(2\kappa)\} > g(3\kappa).
\ee
Now, restricting \eqref{E.R_over_L} to $z\in [1,2]$, we find the lower bound
\be
\label{e.26021204}
    \frac{R(t,z)}{L(t,z)}
    \geq
     \frac{\gamma}{\kappa} \min\{ g(\kappa),g(2\kappa)\}  
    e^{-\frac{1}{t+t_0}}e^{\frac{1}{\sqrt{t+t_0}}\left( \frac{1}{t+t_0} - K \right)} \geq \frac{\gamma}{\kappa} \min\{ g(\kappa),g(2\kappa)\}  e^{-\frac{1}{t_0}-\frac{K}{\sqrt{t_0}}}.
\ee
Let us pick $t_0$ large enough, depending on $\kappa,K$, so that
\be 
\min\{ g(\kappa),g(2\kappa)\}  e^{-\frac{1}{t_0}-\frac{K}{\sqrt{t_0}}} >g(3\kappa),
\ee
which is possible, because of the bound \eqref{e.26021203}. Finally, because of this latter choice of $t_0$, we may choose $\gamma$, depending on $\kappa,K,t_0$ such that
\be 
\label{choice.gamma}
\frac{\gamma}{\kappa} g(3\kappa) < 1 < \frac{\gamma}{\kappa} \min\{ g(\kappa),g(2\kappa)\}  e^{-\frac{1}{t_0}-\frac{K}{\sqrt{t_0}}}.
\ee
As a consequence of this choice and \eqref{e.26021204}, we obtain for $z\in [1,2]$,
\be 
\frac{R(t,z)}{L(t,z)}>1.
\ee
Let us emphasize here that $\kappa$ is selected universally and $K$ is chosen depending only on the initial condition (see \eqref{e.26021205}). Then $t_0$ is selected as a function of $\kappa,K$ and $\gamma$ as a function of $\kappa,K,t_0$. 
This shows that there exist constants $\kappa,K,t_0,\gamma$ such that \eqref{choice.gamma} holds.

\smallskip
\noindent
{\bf \# Step (4):} We now show that $\nicefrac{R(t,z)}{L(t,z)}<1$ for $z>3$, which is a mere consequence of the choices above. Assuming additionally that $t_0\geq 16$, we get
\be 
 e^{-\frac{z^2}{4(t+t_0)}}e^{\frac{1}{\sqrt{t+t_0}}\left( \frac{z^2}{t+t_0} - K \right)} \leq  e^{-\frac{z^2}{4(t+t_0)^{\sfrac32}}\left( \sqrt{t+t_0} - 4 \right)}\leq 1.
\ee
Therefore, using \eqref{E.R_over_L} and \eqref{choice.gamma}, we have for $z\geq 3$
\be 
\frac{R(t,z)}{L(t,z)}
    \leq
     \frac{\gamma}{\kappa} g(3\kappa) <1.
\ee

\smallskip
\noindent
{\bf \# Step (5):} The previous four steps show that $\overline{W}$ is a supersolution to~\eqref{e.Wmv12}. Notice that, by hypothesis, $W_{\rm in}\in L^\infty(\R)$ and $W_{\rm in}(z)\leq e^{-\gamma z^2}$ for some $\gamma>0$ and all $z$ sufficiently large. It follows that, after choosing $\beta$ sufficiently large and, if necessary, further increasing $t_0$, we have $W_{\rm in} \leq \overline W(0,\cdot).$ 
From the comparison principle, it follows that $\breve {W} \leq \overline W$, which completes the proof.
\end{proof}

\subsubsection{Discussion regarding convergence rates to the traveling wave in the nonlocal case}\label{sss.convergence_pp}

An interesting consequence of our above bounds on $\breve W$ is an estimate on the rate of convergence of $\breve P$ to the traveling wave $P_{\rm in}$.  We discuss this briefly here.

Let us observe that the proof of \Cref{l.pp_W_upper} yields
\be
    0 \leq \breve W \lesssim \frac{1}{t},
\ee
with appropriate decay as $z\to\infty$.
From the work of \cite{AnHendersonRyzhik_convergence}, this shows that, with the correct shift,
\be
    \breve P = P_{\rm tw} + O\left( \frac{1}{t}\right).
\ee
By analogy with the Fisher-KPP equation~\cite{AnHendersonRyzhik_convergence,Graham}, this is the expected rate of convergence.

As we see in \Cref{s.pulled}, it is not necessary to use the shape defect function in order to understand the front location in the pulled case.  Hence, we do not obtain a sharp estimate in the correct frame on the decay of $\breve W$, and, thus, we do not obtain a convergence rate as above.  It seems likely, however, that the approach in \Cref{sss.shape-defect-nonlocal} could be extended to the pulled case to obtain a (presumably sharp) $\sfrac1t$ rate of convergence to the traveling wave.

In the pushed case, the only bound we have on the shape defect function in the correct frame is due to \Cref{thm.W.decay}: $\hat W \lesssim \sfrac{1}{\sqrt t}$.  This would yield a convergence rate like $\sfrac1{\sqrt{t}}$.  This, however, is likely not sharp.  The rate is expected to be exponential; see analogous results~\cite{AnHendersonRyzhik_convergence, rothe1981}.

It would be interesting in a future work to obtain convergence rates in all cases (pulled, pushmi-pullyu, pushed) for both models.  Let us note that a useful approach might follow the work of Patterson~\cite{Patterson}, who has developed a strategy to obtain sharp convergence rates based on weighted functional inequalities.

\color{black}

\section{The pulled case: $\chi \in [0,1)$} \label{s.pulled}

We now state precisely the main result in the pulled regime $\chi \in [0,1)$, which was outlined in \Cref{ss.main_results}.

\begin{theorem}
    	Let $\chi \in [0,1)$.  Suppose that  $P$ (resp. $u$) solves~\eqref{pde.P} (resp.~\eqref{pde.u}) with nontrivial initial data whose shape defect function $W_{\rm in}$ (resp.~$w_{\rm in}$) is nonnegative. Suppose further that, for some $\gamma>0$, we have:
        \begin{itemize}
            \item (nonlocal case) for $x\geq 0$, $P_{\rm in}(x)\leq \gamma^{-1} e^{- \gamma x^2}$  and  $\|\rho_{\rm in}\|_{L^\infty(\R)}<1$;
            \item (local case) for all $x$, $0 \leq u_{\rm in}(x) \leq \min\{1, \gamma^{-1} e^{-\gamma x_+^2}\}$.
        \end{itemize}
        Then, there exists $C$ such that, for $t\geq 0$,
	\be
		2t - \frac{3}{2} \log t - C\leq \overline{x}(t) \leq 2t - \frac{3}{2} \log t + C.
	\ee
\end{theorem}

Let us briefly discuss the assumptions.  First, we note that the Heaviside function satisfies the assumptions in the local case. In the nonlocal case, taking $\rho_{\rm in}$  to be the Heaviside function multiplied by a constant strictly smaller than $1$ satisfies the assumptions.

Second, the assumption that $\rho_{\rm in} \leq 1-\varepsilon$ for some $\varepsilon>0$ is natural, since the traveling wave profile is bounded above by $(2-\chi)^{-1}$; recall \Cref{table.tw}. Our reason for imposing this condition is related to the behavior of $P$ as $x \to -\infty$: the proof requires that this asymptote be flatter than slope $-1$ (see the assumptions of \Cref{l.upper_pulled}). We believe that this restriction is purely technical and could be removed at the expense of a more delicate argument.

Third, as discussed below \Cref{thm.asymptotics_pushmipullyu}, the boundedness of $W_{\rm in}$ is equivalent to the boundedness of $\rho_{\rm in}$.  Moreover, the Gaussian decay assumption on $W_{\rm in}$ is weaker than that on $\rho_{\rm in}$.

Finally, the upper and lower bounds can be established under weaker assumptions; see \Cref{l.lower_pulled} and \Cref{l.upper_pulled}, whose proofs are given in \Cref{ss.lower_pulled} and \Cref{ss.upper_pulled}, respectively.

Before turning to the proofs, let us briefly comment on the strategy. 
The lower bound follows via comparison with the Fisher-KPP equation and is therefore obtained essentially ``for free.''  The upper bound, however, requires more work.
In contrast to the case $\chi=1$ in the nonlocal model, where the upper bound was obtained by constructing a supersolution for the shape defect function $W$, in the pulled regime $\chi\in [0,1)$, we work directly at the level of $P$.

We did investigate whether the argument used in the pushmi–pullyu regime could be adapted to the present setting. However, in the pulled case, the natural candidate supersolution for $W$ fails to be uniformly bounded in $L^1$ on intervals of the form $[a,+\infty)$, which prevents us from extracting the desired upper bound on $\overline{x}(t)$ from the supersolution for $W$. For this reason, we instead construct a supersolution directly for $P$. This comes at the cost of the additional condition $\|\rho_{\rm in}\|_\infty<1$ mentioned just above.

\subsection{Lower bound in the pulled case: $\chi \in [0,1)$}
\label{ss.lower_pulled}

Our goal in this section is to prove the following lower bound.

\begin{proposition}[Lower bound]\label{l.lower_pulled}
	Let $\chi \in [0,1)$.  Suppose that $P$ (resp. $u$) solves~\eqref{pde.P} (resp.~\eqref{pde.u}) with nontrivial initial data $P_{\rm in}$ (resp. $u_{\rm in}$) whose shape defect function $W_{\rm in}$ (resp.~$w_{\rm in}$) is nonnegative.  
    Then there is $C$ such that, for $t\geq 0$,
	\be
		\overline{x}(t) \geq 2t - \frac{3}{2} \log t - C.
	\ee
\end{proposition}

\begin{proof}
We prove \Cref{l.lower_pulled} in the nonlocal case; the local case follows \textit{mutatis mutandis} and is therefore omitted. 
Recall that
\be
    0 \leq W=-P_x-\eta(P).
\ee
From \Cref{l.l.eta_nonloc}, we have
\be
    \eta(P) \geq \frac{1}{2-\chi} P
        \quad\text{ for } P \in (0,1).
\ee 
Consequently, for any $t\geq 0$ and any $x > \overline{x}(t)$,
\be\label{e.c112501}
    P(t,x)
    \leq e^{-\frac{1}{2-\chi}(x-\overline{x}(t))}.
\ee 
For any $t> 0$, let $\overline x_{\sfrac12}(t)$ be the unique value such that
\be
    P(t,\overline x_{\sfrac12}(t)) = \frac{1}{2}.
\ee
From~\eqref{e.c112501}, it follows that
\be\label{e.front_ordering}
    \overline x_{\sfrac12}(t)
    \leq
    \overline x(t) + (2- \chi) \log(2).
\ee
It therefore suffices to obtain a lower bound on $\overline x_{\sfrac12}$.

Now, consider the solution to the FKPP equation 
\be\label{e.fkpp}
    v_t=v_{xx}+v(1-v),
\ee
with initial datum $v_{\rm in}=\min\left\{1, P_{\rm in}(\cdot +1) \right\}$.  Clearly $v_{\rm in} <  P_{\rm in}$. A standard comparison argument proves that $v\leq P$.  Indeed, the two functions cannot cross when $x\leq \overline{x}(t)$, since $v(t,x)< 1 \leq P(t,x)$ there. For $x > \overline{x}(t)$, $P$ satisfies the equation $P_t=P_{xx}+P$ and is therefore a supersolution to~\eqref{e.fkpp}. 

Let us define $\xi_{\sfrac12}(t)$ by 
\be
    v(t,\xi_{\sfrac12}(t))=\frac12.
\ee
Because $v(t)\leq P(t)$, we have that
\be\label{e.c011301}
    \xi_{\sfrac12}(t) \leq \overline x_{\sfrac12}(t).
\ee
It is well known that (see, e.g.,~\cite{bramson1978}) that, for some $C$ and all $t>0$,
\be
    \xi_{\sfrac12}(t)
        \geq 2t- \frac{3}{2}\log(t)- C.
\ee 
In view of~\eqref{e.front_ordering} and~\eqref{e.c011301}, the proof is complete.
\end{proof}

\subsection{Upper bound in the pulled case: when $\chi \in [0,1)$}
\label{ss.upper_pulled}

Our goal in this section is to prove the following upper bound.

\begin{proposition}[Upper bound]\label{l.upper_pulled}
Let $\chi \in [0,1)$.  Suppose that  $P$ (resp. $u$) solves~\eqref{pde.P} (resp.~\eqref{pde.u}) with nontrivial initial data such that, for some $\gamma >0$,
\begin{itemize}
\item (nonlocal case) $P_{\rm in}\in L^\infty_{\rm loc}(\R)$, for $x\geq 0$, $P_{\rm in}(x)\leq \gamma^{-1} e^{- \gamma x_+^2}$, and
\be
            \limsup_{x\to -\infty} \frac{P_{\rm in}(x)}{|x|} < 1;
        \ee
\item (local case) $0 \leq u_{\rm in} \leq \min\{1, \gamma^{-1} e^{-\gamma x_+^2}\}$.
    \end{itemize}
Then there is $C$ such that, for $t\geq 0$,
	\be
    \label{bd.claimed}
		\overline{x}(t) \leq 2t - \frac{3}{2} \log t +  C.
	\ee
\end{proposition}

We prove the upper bound $\overline{x}(t)\leq 2t-\frac{3}{2}\log(t)+C$ for $\chi <1$ in both the local and the nonlocal cases via the construction of a supersolution. Our construction is similar to the proof of the upper bound for $\chi=1$ in the local case. We begin with the nonlocal case, as the construction of the supersolution is slightly more technical.  Then, we adapt the analoguous construction for the local model and show that it satisfies the necessary conditions to apply the comparison principle, \Cref{thm.CP_u}. In particular, we show that the shape defect associated with the supersolution is nonnegative.

\begin{proof}[Proof of \Cref{l.upper_pulled} in the nonlocal case]
For $t_0\geq 0$ to be determined, we pass to the moving frame
\be
    \breve P(t,z)
    =
    P\left(t, z + 2t - \frac{3}{2}\log(t+t_0)\right).
\ee
From~\eqref{pde.P}, we find
\be 
\label{e.Pmv32}
    \breve P_t
    = \breve P_{zz}
        +\left( 2 -\frac{3}{2(t+t_0)}\right)\breve P_z -\chi \alpha(\breve P) \breve P_z + \breve P - A(\breve P).
\ee
The proof is based in part on computations that we carried out in the proof for $\chi=1$ in the local case. We recall the functions in $E$, $G$, and $H$ defined in~\eqref{e.pp_supersolution}, and we define a candidate supersolution $R$ on the right by:
\be\label{d.RzEGH}
    R(t,z)
    := \beta z E(z) G(t,z) H(t,z).
\ee
Implicitly in here are positive parameters $\beta, t_0$ and $K$ to be chosen.  

\smallskip
\noindent
{\bf \# Step (1):} Under universal conditions on $t_0$ and $K$, we show that $R$, given by \eqref{d.RzEGH}, is a supersolution when $z>0$ to the parabolic operator
\be
    \cL_{\sfrac32}
    :=
    \partial_t - \partial_{zz} -\left(2-\frac{3}{2(t+t_0)}\right) \partial_z - 1,
\ee
which corresponds to the parabolic operator appearing in \eqref{e.Pmv32} in the region where $\tilde{P}<1$.  Recall the definition~\eqref{def.Lbeta} of $\cL_{\sfrac12}$.

First, observe that
\be
    \left[\cL_{\sfrac32}, z\right]
        = - 2 \partial_z  - \left(2 - \frac{3}{2(t+t_0)}\right)
    \quad\text{ and } \quad
    \mathcal{L}_{\sfrac32}
    =
    \cL_{\sfrac12}
    +\frac{1}{t+t_0}\dz.
\ee
Here the brackets denote the commutator $[A,B] = AB - BA$. We deduce that
\be
\begin{split}
    \frac{\cL_{\sfrac32}R}{R}
    &= \frac{\cL_{\sfrac32}(z EGH)}{zEGH}
    = \frac{z\cL_{\sfrac32} (EGH)}{zEGH}
        + \frac{[\cL_{\sfrac32},z] (EGH)}{zEGH}
    \\&
    = \frac{\cL_{\sfrac32} (EGH)}{EGH}
        - 2 \frac{(EGH)_z}{zEGH}
        - \frac{1}{z}\left(2 - \frac{3}{2(t+t_0)}\right)
    \\&
    = \frac{\cL_{\sfrac12} (EGH)}{EGH}
        + \left(\frac{1}{t+t_0}- \frac{2}{z}\right) \frac{(EGH)_z}{EGH}
        - \frac{1}{z}\left(2 - \frac{3}{2(t+t_0)}\right).
\end{split}
\ee
We have already computed $\cL_{\sfrac12}(EGH)$; see~\eqref{e.LcalFoverF}.  Additionally, via~\eqref{e.EGH}, we have
\be
    \frac{(EGH)_z}{EGH}
    =
    -1
    - \frac{z}{2(t+t_0)}
    + \frac{2z}{(t+t_0)^{\sfrac32}}.
\ee
Again, recalling the computations in~\eqref{e.LcalFoverF} and the choices made in \eqref{bd.K_t0} with $K = 4 + \nicefrac{1}{2}\tilde K$ for $\tilde K>1$ and $t_0>16^2$, we find
\be
\begin{split}
\frac{\cL_{\sfrac32}R}{R}
&\geq
    \frac{z^2}{4(t+t_0)^{\sfrac52}}
    - \frac{z}{4(t+t_0)^2}
    + \frac{\tilde K}{4(t+t_0)^{\sfrac32}}
    +\frac{z}{(t+t_0)^{\sfrac52}}
    \\&\qquad
    +
    \left(\frac{1}{t+t_0}-\frac{2}{z}\right)
    \left(
     -1
     - \frac{z}{2(t+t_0)}
     + \frac{2z}{(t+t_0)^{\sfrac32}}
    \right)
    - \frac{2}{z}
    + \frac{3}{2z (t+t_0)}
\\&
= 
\frac{z^2}{4(t+t_0)^{\sfrac52}}
    - \frac{z}{4(t+t_0)^2}
    + \frac{\tilde K}{4(t+t_0)^{\sfrac32}}
    +\frac{z}{(t+t_0)^{\sfrac52}}
    \\&\qquad
    - \frac{1}{t+t_0}
    - \frac{z}{2(t+t_0)^2}
    + \frac{2z}{(t+t_0)^{\sfrac52}}
    +\frac{2}{z}
    + \frac{1}{t+t_0}
    - \frac{4}{(t+t_0)^{\sfrac32}}
    - \frac{2}{z}
    + \frac{3}{2z (t+t_0)}.
\end{split}
\ee
In the equality, we simply expanded the terms in parentheses.  
Combining like terms yields 
\be
\begin{split}
\frac{\cL_{\sfrac32}R}{R}
&\geq
    \frac{z^2}{4(t+t_0)^{\sfrac52}}
    - \frac{3z}{4(t+t_0)^2}
    + \frac{\tilde K-16}{4(t+t_0)^{\sfrac32}}
    +\frac{3z}{(t+t_0)^{\sfrac52}}
    + \frac{3}{2z (t+t_0)}.
\end{split}
\ee
As long as $\tilde K\geq 16$, the only negative term above is the second one.  Notice that, due to Young's inequality,
\be
    \frac{3z}{4(t+t_0)^2}
    = \frac{1}{4(t+t_0)^{\sfrac32}}
    \left(
    \frac{3z}{(t+t_0)^{\sfrac12}}
    \right)
    \leq
    \frac{1}{4(t+t_0)^{\sfrac32}}\left(
    \frac{z^2}{t+t_0}
        + \frac{9}{4}
    \right).
\ee
Thus, after choosing $\tilde K= 16 + \sfrac9{16}$, we see that
\be
\label{L32Rpos}
    \cL_{\sfrac32}R \geq 0,
\ee
as desired.

\smallskip
\noindent
{\bf \# Step (2):} Define 
\be\label{e.26020701}
    s := \max\left\{\frac{1}{2-\chi},\limsup_{x\to-\infty} \frac{P_{\rm in}(x)}{|x|}\right\} < 1, \quad \text{ and }\quad \varepsilon := \frac{1-s}{2}. 
\ee
We claim that we may choose parameters $\beta$ and $t_0$ in such a way that
$R(t, \cdot )=1$ has at least one solution and its largest solution, which we denote $\overline z_{\rm super}(t)$, is such that for all $t\geq 0$
\be 
\label{bd.zalpha4}
|R_z(t, \overline z_{\rm super}(t))+1|< \e.
\ee

In order to prove \eqref{bd.zalpha4}, define $I_\beta(z)=\beta z e^{-z}$ with $\beta >e$, so that $I_\beta(\cdot)=1$ has two positive roots and denote $z_\beta$ the biggest of root. Let us start by explaining the procedure to obtain \eqref{bd.zalpha4}: on the one hand, we pick $\beta\gg 1$ such that $I'_\beta(z_\beta)\approx -1$. On the other hand, we pick $t_0 \gg 1$ so that $R$ is close to $I_\beta$ and so that $\overline z_{\rm super}(t)\approx z_\beta$. Combining both choices, then leads to \eqref{bd.zalpha4}.

By an easy asymptotic expansion, we have that $z_\beta \sim \ln \beta$ and in particular, $z_\beta\to\infty$ as $\beta \to \infty$. Because of that we may pick $\beta\gg 1$ such that
\be 
\left|\frac{1}{z_\beta}\right|<\nicefrac{\e}{8}.
\ee
Then, through an easy computation, we find that $|I'_\beta(z_\beta)+1|\leq \nicefrac{\e }{8}$. 

Next, we may pick $t_0\gg 1$ depending on $\beta$ such that for every $t\geq0$:
\be 
\label{bd.zalpha1}
\left|\frac{1}{z_\beta}-\frac{z_\beta}{2(t+t_0)}+\frac{2z_\beta}{(t+t_0)^{\frac{3}{2}}}\right|<\nicefrac{\e}{4}.
\ee 
Under the condition that $t_0$ is large enough again, depending on $K$ and $\beta$, $G(t,z)H(t,z)$ is a small perturbation of $1$ uniformly on $[0,+\infty)\times [0,2z_\beta]$, in the sense that 
\be 
\label{e.26020703}
\lim_{t_0\to +\infty}\sup_{(t,z)\in [0,+\infty)\times [0,2z_\beta]}
    \left|\log(G(t,z)H(t,z) ) \right|,
    \left|(\log(G(t,z)H(t,z) ))_z \right|,
    \left|(\log(G(t,z)H(t,z) ))_t \right| 
        =0.
\ee
Using this control and the fact that $I'_\beta(z_\beta)\neq 0$, $R_z(t,z_\beta)$ is bounded away from $0$ uniformly in $t$. The implicit function theorem applied to $R$ yields, for $t_0$ large enough, a unique solution $\overline z_{\rm super}(t)$ near $z_\beta$ satisfying
\be
    \sup_{t\geq 0} |\overline z_{\rm super}(t)- z_\beta |
    \ll 1.
\ee
As a consequence, for $t_0$ large enough (depending again only on $K$ and $\beta$), we have for all $t\geq 0$,
\be 
\label{bd.zalpha2}
\left|\frac{1}{z_\beta}-\frac{1}{\overline z_{\rm super}(t)}\right|,
\left|\frac{z_\beta}{2(t+t_0)}-\frac{\overline z_{\rm super}(t)}{2(t+t_0)}\right|,\left|\frac{2z_\beta}{(t+t_0)^{\frac{3}{2}}}-\frac{2\overline z_{\rm super}(t)}{(t+t_0)^{\frac{3}{2}}}\right|<\nicefrac{\e}{4} .
\ee 
Finally, observe that
\be
 R_z(t,\overline z_{\rm super}(t))=\frac{ R_z(t,\overline z_{\rm super}(t))}{R(t,\overline z_{\rm super}(t))}=\frac{1}{\overline z_{\rm super}(t)}-1-\frac{\overline z_{\rm super}(t)}{2(t+t_0)}+\frac{2\overline z_{\rm super}(t)}{(t+t_0)^{\frac{3}{2}}}.
\label{bd.zalpha3}
\ee 
Passing the $-1$ term on the left, then taking the absolute value and swapping each of the three terms involving $\overline z_{\rm super}(t)$ with the analogous terms involving $z_\beta$ at the expense of $\nicefrac{\e}{4}$, using \eqref{bd.zalpha2}, leads to
\be 
\left|R_z(t,\overline z_{\rm super}(t))+1\right| \leq \nicefrac{3\e}{4}+ \left|\frac{1}{z_\beta}-\frac{z_\beta}{2(t+t_0)}+\frac{2z_\beta}{(t+t_0)^{\frac{3}{2}}}\right|.
\ee
Combined with \eqref{bd.zalpha1}, we have proved the claimed bound \eqref{bd.zalpha4}.

\smallskip
\noindent
{\bf \# Step (3):} We now define $\overline{P}$ as the concatenation of a linear function and $R$, show that $\overline{P}$ is a supersolution, and conclude the claimed upper bound \eqref{bd.claimed}.

Let us define our candidate supersolution on $\R$,
\be
    \overline P(t,z)
    :=
    \begin{cases}
        R(t,z)
            \quad &\text{ if } z \geq \overline z_{\rm super}(t),\\
        -\frac{1+s}{2}(z - \overline z_{\rm super}(t)) + 1
            \quad &\text{ if } z \leq \overline z_{\rm super}(t),
    \end{cases}
\ee
where $s$ is given by \eqref{e.26020701}. Let us immediately observe that by construction $\overline P$ has a concave cusp at $z = \overline z_{\rm super}(t)$. Indeed,
\be 
\label{e.26020702}
 \overline P_z(t,\overline z_{\rm super}(t)^+) =R_z(t,\overline z_{\rm super}(t))< -1+\varepsilon  = -\frac{1+s}{2} =\overline P_z(t,\overline z_{\rm super}(t)^-), 
\ee
where we used \eqref{bd.zalpha4} and the definition of $\e$ from \eqref{e.26020701}. We may also observe that up to translating the initial data,~\eqref{e.26020701} and the Gaussian decay of $P_{\rm in}$ on the right guarantee that $\overline{P}(0,\cdot) \geq P_{\rm in}$. Finally, let us also notice that $\dot{\overline z}_{\rm super}$ can be made as small as we desire, again after increasing $t_0$: indeed $\dot{\overline z}_{\rm super} = - \nicefrac{R_t}{R_z}|_{z=\overline z_{\rm super}}$, but as noted above $R_z|_{z=\overline z_{\rm super}}$ is uniformly bounded away from $0$ and $\nicefrac{R_t}{R}= (\log(GH))_t$ can be controlled through
\eqref{e.26020703} uniformly in $t$.

Now, with these observations, let us show that $\overline P$ is a supersolution. 
First, when $z < \overline z_{\rm super}(t),$
\be
\begin{split}
    \overline P_t + &\chi A^P(\overline P)_z - \left(2 - \frac{3}{2(t+t_0)}\right) \overline P_z- \overline P_{zz} - \overline P + A^P(\overline P)
    \\
    &
    = \frac{1+s}{2} \dot{\overline z}_{\rm super}(t)
        + \left(2 -  \chi - \frac{3}{2(t+t_0)}\right) \frac{1+s}{2} 
        - 1
    \\&
    = 
        \frac{2(1+s) - 1 - \chi(1+s)}{2} 
        +\frac{1+s}{2} \dot{\overline z}_{\rm super}(t) - \frac{3(1+s)}{4(t+t_0)}
\end{split}
\ee
Recalling~\eqref{e.26020701}, we see that the first term is a positive constant.  Hence, after possibly increasing $t_0$, the second two terms can be made small enough to be absorbed into the first.  We deduce that $\overline P$ is  a supersolution when $z<\overline{z}_{\rm super}(t)$. By construction of $R$, $\overline P$ is  also a supersolution when $z>\overline{z}_{\rm super}(t)$. Furthermore, similarly to the proof of the upper bound in the local model for $\chi=1$, we may view 
\be 
\mathcal{P}(\overline{P}):=\overline P_t  + \chi(A^P(\overline P))_z-\left( 2 -\frac{3}{2(t+t_0)}\right)\overline{P}_z
       -\overline P_{zz}
        - \overline P + A^P(\overline P)
\ee   
as a signed Radon measure. Since $\overline P$ is a supersolution for $z\neq\overline{z}_{\rm super}(t)$, the absolutely continuous part $\left[ \mathcal{P}(\overline{P})\right]_\text{ac}$ is nonnegative. Considering its singular part, unlike in the local case $A^P$ is Lipschitz continuous and the only source of a singular part comes from $-\overline{P}_{zz}$.  
Using \eqref{e.26020702}, we have that the singular part of $- P_{zz}$ is nonnegative.  
We deduce that
\be
    \overline P_t + \chi A^P(\overline P)_z - \left(2 - \frac{3}{2(t+t_0)}\right) \overline P_z- \overline P_{zz} - \overline P + A^P(\overline P)
        >0.
\ee
By using the comparison principle \Cref{thm.CP_P}, we conclude that $\breve{P}\leq \overline{P}$ and, thus,
\be
\overline{x}(t)-\left(2t - \frac{3}{2}\log(t+t_0)  \right) \leq  \overline{z}_{\rm super}(t) \leq C.
\ee
The proof is complete.
\end{proof}

\subsubsection{Adaptations for the local case}
\label{sss.upper_pulled_local}
For the local case, the argument applies almost verbatim.  The main difference is the behavior on the left. Let
\be
    \overline{u}(t,z)
    = \left\{
        \begin{array}{ll}
        R(t,z)
            & \text{ if }z\geq  \overline z_{\rm super}(t),  \\
        1
            & \text{ if }z< \overline z_{\rm super}(t),
        \end{array}
        \right.
\ee
where $R$ is given by \eqref{d.RzEGH}.  Observe that $\overline z_{\rm super}(t) > 0$.

We now need to show that the corner at $\overline{z}_{\rm super}(t)$ has the right behavior and that the shape defect function $\overline{w}=-\overline{u}_z-\eta^u(\overline u)$ is nonnegative. The former is a necessity to be a supersolution, while the latter is a technical condition in order for our comparison principle, \Cref{thm.CP_u}, to apply.

Up to changing the definition of $\e$ in \eqref{e.26020701} to
\be 
\e: = \frac{1-\chi}{2},
\ee
we again may choose $\beta$ in the definition in $R$, \eqref{d.RzEGH}, in such a way that for all $t\geq0$,
\be 
\left|R_z(t,\overline z_{\rm super}(t))+1\right| < \e,
\ee
as we did for \eqref{bd.zalpha4}. In particular, we have that
\be 
\label{bd.Rzchi}
R_z(t,\overline z_{\rm super}(t))
    < - \frac{1+\chi}{2}.
\ee
Now, let us proceed as above and view
\be 
\mathcal{P}(\overline{u}):=\overline u_t  + (\chi A^u(\overline u)-\overline u_z)_z-\left( 2 -\frac{3}{2(t+t_0)}\right)\overline{u}_z
        - \overline u + A^u(\overline u)
\ee   
as a signed Radon measure. By construction, the absolutely continuous part $\left[ \mathcal{P}(\overline{u})\right]_\text{ac}$ is nonnegative: indeed, for $z<\overline z_{\rm super}(t)$, $u\equiv 1$ is trivially a supersolution and for $z>\overline z_{\rm super}(t)$, $u\equiv R$ is a supersolution by \eqref{L32Rpos}. This time, similarly again to the proof of the upper bound in the $\chi=1$ case,
 the only source of a singular part comes from $(\chi A^u(\overline u)-\overline u_z)_z$.  To compute this, we require only knowing the jump of $\chi A^u(\overline u) - \overline u_z$ across the interface $z= \bar z_{\rm super}(t)$:
\be 
\begin{split}
\left[ \mathcal{P}(\overline{u})\right]_\text{s}
    &=
    \left[\left(\chi A^u(\overline u)(\bar z_{\rm super}(t)^+)-\overline u_z(\bar z_{\rm super}(t)^+)\right)
    - \left(\chi A^u(\overline u)(\bar z_{\rm super}(t)^-)-\overline u_z(\bar z_{\rm super}(t)^-)\right)\right]_{\overline{z}_{\rm super}(t)} \delta_{\overline{z}_{\rm super}(t)}
    \\&
    =\left(-R_z(t,\overline z_{\rm super}(t))-\chi )\right)\delta_{\overline{z}_{\rm super}(t)}
    > \frac{1-\chi}{2}\delta_{\overline  z_{\rm super}(t)}
    > 0,
\end{split}
\ee
where we used \eqref{bd.Rzchi}.
This shows that $\mathcal{P}(\overline{u})\geq 0$ as a distribution and, thus, that $\overline{u}$ is a supersolution.

It remains now to show that the shape defect function $\overline{w}=-\overline{u}_z-\eta^u(\overline{u})$ associated with the supersolution $\overline{u}$ is nonnegative. From \Cref{l.l.eta_loc}, when $\chi \in [0,1)$, we have the explicit expression of $\eta^u$. For $s\in [0,1)$,
\be 
\label{exp.eta_u}
    \eta^u(s)
    =
    \left( 1+ \frac{1}{W_{-1}(-\kappa s)}\right)s ,
\ee
where $\kappa:= \frac{1}{1-\chi}e^{-\frac{1}{1-\chi}}$ and $W_{-1}$ is the secondary branch of the Lambert $W$ function.

Clearly for $z< \overline  z_{\rm super}(t)$, we have $\overline u(t,z)=1$ and, thus, $\overline{w}(t,z)=0$. 
Hence, we only need to establish the nonnegative of $\bar w$ when $z> \overline  z_{\rm super}(t)$. 
We fix
\be 
\label{d.beta}
\beta > \kappa^{-1}e^{\frac{\sqrt{t_0}}{2(\sqrt{t_0}-4)}+\frac{K}{\sqrt{t_0}}}.
\ee
Recalling that we chose $\beta$ first and then $t_0$ depending on $\beta$, one might worry that this is not possible.  Notice, however, that the expression on the right hand side of~\eqref{d.beta} is bounded as $t_0 \to \infty$.  Hence, there is no issue.

Additionally, for any $t\geq 0$, let
\be 
\label{def.zast}
{\zeta}(t):
=\sqrt{\frac{2(t+t_0)^{3/2}}{\sqrt{t+t_0}-4}}.
\ee
Up to increasing $t_0$, the above is real-valued.

The proof is concluded if we show that, for all $z> \overline z_{\rm super}(t)$,
\be 
    -R_z(t,z)-\eta^u(R(t,z))\geq 0.
\ee
When $z \geq \zeta(t)$, we have
\be\label{e.c011302}
    -R_z-\eta^u(R)\geq R\left(-\frac{R_z}{R}-1 \right) =R(t,z)\left(-\frac{1}{z}+2\left(\frac{1}{4(t+t_0)}- \frac{1}{(t+t_0)^{3/2}}\right) z \right).
\ee
Here we used that $\eta^u(s) \leq s$, which follows from \eqref{eta.unif.bound_u} shown in \Cref{l.l.eta_loc}.  We then see that the right hand side is positive precisely when $z > \zeta(t)$; that is,
\be
    -R_z-\eta^u(R)
    > 0
    \qquad\text{ for all } z> \zeta(t).
\ee
As a result, the proof is complete whenever $\overline z_{\rm super}(t) \geq \zeta(t)$.

We now consider the case when $\overline z_{\rm super}(t) < \zeta(t)$.
Observe that 
\be 
\label{e.inf.beta}
\inf_{t\geq 0, z\in [\overline{z}_{\rm super}(t), {\zeta}(t)]}G(t,z)H(t,z) \geq \frac{1}{\kappa \beta}.
\ee
Indeed, by~\eqref{e.pp_supersolution}, $H(t,z)$ may be lower bounded by $\exp\{-\sfrac{K}{\sqrt{t_0}}\}$ and $G(t,z)$ is decreasing in $z$ (for $z>0$). Thus, for $z\in  [\overline  z_{\rm super}(t), {\zeta}(t)]$, 
\be 
G(t,z)\geq  e^{-\frac{{\zeta}(t)^2}{4(t+t_0)}}\geq e^{-\frac{\sqrt{t+t_0}}{2(\sqrt{t+t_0}-4)}}\geq e^{-\frac{\sqrt{t_0}}{2(\sqrt{t_0}-4)}}.
\ee
We hence deduce~\eqref{e.inf.beta} directly from~\eqref{d.beta}. As a consequence of that, for
$t\geq 0$ and $z\in [\overline z_{\rm super}(t), {\zeta}(t)]$, we have
\be 
\label{bd.kappaR}
\kappa R(t,z) \geq \kappa \beta z E(z)G(t,z)H(t,z) \geq ze^{-z}.
\ee
Next, we observe that
\be 
-R_z-\eta^u(R) =
R\left(1-\frac{1}{z}+2\left(\frac{1}{4(t+t_0)}- \frac{1}{(t+t_0)^{3/2}}\right)z -\left(1+ \frac{1}{W_{-1}(-\kappa R)}\right) \right).
\ee
We recall that $W_{-1}$ is decreasing on $[-e^{-1},0)$.  Combining this with \eqref{bd.kappaR}, we obtain for $z\in [\overline{z}_{\rm super}(t), {\zeta}(t)]$
\be 
W_{-1}(-\kappa R) \geq W_{-1}(-ze^{-z})=-z.
\ee
Carrying out the algebra and using $t_0>16$, we find
\be 
-R_z-\eta^u(R)\geq2zR\left(\frac{1}{4(t+t_0)}- \frac{1}{(t+t_0)^{3/2}}\right)\geq 0.
\ee
In the last inequality, we used that $z >0$ because $\overline z_{\rm super}(t)>0$. This completes the proof that $\overline w \geq 0$.  

Thus we may apply the comparison principle, \Cref{thm.CP_u}. Using the same argument as in Step (5) of the proof of the upper bound in the local case when $\chi=1$, from the ordering $\tilde{u}\leq \overline{u}$, we may deduce that
\be 
\overline{x}(t) \leq 2t - \frac{3}{2}\log(t+t_0)   +C,
\ee
which completes the proof.

\bibliographystyle{plain}
\bibliography{biblio.bib}

\end{document}